\documentclass[letterpaper,10pt]{article}
\usepackage{setspace}  

\usepackage[colorlinks=true,linkcolor=blue,citecolor=red]{hyperref} 

\usepackage{amsmath}

\usepackage{float}
\usepackage{amssymb}
\usepackage{amsfonts}
\usepackage{graphicx}
\usepackage{epstopdf}
\usepackage{color}
\usepackage{multirow}
\usepackage[title]{appendix}
\usepackage{arydshln} 

\graphicspath{{figures/}}
\usepackage{subcaption}
\usepackage{enumitem}

\usepackage{xcolor, colortbl}
\definecolor{Gray}{gray}{0.90}
\definecolor{LightCyan}{rgb}{0.65,1,1}
\definecolor{Cyan}{rgb}{0.95,1,1}
\definecolor{GRAY}{gray}{0.75}

\usepackage[T1]{fontenc}
\usepackage[utf8]{inputenc}
\usepackage[font=small,labelfont=bf]{caption}

\usepackage[numbers,sort&compress]{natbib}

\newcommand{\bsub}{\begin{subequations}}
\newcommand{\esub}{\end{subequations}$\!$}

\numberwithin{equation}{section}

\usepackage[numbers,sort&compress]{natbib}

\DeclareTextFontCommand{\textcyr}{\cyr}

\usepackage{tabularx}
\usepackage{dirtytalk}
\usepackage{amsthm}
\usepackage{wrapfig}
\usepackage{mathtools}
\usepackage{bbm}

\usepackage[format=plain,
            labelfont={it},
            textfont=it]{caption}

\numberwithin{equation}{section}
\newtheorem{theorem}{Theorem}[section]
\newtheorem{conj}[theorem]{Conjecture}
\newtheorem{lemma}[theorem]{Lemma}

\newtheorem{cor}[theorem]{Corollary}
\newtheorem{assumption}[theorem]{Assumption}
\newtheorem{remark}[theorem]{Remark}

\newcommand{\bbR}{\mathbb{R}}

\newcommand{\bbN}{\mathbb{N}}
\newcommand{\bbZ}{\mathbb{Z}}

\newcommand{\bbC}{\mathbb{C}}
\newcommand{\cO}{\mathcal{O}}
\newcommand{\cL}{\mathcal{L}}

\newcommand{\range}{\mathrm{range}}
\newcommand{\dm}{\mathbbm{d}}
\renewcommand{\Re}{\,\mathrm{Re}\,}
\renewcommand{\Im}{\,\mathrm{Im}\,}

\newcommand{\vu}{{\bf{u}}}

\newcommand{\vF}{{\bf{F}}}

\newcommand{\csch}{\, \mathrm{csch}}
\newcommand{\sech}{\, \mathrm{sech}}
\newcommand{\fint}{f_\text{int}}

\begin{document}
\begin{center}
{\fontsize{14}{14}\fontfamily{cmr}\fontseries{b}\selectfont{
Dynamical behavior of diffusively coupled scalar differential equations as Wentzell boundary conditions}}\\[0.2in]
Merlin Pelz$^{\dagger}$\\[0.1in]
\textit{\footnotesize
$^\dagger$School of Mathematics, University of Minnesota, Minneapolis, 206 Church St SE, MN 55455, USA \\ (mpelz@umn.edu)
}
\end{center}

\begin{abstract}
Two identical scalar dynamical systems coupled through a scalar diffusion equation are studied herein, with respect to bifurcations from a symmetric steady-state to symmetric and asymmetric steady-states and to in-phase and anti-phase oscillations. Numerical continuations based on the developed theory show the shape of the bifurcation branches and the attracting nonlinear states far from bifurcation onset and confirm their, for the most part, derived stability. This study extends the work on the dynamical properties of a single scalar dynamical system coupled to its own delay through an adjacent diffusion field and quantifies further the delay of diffusive information transmission between two such Wentzell boundaries. The mathematical system is motivated by modeling biological membranes with yet unknown effective local fluxes that are coupled to bulk diffusion.
\end{abstract}

\section{Introduction}\label{sec:intro}
Consider a simple diffusion equation for $u(t, x) \in \bbR$, $t\in \bbR$, in the bounded spatial interval $(0, L)$, $L > 0$, coupled to two dynamic, also called Wentzell, boundary conditions,
\begin{subequations} \label{eq:sys}
\begin{eqnarray}
    \partial_t u &=& \partial_{xx}u - \sigma^2 u, \quad x\in(0, L), \label{eq:diff} \\
    u(t,0) = u^-(t), && u(t,L) = u^+(t), \label{eq:bc} \\
    \frac{d}{dt} u^- = f(u^-,\partial_n u(t,0),\mu), && \frac{d}{dt} u^+ = f(u^+,\partial_n u(t,L),\mu).\label{eq:dynb}
\end{eqnarray}
\end{subequations}
where $\sigma^2 \geq 0$ is the degradation rate in the interior of the domain and $\partial_n$ denotes the outward normal derivative with outward normal vector $n$. All dynamical transitions of \eqref{eq:sys} when varying the parameter $\mu \in \bbR$ can be found by first searching for transitions to oscillatory states branching off a steady-state $u_*(x)$ with $u_*(0) = u_*^-$ and $u_*(L) = u_*^+$ that solves \eqref{eq:sys}. In this way, three key assumptions are needed to show the occurrence of such oscillatory transitions that are developed in the following paragraphs. 

\paragraph{Main result --- assumptions.}
Note that for $\sigma \geq 0$, i.e., with or without degradation, steady-states of \eqref{eq:sys} have the form 
\begin{subequations}
\begin{eqnarray} \label{eq:steadystates}
    u_*(x) = u_*^- \frac{\sinh(\sigma(L-x))}{\sinh(\sigma L)} + u_*^+ \frac{\sinh(\sigma x)}{\sinh(\sigma L)}
\end{eqnarray}
with some $u_*^- \in \bbR$ and $u_*^+ \in \bbR$, so that 
\begin{equation} \label{eq:steadystate_dn}
    \partial_n|_{x=0} u_* = \sigma u_*^- \coth(\sigma L) - \sigma u_*^+ \csch(\sigma L) \quad \text{and} \quad \partial_n|_{x=L} u_* = -\sigma u_*^- \csch(\sigma L) + \sigma u_*^+ \coth(\sigma L).
\end{equation}
\end{subequations}

\begin{remark}[Case of no bulk degradation, $\sigma = 0$] \label{rmk:sigma=0}
    You may have stumbled upon the case $\sigma = 0$ while reading \eqref{eq:steadystates} and \eqref{eq:steadystate_dn}. Here and in the following, as long as well-defined and attainable, those expressions for $\sigma = 0$ are defined in the limit sense for $\sigma \downarrow 0$. Of course then it holds that 
    \begin{equation*}
        u_*(x) = u_*^- \cdot \left(1 - \frac{x}{L}\right) + u_*^+ \cdot \frac{x}{L},
    \end{equation*}
    so that $\partial_n|_{x=0} u_*(x) = \frac{u_*^- - u_*^+}{L}$ and $\partial_n|_{x=L} u_*(x) = \frac{u_*^+ - u_*^-}{L}$ at $\sigma = 0$. It will be stated clearly when complications arise as $\sigma \downarrow 0$. 
\end{remark}
This leads to the first assumption.

\begin{assumption}[Existence of symmetric equilibrium at $\mu = \mu_*^i$ and $\mu = \mu_*^a$] \label{assu:equi}
    I assume that
    \begin{equation*}
        f(u_*^s, \sigma_* u_*^s \tanh(\sigma_*\frac{L}{2}), \mu) = 0
    \end{equation*}
    at some $\mu = \mu_*^i$ and at some $\mu = \mu_*^a$ whose purpose will become clear in the next steps, so that 
    \begin{equation*}
        (u^-, u, u^+) = \left(u_*^s, \quad u_*^s \frac{\sinh(\sigma(L-x))}{\sinh(\sigma L)} + u_*^s \frac{\sinh(\sigma x)}{\sinh(\sigma L)}, \quad u_*^s\right)
    \end{equation*}
    for some fixed $\sigma_* \geq 0$ and some $u_*^s\in \bbR$ is an equilibrium solution to \eqref{eq:sys}, which is a symmetric one.
\end{assumption}

\begin{remark}[Choice of base equilibrium]
    Assumption \ref{assu:equi} constrains the symmetry of the base equilibrium $u_*$ assumed to exist for this study. However, the base equilibrium could either be a symmetric equilibrium $u_*$ with $u_*(0) = u_*(L)$ or an antisymmetric equilibrium with $u_*(0) = - u_*(L)$, or it can even be a mixed equilibrium state consisting of a linear combination of both, so that simply $u_*(0) \neq u_*(L)$, created by an antisymmetric bifurcation branch emerging from a symmetric equilibrium branch or vice versa. From any such base equilibrium, one can search for solutions that again have two more potential symmetries branching off that base state at some $\mu$ similar to how it is done in this study with the symmetric base state. Later on, in Section \ref{sec:examples}, I will illustrate the theory with example kinetics that have a trivial solution branch $u_* \equiv 0$ that is both symmetric and antisymmetric. I will then choose this branch as the base branch instead of another nontrivial symmetric or asymmetric branch for most of the illustration.
\end{remark}

To find solutions branching off the equilibrium branch $u_*$ for $\mu \in \bbR$ and to derive stability properties of $u_*$ and the new solution branches, perturb $u_*$ in the way
\begin{equation*}
    u^- = u_*^- + V^-, \quad u = u_* + V, \quad u^+ = u_*^+ + V^+
\end{equation*}
with perturbations of the linearized system and their corresponding eigenvalue (complex growth rate) $\lambda \in \bbC$,
\begin{equation} \label{eq:linearpert}
    V^-(t) = e^{\lambda t} V_0^-, \quad V(t, x) = e^{\lambda t}V_0(x), \quad V^+(t) = e^{\lambda t} V_0^+
\end{equation}
where $V_0^-, V_0, V_0^+ \in \bbR$ are of small extent, i.e., $|V_0^-|\ll 1, |V_0|\ll 1, |V_0^+| \ll 1$. Plugged into \eqref{eq:sys} and linearized about $(V_0^-, V_0, V_0^+) = (0, 0, 0)$, the small perturbations therefore need to satisfy
\begin{subequations} \label{eq:pertsys}
\begin{eqnarray}
    \lambda V_0 = V_0'' - \sigma^2 V_0, \quad &x\in(0, L), \label{eq:pertdiff} \\
    V_0(0) = V_0^-,  V_0(L) = V_0^+,&\label{eq:pertbc}
\end{eqnarray}
\begin{equation}
\begin{array}{l}    
    \lambda \begin{pmatrix} V_0^- \\ V_0^+ \end{pmatrix} \\
    =
    \begin{pmatrix}
        \partial_1 f(u_*^-, \sigma_* (u_*^- \coth(\sigma_* L) - u_*^+\csch(\sigma_*L)), \mu) &0 \\
        0 &\partial_1 f(u_*^+, \sigma_* (-u_*^- \csch(\sigma_* L) + u_*^+\coth(\sigma_*L)), \mu)        
    \end{pmatrix}
    \begin{pmatrix}
        V_0^- \\ V_0^+
    \end{pmatrix} \\
    + \begin{pmatrix}
        \partial_2 f(u_*^-, \sigma_* (u_*^- \coth(\sigma_* L) - u_*^+\csch(\sigma_*L)), \mu) &0 \\
        0 &\partial_2 f(u_*^+, \sigma_* (-u_*^- \csch(\sigma_* L) + u_*^+\coth(\sigma_*L)), \mu)        
    \end{pmatrix}
    \begin{pmatrix}
        \partial_n|_0 V_0 \\ \partial_n|_L V_0
    \end{pmatrix}. \label{eq:pertdynb}
\end{array}
\end{equation} 
\end{subequations}
at either $\mu = \mu_*^i$ or $\mu = \mu_*^a$. The partial derivatives $\partial_j$ are in this study understood as acting on the $j^\text{th}$ argument of the function it operates on. The two-dimensional function space of solutions of \eqref{eq:pertsys} is spanned by functions of the form
\begin{equation*}
    V_0(x) = V_0^- \frac{\sinh(\sqrt{\lambda+\sigma_*^2}\; (L-x))}{\sinh(\sqrt{\lambda+\sigma_*^2}\; L)} + V_0^+ \frac{\sinh(\sqrt{\lambda+\sigma_*^2}\; x)}{\sinh(\sqrt{\lambda+\sigma_*^2}\; L)}
\end{equation*}
with $V_0^-, V_0^+\in\bbR$ and for $\lambda \in \bbC\setminus(-\infty, -\sigma_*^2]$ with the principal branch taken for the square root, so that $\Re(\sqrt{\lambda + \sigma_*^2}) \geq 0$. Substituting this explicit form into the boundary system \eqref{eq:pertdynb} gives due to 
\begin{eqnarray*} 
    \partial_n|_{x=0} V_0 &=& \sqrt{\lambda+\sigma_*^2} \; V_0^- \coth(\sqrt{\lambda+\sigma_*^2} \; L) - \sqrt{\lambda+\sigma_*^2} \; V_0^+ \csch(\sqrt{\lambda+\sigma_*^2} \; L), \\ 
    \partial_n|_{x=L} V_0 &=& -\sqrt{\lambda+\sigma_*^2} \; V_0^- \csch(\sqrt{\lambda+\sigma_*^2} \; L) + \sqrt{\lambda+\sigma_*^2} \; V_0^+ \coth(\sqrt{\lambda+\sigma_*^2} \; L)
\end{eqnarray*}
nontrivial solutions precisely when 
\begin{equation} \label{eq:d*}
    \dm_*(\lambda) := \det
    \begin{pmatrix}
        \lambda - \partial_1 f_*^- - \partial_2 f_*^- \sqrt{\lambda+\sigma_*^2} \coth(\sqrt{\lambda+\sigma_*^2} \; L) &\partial_2f_*^- \sqrt{\lambda+\sigma_*^2}\; \csch(\sqrt{\lambda+\sigma_*^2} \; L) \\
        \partial_2f_*^+ \sqrt{\lambda+\sigma_*^2}\; \csch(\sqrt{\lambda+\sigma_*^2} \; L) &\lambda - \partial_1 f_*^+ - \partial_2 f_*^+ \sqrt{\lambda+\sigma_*^2} \coth(\sqrt{\lambda+\sigma_*^2} \; L)
    \end{pmatrix} = 0
\end{equation}
where at $\mu = \mu_*^i$ or $\mu = \mu_*^a$ the starred symbols $\partial_1 f_*^-$ and $\partial_1 f_*^+$ respectively stand for $\partial_1 f$ evaluated at the base equilibrium $u_*$ with $u_*^- = u_*(0)$ and $u_*^+ = u_*(L)$, etc.

Since the reaction kinetic function $f$ is the same on the $x=0$ boundary as on the $x=L$ boundary and since the base equilibrium is assumed to be symmetric with $u_*^- = u_*^+$ by Assumption \ref{assu:equi}, the function $\dm_*$ can be factored into $\dm_* = \dm_*^i \cdot \dm_*^a$ and thus is zero precisely when either the eigenvalue
\begin{subequations}
\begin{equation} \label{eq:d*i}
    \dm_*^i(\lambda) := \lambda - \partial_1 f_* - \partial_2 f_* \sqrt{\lambda+\sigma_*^2} \tanh(\sqrt{\lambda+\sigma_*^2} \; \frac{L}{2}) \quad \text{with eigenvector} \quad (1, 1)^T
\end{equation}
of the matrix in \eqref{eq:d*} is zero at $\mu = \mu_*^i$, or the other eigenvalue
\begin{equation} \label{eq:d*a}
    \dm_*^a(\lambda) := \lambda - \partial_1 f_* - \partial_2 f_* \sqrt{\lambda+\sigma_*^2} \coth(\sqrt{\lambda+\sigma_*^2} \; \frac{L}{2}) \quad \text{with eigenvector} \quad (1, -1)^T 
\end{equation}
\end{subequations}
of the matrix in \eqref{eq:d*} is zero at $\mu = \mu_*^a$, which clarifies the meanings of $\mu_*^i$ and $\mu_*^a$. Hence, when $\dm_*^i(\lambda) = 0$ for some $\lambda$ at $\mu = \mu_*^i$, the solution space linearized about $u_*$, denoted by $\mathfrak{L}(u_*)$, is spanned by eigenperturbation $V_i$ with $V_i^- = V_i^+$ that is \emph{in-phase}, if oscillatory ($\lambda \in \bbC \setminus \bbR$), or \emph{symmetric}, if steady ($\lambda \in \bbR$), with exponential onset growth rate $\Re(\lambda)$ and potential oscillation frequency $\Im(\lambda)$. In this case, both $u_*^- + V_i^-$ and $u_*^+ + V_i^+$ are equal. Furthermore, whenever $\dm_*^a(\lambda) = 0$ at $\mu = \mu_*^a$, $\mathfrak{L}(u_*)$ is spanned by eigenperturbation $V_a$ with $V_a^- = -V_a^+$ that are \emph{anti-phase}, if oscillatory ($\lambda \in \bbC \setminus \bbR$), or \emph{antisymmetric}, if steady ($\lambda \in \bbR$), with exponential onset growth rate $\Re(\lambda)$ and potential oscillation frequency $\Im(\lambda)$. In that case, $u_*^- + V_a^-$ and $u_*^+ + V_a^+$ are asymmetric as $u_*+V_a$ is a mixture (linear combination) of a symmetric and an antisymmetric function in $x\in[0,L]$. At distinct points $\mu$, it can also happen that there exists a $\lambda$ such that both $\dm_*^i(\lambda) = 0$ and $\dm_*^a(\lambda) = 0$ at such a single $\mu$. At those $\mu$ and $\lambda$, the corresponding perturbation solution space $\mathfrak{L}(u_*)$ is spanned by both in-phase and anti-phase (or symmetric and antisymmetric) solutions with the same growth rate $\Re(\lambda)$ and frequency $\Im(\lambda)$.

The next assumption of the co-existence of two marginally stable periodic solutions, one in-phase and the other one anti-phase, can be stated. In case only one of them exists, the following theory proceeds similarly but without the interesting potential interaction of the in-phase and anti-phase solutions (see Figure \ref{fig:timesol_beta1_part2} in Section \ref{sec:numericalasympt}). 

\begin{assumption}[Two simple imaginary eigenvalues] \label{assu:simpledegen}
    I assume that 
    \begin{subequations} \label{eq:assu2}
    \begin{eqnarray} 
        \text{(1i)}\; \exists \omega_*^i > 0: \dm_*^i(i\omega_*^i) = 0, && \text{(1a)}\; \exists \omega_*^a > 0: \dm_*^a(i\omega_*^a) = 0, \\
        \text{(2i)} \; \forall \omega \in \bbR \text{ with } |\omega| \neq \omega_*^i: \dm_*^i(i\omega) \neq 0, && \text{(2a)} \; \forall \omega \in \bbR \text{ with } |\omega| \neq \omega_*^a: \dm_*^a(i\omega) \neq 0, \label{eq:assu2b}\\
        \text{(3i)} \; \partial_\lambda|_{\lambda=i\omega_*^i} \dm_*^i(\lambda) \neq 0. && \text{(3a)} \; \partial_\lambda|_{\lambda=i\omega_*^a} \dm_*^a(\lambda) \neq 0.
    \end{eqnarray}
    \end{subequations}
    Assumptions (1i|a) imply the existence of the respective in-phase (i) or anti-phase (a) oscillation with neutral growth and frequency $\omega_*^{i|a}$, (2i|a) imply the absence of solutions with different frequency but also neutral growth rate $\Re(\lambda) = 0$, and (3i|a) imply that the eigenmode with growth rate $\lambda = i\omega_*^{i|a}$ is algebraically simple. The symbol ``$i|a$'' should be read as ``respectively `$i$' or `$a$' ''. The extension of this assumption to $\omega_*^i\downarrow 0$ and $\omega_*^a\downarrow 0$ to include the existence of real eigenvalues that correspond to growth rates to bifurcating steady-states is straightforward.
\end{assumption}

Due to $\dm_*^i(0) \neq 0$ that follows from (2i) in \eqref{eq:assu2b}, we can use the Implicit Function Theorem to track the steady-state $u_*$ as a function of $(\mu, \sigma)$ close to $(\mu_*^i, \sigma_*)$, with the smooth branch of symmetric solutions $(u_*^{-}(\mu, \sigma), u_*^{+}(\mu, \sigma)) = u_*^s(\mu ,\sigma)\cdot (1, 1)$ that is determined by
\begin{eqnarray*}
    f(u_*^{-}(\mu, \sigma), \sigma u_*^{-}(\mu, \sigma) \coth(\sigma L) - \sigma u_*^+\csch(\sigma L), \mu_*^i) = 0, \\
    f(u_*^{+}(\mu, \sigma), -\sigma u_*^{-}(\mu, \sigma) \csch(\sigma L) + \sigma u_*^+\coth(\sigma L), \mu_*^i) = 0.
\end{eqnarray*}
Along this branch, we can linearize and derive a characteristic equation analogous to the above procedure, denoted by $\dm^i$ with 
\begin{subequations} \label{eq:d}
\begin{equation} \label{eq:di}
\begin{split}
    &\dm^i(\lambda; \mu, \sigma) \\
    &:= (1, 0)
    \begin{pmatrix}
        \lambda - \partial_1 f_* - \partial_2 f_* \sqrt{\lambda+\sigma_*^2} \coth(\sqrt{\lambda+\sigma_*^2} \; L) &\partial_2f_* \sqrt{\lambda+\sigma_*^2}\; \csch(\sqrt{\lambda+\sigma_*^2} \; L) \\
        \partial_2f_* \sqrt{\lambda+\sigma_*^2}\; \csch(\sqrt{\lambda+\sigma_*^2} \; L) &\lambda - \partial_1 f_* - \partial_2 f_* \sqrt{\lambda+\sigma_*^2} \coth(\sqrt{\lambda+\sigma_*^2} \; L)
    \end{pmatrix} 
    \begin{pmatrix}
        1 \\ 1
    \end{pmatrix} \\
    &= \lambda - \partial_1 f_* - \partial_2 f_* \sqrt{\lambda+\sigma_*^2} \tanh(\sqrt{\lambda+\sigma_*^2} \; \frac{L}{2}).
\end{split}
\end{equation}
where $\partial_1f_*$ now stands for $f$ evaluated at $u_*^s(\mu, \sigma)$ etc. By Assumption \ref{assu:simpledegen} (1i) and (3i), it holds that $\dm(i\omega_*^i; \mu_*^i, \sigma_*) = 0$ and $\partial_\lambda \dm(i\omega_*^i, \mu_*^i, \sigma_*) \neq 0$, so that by the Implicit Function Theorem  we obtain a unique family of roots $\lambda_*^i(\mu, \sigma)$ for $(\mu, \sigma) \sim (\mu_*^i, \sigma_*)$ with 
\begin{equation*}
    \dm^i(\lambda_*^i(\mu, \sigma); \mu, \sigma) = 0 \quad \text{with} \quad  \lambda_*^i(\mu_*^i, \sigma_*) = i\omega_*^i.
\end{equation*}
The continuation of the anti-phase eigenvalue $\lambda_*^a(\mu, \sigma)$ can be obtained similarly via
\begin{equation} \label{eq:da}
\begin{split}
    &\dm^a(\lambda; \mu, \sigma) \\
    &:= (1, 0)
    \begin{pmatrix}
        \lambda - \partial_1 f_* - \partial_2 f_* \sqrt{\lambda+\sigma_*^2} \coth(\sqrt{\lambda+\sigma_*^2} \; L) &\partial_2f_* \sqrt{\lambda+\sigma_*^2}\; \csch(\sqrt{\lambda+\sigma_*^2} \; L) \\
        \partial_2f_* \sqrt{\lambda+\sigma_*^2}\; \csch(\sqrt{\lambda+\sigma_*^2} \; L) &\lambda - \partial_1 f_* - \partial_2 f_* \sqrt{\lambda+\sigma_*^2} \coth(\sqrt{\lambda+\sigma_*^2} \; L)
    \end{pmatrix} 
    \begin{pmatrix}
        1 \\ -1
    \end{pmatrix} \\
    &= \lambda - \partial_1 f_* - \partial_2 f_* \sqrt{\lambda+\sigma_*^2} \coth(\sqrt{\lambda+\sigma_*^2} \; \frac{L}{2}).
\end{split}
\end{equation}
\end{subequations}

With both continued eigenvalues $\lambda^{i|a}_*$ in place, the next assumption on their behavior as the parameter $\mu$ is varied can be stated.

\begin{assumption}[Strict crossing] \label{assu:bif}
    I assume that 
    \begin{equation*}
        \Re(\partial_\mu \lambda_*^i(\mu_*^i, \sigma_*)) \neq 0 \quad \text{and} \quad \Re(\partial_\mu \lambda_*^a(\mu_*^a, \sigma_*)) \neq 0.
    \end{equation*}
    In words, the exponential growth rates $\Re(\lambda^{i|a})$ of the respective perturbations is increasing or decreasing with nonzero speed as the respective critical parameter value $\mu = \mu_*^i$ and $\mu = \mu_*^a$ is traversed at fixed $\sigma = \sigma_*$.
\end{assumption}

From these assumptions follows the existence of two time-periodic solutions bifurcating from the base branch of temporally constant solutions $u_*$ at generically different points $\mu_*^i$ and $\mu_*^a$. The critical crossings can be computed from $\Re(\lambda_*^{i|a} (\mu, \sigma)) = 0$ for $\mu = \mu_*^{i|a}$ with frequencies $\omega_*^{i|a}(\mu, \sigma) := \Im(\lambda_*^{i|a}(\mu_*^{i|a}(\mu, \sigma), \sigma)$, where this expression is to be read once for upper index ``$i$'' and once for upper index ``$a$''. The extension of this assumption to real eigenvalues $\lambda_*^i$ and $\lambda_*^a$ is straightforward, so that bifurcations from the symmetric base steady-state $u_*$ to further symmetric or asymmetric steady-states can be addressed as well.

\paragraph{Main result --- statement.}

The foundational assumptions for the existence of nonlinear time-periodic solutions bifurcating off the temporally constant but spatially variable base solution $u_*$ are now stated. 

\begin{theorem}[In-phase and anti-phase Hopf oscillations] \label{theorem:hopf}
    Given Assumptions \ref{assu:equi}, \ref{assu:simpledegen}, and \ref{assu:bif} are satisfied and given that the base equilibrium $u_*$ of \eqref{eq:steadystates} is symmetric with $u_*^- = u_*^+ = u_*^s\in\bbR$, the following holds close to $u_*$.
    \begin{enumerate}
        \item[(i)] For sufficiently small amplitudes $|r^i| < R^i$ and sufficiently small $|\sigma-\sigma_*|$, there exist smooth functions $\mu^i_{\text{nl}}$ and $\omega^i_{\text{nl}}$ of $r^i$, $\mu$, and $\sigma$, even in $r^i$, with $\mu^i_{\text{nl}}(r^i=0, \mu, \sigma) = \mu_*^i(\mu, \sigma)$ and $\omega^i_{\text{nl}}(r^i=0, \mu, \sigma) = \omega_*^i(\mu, \sigma)$. With time-parametrization $s^i = \omega^i_{\text{nl}}t$, there exists a time-periodic function 
        \begin{equation*}
            \tilde{u}^i(s^i, x; r^i, \sigma) = \tilde{u}^i(s^i + 2\pi, x; r^i, \sigma)
        \end{equation*}
        such that  
        \begin{equation*}
            u(t, x; r^i,\sigma) = \tilde{u}^i(\omega^i_{\text{nl}} t, x; r^i, \sigma)
        \end{equation*}
        has boundary values 
        \begin{equation*}
            u^-(t; r^i, \sigma) := \tilde{u}^i(\omega^i_{\text{nl}} t, 0; r^i, \sigma) \quad \text{and} \quad u^+(t; r^i, \sigma) := \tilde{u}^i(\omega^i_{\text{nl}} t, L; r^i, \sigma)
        \end{equation*}
        and is a solution to \eqref{eq:sys} with $\mu = \mu_{\text{nl}}(r^i, \mu, \sigma)$. The amplitude of oscillations at the boundary is given by $r^i$, and
        \begin{equation*}
            u(t, 0; r^i, \sigma) = u_*^- + r^i\cos(\omega^i_{\text{nl}} t) + \cO((r^i)^2), \quad u(t, L; r^i, \sigma) = u_*^+ + r^i\cos(\omega^i_{\text{nl}} t) + \cO((r^i)^2).
        \end{equation*}
        The branch of these in-phase oscillatory solutions is unique up to shifts in $s^i$.
        \item[(a)] Furthermore, for sufficiently small amplitudes $|r^a| < R^a$ and sufficiently small $|\sigma-\sigma_*|$, there exist smooth functions $\mu^a_{\text{nl}}$ and $\omega^a_{\text{nl}}$ of $r^a$, $\mu$, and $\sigma$, even in $r^a$, with $\mu^a_{\text{nl}}(r^a=0, \mu, \sigma) = \mu_*^a(\mu, \sigma)$ and $\omega^a_{\text{nl}}(r^a=0, \mu, \sigma) = \omega_*^a(\mu, \sigma)$. With time-parametrization $s^a = \omega^a_{\text{nl}}t$, there exists a time-periodic function 
        \begin{equation*}
            \tilde{u}^a(s^a, x; r^a, \sigma) = \tilde{u}^i(s^a + 2\pi, x; r^a, \sigma)
        \end{equation*}
        such that  
        \begin{equation*}
            u(t, x; r^a,\sigma) = \tilde{u}^a(\omega^a_{\text{nl}} t, x; r^a, \sigma)
        \end{equation*}
        has boundary values 
        \begin{equation*}
            u^-(t; r^a, \sigma) := \tilde{u}^a(\omega^a_{\text{nl}} t, 0; r^a, \sigma) \quad \text{and} \quad u^+(t; r^a, \sigma) := \tilde{u}^a(\omega^a_{\text{nl}} t, L; r^a, \sigma)
        \end{equation*}
        and is a solution to \eqref{eq:sys} with $\mu = \mu_{\text{nl}}(r^a, \mu, \sigma)$. The amplitude of oscillations at the boundary is given by $r^a$, and
        \begin{equation*}
            u(t, 0; r^a, \sigma) = u_*^- + r^a\cos(\omega^a_{\text{nl}} t) + \cO((r^a)^2), \quad u(t, L; r^a, \sigma) = u_*^+ +r^a\cos(\omega^a_{\text{nl}} t + \pi) + \cO((r^a)^2).
        \end{equation*}
        The branch of these anti-phase oscillatory solutions is unique up to shifts in $s^a$.
        \item[(i\&a)] The two boundaries $u^-$ and $u^+$ are coupled to one another with coupling strength varying as $\csch(\sigma L)$ for varying $L$ and $\sigma$, so that the boundaries strongly couple and become one as $\csch(\sigma L)\uparrow \infty$ when $\sigma L \downarrow 0$ and decouple as $\csch(\sigma L) \downarrow 0$ when $\sigma L \uparrow \infty$. Especially, when $L$ and $\sigma$ vary at the same time (non-fixed $\sigma= \sigma_*$) while $\frac{L}{\sigma} = const$, the coupling strength and its effect on the dynamics do not change. For fixed $\sigma=\sigma_* \geq 0$, however, the boundaries completely decouple when $L\to\infty$. Generically, it holds that the in-phase Hopf bifurcation point $\mu_*^i$ and anti-phase Hopf bifurcation point $\mu_*^a$ are different.
    \end{enumerate}
\end{theorem}

\paragraph{Dynamic or Wentzell boundary conditions.} 
Boundary conditions of the form \eqref{eq:bc} with \eqref{eq:dynb} were derived by Wentzell (sometimes spelled Venttsel) in a closed bounded region as a most general boundary condition to an elliptic operator $\cL$ representing an infinitesimal generator of a Markov process in the region \cite{venttsel1959boundary}. They are not boundary conditions in the usual sense since the gradient term is nonlocal at the boundary. A more recent derivation is given in \cite{goldstein2006derivation} for the diffusion and wave equations. For the linear case in arbitrary dimension $N$, it reads $\frac{d}{dt} u(t, x_0) = -a(x_0)u(t, x_0) + b(x_0)  \;\partial_n u(t, x_0)$ for $x \in \partial \Omega$ of a domain $\Omega \subset \bbR^d$ and can be interpreted as an external heat source located at the boundary sending a concentration wave into an infinitesimal layer close to the boundary, towards $\Omega$ if $b(x_0) < 0$ or attracts concentration towards $\partial\Omega$ if $b(x_0) > 0$: $u (t, x) = e^{-at} F (x + bt n)$ where $n$ is the outward unit normal vector and $F$ is some function to be found. In the following, as in \cite{pelz2026oscillations}, I will use the expressions \emph{Wentzell boundary condition} and \emph{dynamic boundary condition} interchangeably. The reasoning behind the latter expression is that the condition involves a partial derivative with respect to time along with one with respect to space. The time derivative cannot easily be eliminated by replacing it, e.g. on the left boundary, by $\Delta u(t, 0) - \sigma^2 u_-$, since the interpretation of $\Delta u(t, 0)$ on the boundary point is unclear. As already seen in the analysis of the system corresponding to \eqref{eq:sys} on the half-unbounded spatial domain $[0, \infty)$  with a single boundary point at $x = 0$, also within this study it will turn out that the key to the presence of a Hopf bifurcation leading to asymptotically stable oscillations is a positive coefficient in the dependence on the flux $\partial_n u$ as, e.g., $\frac{d}{dt} u^- = -\alpha u^- + \sqrt{2\alpha} \; \partial_n|_{x=0} u$. Such terms were motivated in \cite{scheel2021signaling} through directed motion and resulting gradient sensing of stem cells. 

The dynamic boundary condition, which is crucial for the existence of Hopf bifurcations to asymptotically stable oscillations in \eqref{eq:sys}, can also be interpreted as incorporating additional spatial information close to the boundary. More explanations to the dynamic or Wentzell boundary condition are given in \cite{pelz2026oscillations}.

\paragraph{Origin of the system of interest.}
System \eqref{eq:sys} is biologically motivated by its derivation as an adiabatically reduced effective system for modeling planarian flatform regeneration dynamics after cutting and crafting experiments in \cite{scheel2021signaling}. Therein, the dynamic boundary conditions are able to stabilize a global signaling gradient whose existence is undermined by the experimental data, while simple Robin boundary conditions cannot restore a robust gradient field in this manner for flatworm polarity preservation.

\paragraph{Embedding in mathematical biology.}
While system \eqref{eq:sys} exhibits qualitatively similar behaviors to ones obtained for a class of nonlinear Schrödinger/Gross-Pitaevskii equations with a simple double-well potential (cf. \cite{kirr2008symmetry}), I am choosing to focus here especially on embedding this study into the broader research context in mathematical biology. One of the earliest studies that have shown that self-sustained spatiotemporal oscillations can easily arise in systems where localized spatial features in the domain or on the domain boundaries are coupled through linear diffusion (e.g., the boundaries themselves) is \cite{gomez2007self}, illustrated with prototypical FitzHugh-Nagumo kinetics. A few years earlier, it was shown in \cite{levine2005membrane} that such bulk-surface coupling, due to additional permeability or reactive parameters on the surface for the bulk-surface interaction, can also easily lead to temporally steady pattern formation on that surface, even with equal bulk diffusivities in some cases. Later on, various works along the lines of \cite{gomez2021pattern} for boundary surface (or membrane) pattern formation due to diffusive bulk coupling have addressed Turing and Hopf bifurcations in multi-dimensional such systems through asymptotic analysis. The analyzed system of \cite{gou2015synchronized}, in which linearized spectral results close to a simple steady base state and a weakly nonlinear analysis with regard of the sub- or supercriticality of the Hopf branches are provided, is most related to this study for Hopf bifurcations of diffusively coupled boundaries of a bounded one-dimensional domain. Novel in the highly general system \eqref{eq:sys} is that, in contrast to all previously mentioned works in this paragraph, only one diffusiving and reacting species is included, and further novel are the dynamic boundaries that have the potential to change the oscillatory behavior drastically. Furthermore, the present study provides an exhaustive analysis with respect to both steady-state and oscillatory (Hopf) bifurcations.

Diffusive coupling is omnipresent in biological cell systems. Intra-cell membrane coupling of Min proteins through the diffusive cytosol (intracellular space) is described to be a key driver of intracellular synchronized membrane pattern formation that are an ingredient for the symmetry-breaking that leads to cell-division in E. coli bacterial cells \cite{brauns2021bulk}. See a further biological study on intracellular membrane pattern formation via inter-membrane crosstalk in \cite{rassam2018intermembrane}. The mathematical systems described above were mainly modeling intracellular behavior as well, with the exception of \cite{gomez2021pattern} where also extracellular compartment/cell coupling was addressed and of \cite{scheel2021signaling} that aims to model the regeneration dynamics of the whole flatworm organism.  Another well-researched species worth mentioning with regard of intra-organism pattern formation belong to the genus of the fresh-water polyp \emph{hydra} that exhibits similar robust regeneration dynamics as the planarian flatworms (cf. \cite{marciniak20128, wang2023mechano}). On the other hand, the fundamental importance of extracellular diffusive coupling in cell systems (inter-cell signaling) has started to take the research spotlight since Bassler coined the term \emph{quorum sensing} for synchronized switch-like up- or down-regulation of cell behavior in a cell group when the cell density is exceeding some threshold \cite{waters2005quorum}. Mathematical descriptions with the goal of unraveling the mechanisms and resulting dynamics of quorum sensing can be finely tested using chemical beads (cf. \cite{taylor2009dynamical, taylor2015insights}). However, much is not known about biological cell membranes, motivating the dynamic boundary condition as explained above. Quorum sensing only constitutes the starting point of research on the breadth of behaviors of cell groups due to their global diffusive coupling. See \cite{dang2020cellular} on pattern formation without morphogen gradients through diffusive inter-cell communication and on spiral waves in a disordered cell colony. Spatial arrangement within the diffusive domain, in turn, guides cell specialization through gene expression of microbes in microbial communities \cite{gupta2020investigating}. Slow diffusive properties of signaling molecules can further be overcome by the collective activating action of a cell group in a common environment, leading to propagation of diffusive waves with a fixed wave speed and which is called \emph{diffusive signal relaying} in biological literature \cite{dieterle2020dynamics}. Having thoroughly provided biological motivation for the study of system \eqref{eq:sys}, let us move on to why this system is of further mathematical interest.

\paragraph{Effect of delay and buffering}
The gradient dependence and, through it, the dependence on diffusive signal transmission introduces a delay in the communication of the two boundaries of \eqref{eq:sys} with each other. It henceforth is the main driver of the transition to oscillations in the system. Mallet-Paret and Nussbaum have extended our understanding of delay-differential equations in a large series of works (e.g., \cite{mallet1986bifurcation, mallet1986global, nussbaum2003singperturb}), which were highlighting the nonlinear shape of the bifurcating Hopf oscillation far from its bifurcation point that can also be seen in the present study and the previous study on a single dynamic boundary coupled to diffusion of \cite{pelz2026oscillations}. However, their work and most other analytical work on delay-differential equations has solely addressed single delay-differential equations, not systems of such equations. Exceptions include the work \cite{pyragas1998synchronization} on coupled time-delay chaotic systems using both Krasovskii-Lyapunov theory representing an extension of the
second Lyapunov method for delay-differential equations and using perturbation theory for large delay time. In system \eqref{eq:sys} within the present study, the delay time is only large for large domain size $L$. Along the lines of quorum sensing and coupled oscillators, work has been done on Kuramoto oscillators with time-delay in \cite{yeung1999time} with respect to bistability between synchronized and incoherent states. Further work includes the study of two van der Pol oscillators with delayed velocity coupling \cite{yeung1999time} and a series of studies to general parabolic equations with additional time-delay \cite{pao1996dynamics, pao2004global, pao2005stability} mostly illustrated with Lotka-Volterra systems with competition dynamics. As neither of these works provides an explicitly worked out connection of the diffusive-coupling of two differential equations to its induced delay, system \eqref{eq:sys} could provide a minimal starting point to fill this vacuum in the mathematical literature. Herein, neither an equivalent delay-differential equation system to \eqref{eq:sys} nor a further quantification of the precise time-delay of the diffusive coupling along the domain $[0, L]$ is provided, but it can be observed from the illustrations in Section \ref{sec:numericalasympt} how reminiscent the oscillatory solutions to the example system \eqref{eq:sys_tocubic} are to solutions of delay-differential equations.

\paragraph{Outline.} 
Section \ref{s:2} leads to a constructive proof of Theorem \ref{theorem:hopf} for the existence of time-periodic in-phase and anti-phase oscillations branching off the symmetric base steady-state $u_*$. It involves an in-phase and anti-phase Fredholm operator as the linearizations to respective boundary-integral equations after the solution within the bulk $(0, L)$ is constructed. The kernels of the Fredholm operators can then be solved for to arrive at the oscillatory solutions of interest and the local Hopf branch geometry close to the bifurcation point from $u_*$. In Section \ref{sec:examples}, I will choose a general cubic expansion for the boundary reaction kinetics without a constant term, namely $f(u^{-|+}, \partial_n|_{x=0|L} u, \mu) = -\alpha u^{-|+} + \beta (u^{-|+})^2 + \gamma (u^{-|+})^3 + \mu \;\partial_n|_{x=0|L} u$ with $\alpha>0$, $\beta\in \bbR$, $\gamma\in\bbR$ that I then shift by an arbitrary $u_*$, in order to illustrate computing the local leading-order asymptotic expansion of the in-phase and anti-phase Hopf branches as well as the stability change of the base steady-state branch $u_*$. I will do the same for the symmetric and asymmetric steady-state branches bifurcating off the symmetric base branch $u_*$, also addressing the stability properties of those already at that point and their fold bifurcation lines. In Section \ref{sec:stab}, additionally the stability of the in-phase and anti-phase Hopf branches is addressed and the explicit critical curve for the transition from subcriticality to supercriticality in parameter space is provided. Section \ref{sec:attractorcompetition} is devoted to the competition of the dynamical attractors of the system, the stable steady-states and the stable Hopf oscillations, where primarily the relation of the steady-states to the Hopf oscillations close to onset at the base branch is regarded. Section \ref{sec:numericalasympt} serves as the numerical illustration of the theory of this study --- showing the parameter regions of stable symmetric and asymmetric steady-state solutions that bifurcate from the symmetric $u_*$ along with their growth rates (and the \emph{Hopf dance}), showing the critical curves of sub- to supercriticality to the Hopf branches, showing the nonlinear shapes of the in-phase and anti-phase Hopf branches, and showing full spatiotemporal numerical solutions of the system with the cubically expanded reaction kinetic function at the diffusively coupled boundaries. Lastly, the results of this study are discussed in Section \ref{sec:discussion} and interesting next steps are mentioned therein.

\section{Proof of Theorem~\ref{theorem:hopf}}\label{s:2}

This section establishes the statements of Theorem \ref{theorem:hopf} in a constructive fashion.

\paragraph{Reduction to boundary-integral equation.}
First rescale time with $s = \omega t$ where $\omega$ to be found is the oscillation frequency of $u$ at both boundaries, oscillating either in-phase or anti-phase, and now focus on the rescaled function $\tilde{u}(s, x) := u(\frac{s}{\omega}, x)$. The bulk equation then reads
\begin{equation*}
    \omega \partial_s \tilde{u}(s, x) = \partial_{xx} \tilde{u} - \sigma^2 \tilde{u}.
\end{equation*}
Expressing $\tilde{u}$ as the Fourier series $\tilde{u}(s, x) = \sum_{\ell\in\bbZ} \hat{u}_\ell(x) e^{i\ell s}$, the Fourier coefficients need to solve $(i\omega\ell+\sigma^2)\hat{u}_\ell = \hat{u}_\ell''$ on $(0, L)$, so that 
\begin{equation*}
    \forall \ell \in \bbZ\setminus\{0\}: \quad \hat{u}_\ell(x) = \hat{u}_\ell^- \frac{\sinh(\sqrt{i\omega\ell+\sigma^2}\;(L-x))}{\sinh(\sqrt{i\omega\ell+\sigma^2}\;L)} + \hat{u}_\ell^+ \frac{\sinh(\sqrt{i\omega\ell+\sigma^2}\;x)}{\sinh(\sqrt{i\omega\ell+\sigma^2}\;L)}
\end{equation*}
taking the principal branch for the square root and for all $\sigma\geq 0$. Substituting this form into the dynamic boundary equations \eqref{eq:dynb} gives for the time derivative and for the gradient at the left boundary
\begin{subequations} \label{eq:def_D_left}
\begin{eqnarray}
    \dot{u}^- &=& \omega\partial_s\tilde{u}(s, 0) = \sum_{\ell\in\bbZ} i\omega\ell \; \hat{u}_\ell^- \; e^{i\ell s} =: D(\omega) u^-, \\
    \partial_n u(t, 0) &=& \hspace{-0.5cm} \sum_{\ell\in\bbZ\setminus\{0\}} \hspace{-0.2cm}\left(\sqrt{i\omega\ell + \sigma^2}\;\coth(\sqrt{i\omega\ell + \sigma^2}\;L) \; \hat{u}_\ell^- - \sqrt{i\omega\ell + \sigma^2}\;\csch(\sqrt{i\omega\ell + \sigma^2}\;L) \; \hat{u}_\ell^+\right) \;e^{i\ell s} \\
    &&+ \; \sigma \coth(\sigma L) \; \hat{u}_0^- - \sigma \csch(\sigma L) \hat{u}_0^+ \\
    &=:& D(\omega, \sigma)^{1/2} \;u^- -  D_c(\omega, \sigma)^{1/2} \; u^+
\end{eqnarray}
\end{subequations}
and similarly at the right boundary
\begin{subequations} \label{eq:def_D_right}
\begin{eqnarray}
    \dot{u}^+ \hspace{-0.2cm} &=& \omega\partial_s\tilde{u}(s, L) = \sum_{\ell\in\bbZ} i\omega\ell \; \hat{u}_\ell^+ \; e^{i\ell s} =: D(\omega) u^+, \\
    \partial_n u(t, L) \hspace{-0.2cm} &=& \hspace{-0.5cm} \sum_{\ell\in\bbZ\setminus\{0\}} \hspace{-0.3cm} \left(-\sqrt{i\omega\ell + \sigma^2}\;\csch(\sqrt{i\omega\ell + \sigma^2}\;L) \; \hat{u}_\ell^- + \sqrt{i\omega\ell + \sigma^2}\;\coth(\sqrt{i\omega\ell + \sigma^2}\;L) \; \hat{u}_\ell^+\right) \;e^{i\ell s} \\
    \hspace{-0.2cm} &&- \sigma \csch(\sigma L) \; \hat{u}_0^- + \sigma \coth(\sigma L) \hat{u}_0^+ \\
    \hspace{-0.2cm} &=:& -D_c(\omega, \sigma)^{1/2} \;u^- + D(\omega, \sigma)^{1/2} \; u^+
\end{eqnarray}
\end{subequations}
Note that the nonlocal differential operator $D(\omega, \sigma)^{1/2}$ is the square root of $D(\omega)$ when $\sigma = 0$ and is a Dirichlet-to-Neumann operator for the system as $L \uparrow \infty$ (cf. the previous study \cite{pelz2026oscillations}). The differential operator $D_c(\omega, \sigma)^{1/2}$ represents the effect of the diffusive interaction of both boundaries with each other and converges to the zero operator as $L\uparrow \infty$ for $\sigma>0$ as well as for $\sigma = 0$ due to $\lim_{L\to\infty} \lim_{\sigma\to 0} \sigma\csch(\sigma L) = 0 = \lim_{\sigma\to 0} \lim_{L\to\infty} \sigma\csch(\sigma L)$.

With the constructed bulk solution and the rewritten differential operators, system \eqref{eq:sys} can now be restricted to the following time-periodic boundary-integral equation shifted to the equilibrium branch with boundary points $(u^-, u^+) = (u_*^-(\mu, \sigma) + v^-, u_*^+(\mu, \sigma) + v^+)$,
\begin{equation} \label{eq:bndryintegraleq}
\begin{array}{l}
    \begin{pmatrix}
        0 \\ 0
    \end{pmatrix}
    =
    \vF(v^-, v^+, \mu, \omega, \sigma) \\
    := 
    \begin{pmatrix}
        D(\omega) v^- - f(u_*^-(\mu, \sigma) + v^-, \sigma\tanh(\sigma \frac{L}{2}) u_*^- - \sigma\csch(\sigma L) (u_*^+ - u_*^-) + D(\omega, \sigma)^{1/2}\;v^- - D_c(\omega, \sigma)^{1/2}\; v^+, \mu) \\ 
        D(\omega) v^+ - f(u_*^+(\mu, \sigma) + v^+, \sigma\tanh(\sigma \frac{L}{2}) u_*^+ - \sigma\csch(\sigma L) (u_*^- - u_*^+) - D_c(\omega, \sigma)^{1/2}\;v^- + D(\omega, \sigma)^{1/2}\; v^+, \mu)
    \end{pmatrix}
\end{array}
\end{equation}
where $v^- = v(t, 0)$ and $v^- = v(t, L)$ of the perturbed solution $u_*(\mu, \sigma) + v$ with fully nonlinear perturbation $v$ to \eqref{eq:sys}, as opposed to the linearized perturbation $V$ that only solves the linearized system \eqref{eq:pertsys} at $u_*$. In order to proceed, I use Assumption \ref{assu:equi} that $u_*$ is symmetric with $u_*^- = u_*^+ = u_*^s$ and that the reaction kinetic function $f$ is the same at both boundaries. Then \eqref{eq:bndryintegraleq} can be split into a single boundary-integral equation for an in-phase solution $(v^-(t), v^+(t)) = (v^i(t), v^i(t))$, $v^i: \bbR \to \bbR$,
\begin{subequations}
\begin{equation} \label{eq:bndryintegraleq_i}
        F^i(v^i, \mu, \omega^i, \sigma) := D(\omega^i) v^i - f(u_*^s(\mu, \sigma) + v^i, \sigma \tanh(\sigma \frac{L}{2}) \; u_*^s + (D(\omega^i, \sigma)^{1/2} - D_c(\omega^i, \sigma)^{1/2})\; v^i, \mu) = 0,
\end{equation}
and a single boundary-integral equation for an anti-phase solution $(v^-(t), v^+(t)) = (v^a(t), v^a(t+\pi/\omega^a))$, $v^a: \bbR \to \bbR$,  
\begin{equation} \label{eq:bndryintegraleq_a}
    F^a(v^a, \mu, \omega^a, \sigma) := D(\omega^a) v^a - f(u_*^s(\mu, \sigma) + v^a, \sigma\tanh(\sigma \frac{L}{2}) u_*^s + D(\omega, \sigma)^{1/2}\; v^a - D_c(\omega, \sigma)^{1/2}\; v^a(\cdot + \pi/\omega^a), \mu)
\end{equation}
\end{subequations}
since $D(\omega^a) v^a(\cdot\, + \pi/\omega^a) - f(u_*^s(\mu, \sigma) + v^a(\cdot \,+ \pi/\omega^a), \sigma\tanh(\sigma \frac{L}{2}) u_*^s + D(\omega, \sigma)^{1/2}\; v^a(\cdot\, + \pi/\omega^a) - D_c(\omega, \sigma)^{1/2})\; v^a, \mu)$ is the same equation in disguise, with a shifted $v^a$ upon noticing from the Fourier series representation of $v^a$ that $v^a(\cdot + 2\pi/\omega^a) = v^a(\cdot)$.

\paragraph{Properties of the linearization.}

The Jacobian matrix of $\vF$ with respect to $(v^-, v^+)$ about the point $(v^-, v^+) = (0, 0)$ is 
\begin{equation*}
    D\vF = 
    \begin{pmatrix}
        D(\omega) - \partial_1f_* - \partial_2f_*D(\omega, \sigma)^{1/2} &\partial_2f_*D_c(\omega, \sigma)^{1/2} \\
        \partial_2f_*D_c(\omega, \sigma)^{1/2} &D(\omega) - \partial_1f_* - \partial_2f_*D(\omega, \sigma)^{1/2}
    \end{pmatrix}
\end{equation*}
where $\partial_1f_*$ denotes $\partial_1f$ evaluated at $u_*^s$. Matrix $D\vF$ is not invertible at the in-phase critical point $\mu_*^i$ and at the anti-phase critical point $\mu_*^a$ at $\sigma = \sigma_*$ due to Assumption \ref{assu:simpledegen} and thus indicate bifurcation points. Note that the nonlinear $v$ linearized about zero is $V$ from \ref{eq:linearpert}. The linear operator $D\vF$ naturally splits at the respective critical points into its eigenvalues
\begin{subequations} \label{eq:Fredholm_operators}
\begin{equation} \label{eq:Fredholm_operatori}
    \cL^i := (1, 0)\; D\vF(0, \mu_*^i, \omega_*^i, \sigma_*) 
    \begin{pmatrix}
        1 \\ 1
    \end{pmatrix}
    = D(\omega) - \partial_1f_* - \partial_2f_*(D(\omega, \sigma)^{1/2} - D_c(\omega, \sigma)^{1/2})
    : H^2(S^1, \bbR) \to H^1(S^1, \bbR)
\end{equation}
and 
\begin{equation} \label{eq:Fredholm_operatora}
    \cL^a := (1, 0)\; D\vF(0, \mu_*^a, \omega_*^a, \sigma_*) 
    \begin{pmatrix}
        1 \\ -1
    \end{pmatrix}
    = D(\omega) - \partial_1f_* - \partial_2f_*(D(\omega, \sigma)^{1/2} + D_c(\omega, \sigma)^{1/2})
    : H^2(S^1, \bbR) \to H^1(S^1, \bbR).
\end{equation}
\end{subequations}

To proceed with computing the in-phase and anti-phase bifurcation branches near onset respectively at $\mu_*^i$ and $\mu_*^a$ as the kernels of $\cL^i$ and $\cL^a$, it is crucial to check the Fredholm properties of $\cL^i$ and $\cL^a$.

\begin{lemma}[Fredholm properties of $\cL^i$ and $\cL^a$]
    The linear differential operators $\cL^i$ and $\cL^a$ have Fredholm index 0 and their kernels are spanned by $\{e^{is}, e^{-is}\}$, so that $\dim_\bbR(\ker(\cL^i)) = 2 = \dim_\bbR(\ker(\cL^a))$. Furthermore, $\range(\cL^i)^\perp = \ker(\cL^i)$ where $\perp$ refers to the $L^2$ scalar product, and $\range(\cL^a)^\perp = \ker(\cL^a)$.
\end{lemma}
\begin{proof}
    The principal part $D(\omega_*^i)$ of $\cL^i$ minus the identity, $D(\omega_*^i) - 1$, is bounded invertible, as can be seen directly using Fourier series for $D(\omega_*^i) - 1$. The remaining terms of $\cL^i$ are compact since they can be considered as bounded operators from $H^2$ into $H^{3/2}$ due to Remark \ref{rmk:sigma=0} regarding the case $\sigma = 0$. The space $H^{3/2}$, in turn, is compactly embedded in $H^1$, so that $\cL^i$ is Fredholm of index 0. Similar statements hold for $\cL^a$. \\
    Any elements of $\ker(\cL^{i|a})$ can be written as Fourier series. The actions of $\cL^{i|a}$ on Fourier coefficients are diagonal,
    \begin{eqnarray*}
        \cL^i e^{i\ell s} = \left(i\omega_*^i\ell - \partial_1f_* - \partial_2f_*\sqrt{i\omega_*^i\ell + \sigma_*^2} \tanh(\sqrt{i\omega_*^i\ell + \sigma_*^2}\frac{L}{2})\right) \; e^{i\ell s} = \dm^i(i\omega_*^i\ell; 0, \sigma_*) \; e^{i\ell s}, \\
        \cL^a e^{i\ell s} = \left(i\omega_*^a\ell - \partial_1f_* - \partial_2f_*\sqrt{i\omega_*^a\ell + \sigma_*^2} \coth(\sqrt{i\omega_*^a\ell + \sigma_*^2}\frac{L}{2})\right) \; e^{i\ell s} = \dm^a(i\omega_*^a\ell; 0, \sigma_*) \; e^{i\ell s},
    \end{eqnarray*}
    so that $\cL^{i|a} \;e^{i\ell s} = 0$ each when $|\ell|=1$ by Assumption \ref{assu:simpledegen}. Hence, the kernels of $\cL^i$ and $\cL^a$ are each spanned by $\{e^{i\ell s}, e^{-i\ell s}\}$, which, restricting to real functions gives the result. Now note that the complex conjugate of $\sqrt{i\omega_*^i\ell + \sigma_*^2} \tanh(\sqrt{i\omega_*^i\ell + \sigma_*^2}\frac{L}{2})$ is $\sqrt{i\omega_*^i(-\ell) + \sigma_*^2} \tanh(\sqrt{i\omega_*^i(-\ell) + \sigma_*^2}\frac{L}{2})$ and similar for $\sqrt{i\omega_*^i\ell + \sigma_*^2} \coth(\sqrt{i\omega_*^i\ell + \sigma_*^2}\frac{L}{2})$. Thus, the $L^2$-adjoints of $\cL^{i|a}$ respectively have the same form, replacing $\ell$ by $-\ell$, so that the kernels of the adjoints and hence the orthocomplements of the ranges are identical to the kernels of $\cL^{i|a}$, respectively.
\end{proof}

\paragraph{Proof of Theorem~\ref{theorem:hopf}.}

The next step is to split the solution space for $v$ into the space of the respective critical in-phase or anti-phase oscillation spanned by the basis $\{e^{i\ell s}, e^{-i\ell s}\}$ and the complement of this space via Lyapunov-Schmidt splitting, in order to focus on the critical oscillatory solution. This can be done as in the one-boundary case ($L\to \infty$, $\sigma_*\geq 0$) of \cite{pelz2026oscillations}, once with $\cL^i$ and once with $\cL^a$, and the results in (i) and (a) follow with respective reduced equations from which the expansions of $(\mu_{\text{nl}}^i, \omega_{\text{nl}}^i)$ and $(\mu_{\text{nl}}^a, \omega_{\text{nl}}^a)$ in the oscillation amplitudes $r^{i|a}$ can be obtained.

For (i\&a), the interaction of the boundaries with one another is quantified by the multiplier $\csch(\sigma L)$, by the $\partial_2f_* D_c(\omega, \sigma)^{1/2}$ operator as the off-diagonal entries of the linearization $D\vF$ to the boundary-integral equation $\vF = 0$, and by their respective specializations to the in- and anti-phase form, namely $\partial_2f_* D_c(\omega_*^i, \sigma)^{1/2}$ and $\partial_2f_* D_c(\omega_*^a, \sigma)^{1/2}$. All of these converge with $\sigma_* > 0$ and also for with $\sigma_*=0$ (see Remark \ref{rmk:sigma=0}) to zero as $L\to \infty$, so that the boundaries naturally decouple. 

In the in-phase case, the coupling is more accurately quantified in the corresponding Fredholm operator $\cL^i$ through the multiplier $\tanh(\sqrt{i\omega_*^i + \sigma_*^2} \frac{L}{2}) \uparrow 1$ exponentially in $\frac{L}{2}$ and $\sigma_*$ as they approach $\infty$. In the anti-phase case, the multiplier in the corresponding Fredholm operator $\cL^a$ reads $\coth(\sqrt{i\omega_*^i + \sigma_*^2} \frac{L}{2}) \downarrow 1$ exponentially in $\frac{L}{2}$ and $\sigma_*$ as they approach $\infty$. 

Note that the limits commute, i.e., $\lim_{L\to\infty} \lim_{\sigma_*\to 0} D_c^{1/2} = \lim_{\sigma_*\to\infty} \lim_{L\to 0} D_c^{1/2}$. Interestingly, when the domain size becomes infinite while the degradation becomes negligible in a dynamic way with $(L, \sigma_*) \to (\infty, 0)$ while $\sigma L = const$ and when the boundaries come closer until they merge while bulk degradation becomes infinite with $(L, \sigma) \to (0, \infty)$ together with $\sigma L = const$, the boundaries do not decouple and act like the same two coupled boundaries. Especially the former case is of interest in, to a certain scale, long-range systems that have negligible degradation of the diffusing signal between the boundaries. For fixed $\sigma_*$, however, none of these coupled limits is relevant.

Due to difference of $\dm^i$ to $\dm^a$, the in-phase bifurcation point $\mu_*^i$ and anti-phase bifurcation point $\mu_*^a$ are generically different.

\section{Leading-order expansions for up-to-cubic boundary kinetics} \label{sec:examples}

In the following, I am working with a reaction kinetic function $f(u^{-|+}, \partial_n u, \mu)$ at the boundaries that is expanded in $u^{-|+} - u_*^s$ up to cubic order, following the usual normal form theory for scalar differential equations but without a constant term and shifted by any $u_*^s\in\bbR$ of choice. Hence, $u_*$ of \eqref{eq:steadystates} with boundary values $u_*(0) = u_*^s$ and $u_*(L) = u_*^s$ is a symmetric steady-state that serves in this section as the simple base state from Assumption \ref{assu:equi}. Furthermore, I will choose linear dependence on $\partial_n u$, shifted by $\partial_nu_*$, with $\mu$ as its multiplicator. This will both serve as an illustration of Theorem \ref{theorem:hopf} and will lead to further explicit expansions of the respective bifurcation parameters and oscillation frequencies, $(\mu^i, \omega^i)$ and $(\mu^a, \omega^a)$, along with the stability properties of the bifurcation branches for such shifted and cubically expanded $f$. The system of consideration now reads for $t>0$
\begin{subequations} \label{eq:sys_tocubic}
\begin{eqnarray}
    \partial_t u &=& \partial_{xx}u - \sigma^2 u, \quad x\in(0, L), \label{eq:diff_tocubic} \\
    u(t,0) = u^-(t), && u(t,L) = u^+(t), \label{eq:bc_tocubic} \\
    \frac{d}{dt} u^- = f_\text{int}(u^- - u_*^s) + \mu \; \partial_n|_{x=0} (u - u_*), && \frac{d}{dt} u^+ = f_\text{int}(u^+ - u_*^s) + \mu \; \partial_n|_{x=L} (u - u_*).\label{eq:dynb_tocubic}
\end{eqnarray}
\end{subequations}
with strictly internal reaction kinetic function $f_\text{int}(\cdot) = -\alpha (\cdot) + \beta (\cdot)^2 + \gamma (\cdot)^3$ for $\alpha > 0$ and $\beta, \gamma\in\bbR$, so that $f(u^-, \partial_n u(t, 0), \mu) = f_\text{int}(u^- - u_*^s) + \mu \; \partial_n|_{x=0} (u - u_*)$ and similarly at the right boundary $x=L$. Due to the shift $u-u_*$ in the gradient, it holds for the equilibrium $u_*$ that $\partial_\mu u_* = 0$ at both boundary points, and recall from \eqref{eq:steadystate_dn} that $\partial_nu_*(0) = \sigma u_*^s \tanh(\sigma \frac{L}{2})$ and $\partial_nu_*(L) = -\sigma u_*^s \tanh(\sigma \frac{L}{2})$. The restriction $\alpha > 0$ is taken so that $f_\text{int}(u^{-|+} - u_*^s)$ is dynamically attracting to its root $u^{-|+} \equiv u_*^s$ close to $u^{-|+} \equiv u_*^s$. In this way,  interesting stability exchanges with further solution branches of \eqref{eq:sys_tocubic} bifurcating from the symmetric equilibrium $u_*$ could be found.

System \eqref{eq:sys_tocubic} does not generically have a Hopf bifurcation at $\mu = 0$ to either in-phase or anti-phase oscillations, but instead such bifurcations occur under the assumptions to Theorem \ref{theorem:hopf} at some $\mu = \mu_*^i \neq 0$ and $\mu = \mu_*^a \neq 0$, respectively, that can be found easily. Once found, changing the boundary equation term $\mu\;\partial_n (u -u_*)$ to $(\mu + \mu_*^{i|a})\;\partial_n (u- u_*)$ shifts the respective bifurcation point to zero.

\begin{theorem}[Hopf bifurcation points, onset frequencies, and base branch stability change] \label{thm:leading-orderHopf}
    The system \eqref{eq:sys_tocubic} satisfies Assumptions 1 -- 3 with a unique in-phase bifurcation point and onset frequency, $(\mu_*^i, \omega_*^i)$, always when $\sigma_*=0$ and whenever $\sigma_*>0$ only if $\frac{1}{2\sigma_*}\tanh(\sigma_*\frac{L}{2}) + \frac{L}{4}\sech(\sigma_*\frac{L}{2})^2 > \frac{\sigma_*}{\alpha}\tanh(\sigma_*\frac{L}{2})$ and with a unique anti-phase pendant, $(\mu_*^a, \omega_*^a)$, when $\sigma_*=0$ if $\alpha>\frac{12}{L^2}$ and whenever $\sigma_*>0$ if $\frac{1}{2\sigma_*}\coth(\sigma_*\frac{L}{2}) - \frac{L}{4}\csch(\sigma_*\frac{L}{2})^2 > \frac{\sigma_*}{\alpha}\coth(\sigma_*\frac{L}{2})$, on the symmetric base equilibrium branch $u_*$ of \eqref{eq:steadystates} with boundary points $u_*^- = u_*^+ = u_*^s \in\bbR$. 
    
    Let 
    \begin{equation*}
        c^i(\omega) := \sqrt{i\omega + \sigma_*^2} \; \tanh(\sqrt{i\omega + \sigma_*^2} \frac{L}{2}), \quad c^a(\omega) := \sqrt{i\omega + \sigma_*^2} \; \coth(\sqrt{i\omega + \sigma_*^2} \frac{L}{2}).
    \end{equation*}
    Then $(\mu_*^i, \omega_*^i)$ and $(\mu_*^a, \omega_*^a)$ respectively are the unique solutions of
    \begin{subequations} \label{eq:Hopfpoints}
    \begin{equation*} \label{eq:Hopfpoint_i}
        \frac{\omega_*^i}{\Im(c^i(\omega_*^i))} = \frac{\alpha}{\Re(c^i(\omega_*^i))}, \quad \mu_*^i = \frac{\alpha}{\Re(c^i(\omega_*^i))}
    \end{equation*}
    and
    \begin{equation*} \label{eq:Hopfpoint_a}
        \frac{\omega_*^a}{\Im(c^a(\omega_*^a))} = \frac{\alpha}{\Re(c^a(\omega_*^a))}, \quad \mu_*^a = \frac{\alpha}{\Re(c^a(\omega_*^a))}.
    \end{equation*}
    \end{subequations}
    Especially, only the linear part of $f_\text{int}$ evaluated at the base equilibrium affects those points. 
    
    Furthermore, both $\Re(\partial|_{\mu=\mu_*^i} \lambda^i) > 0$ and $\Re(\partial|_{\mu=\mu_*^a} \lambda^a) > 0$ hold, so that the base branch $u_*$ loses stability to the in-phase or anti-phase Hopf branches for respectively $\mu > \mu_*^i$ or $\mu > \mu_*^a$.
\end{theorem}

\begin{proof}
    First assume that there is degradation in the bulk region with $\sigma_* > 0$. A nontrivial in-phase oscillation with onset frequency $\omega_*^i$ bifurcates at $\mu = \mu_*^i$ from the symmetric base steady-state $u_*$ precisely when the characteristic equation 
    \begin{equation} \label{eq:dstari_uptocubic}
        \dm_*^i(i\omega_*^i) = i\omega_*^i - \partial_1f_\text{int}(u_*^s) - \mu_*^i c^i(\omega_*^i) = 0
    \end{equation}
    is satisfied where $c^i(\omega) := \sqrt{i\omega + \sigma_*^2} \; \tanh(\sqrt{i\omega + \sigma_*^2} \frac{L}{2})$ as in the statement of this theorem. The real and imaginary parts of \eqref{eq:dstari_uptocubic} are
    \begin{equation*}
        -\partial_1f_\text{int}(u_*^s) - \mu_*^i\Re(c^i(\omega_*^i)) = 0 \quad \Leftrightarrow \quad \mu_*^i = - \frac{\partial_1f_\text{int}(u_*^s)}{\Re(c^i(\omega_*^i))}
    \end{equation*}
    and 
    \begin{equation*}
        \omega_*^i - \mu_*^i \Im(c^i(\omega_*^i)) = \omega_*^i + \partial_1\fint(u_*^s)\frac{\Im(c^i(\omega_*^i))}{\Re(c^i(\omega_*^i))} = 0.
    \end{equation*}
    Because of searching for a Hopf bifurcation with $\omega_*^i > 0$, it holds that $\Re(c^i(\omega_*^i)) \neq 0 \neq \Im(c^i(\omega_*^i))$, so that I can define
    \begin{equation*}
        W^i(\omega) := 1 - \frac{\alpha}{\omega} \frac{\Im(c^i(\omega))}{\Re(c^i(\omega))}.
    \end{equation*}
    I am now establishing that $W^i(\omega) < 0$ as $\omega\downarrow 0$ if $\frac{1}{2\sigma_*}\tanh(\sigma_*\frac{L}{2}) + \frac{L}{4}\sech(\sigma_*\frac{L}{2})^2 > \frac{\sigma_*}{\alpha}\tanh(\sigma_*\frac{L}{2})$, that $W^i(\omega) \to 1$ as $\omega \uparrow \infty$, and that $W^i$ is strictly monotonically increasing. It would then follow that $W^i$ has a unique root, which then is the unique in-phase onset frequency $\omega_*^i$. To prove these claims, note that 
    \begin{equation*}
        \frac{\Im(c^i(\omega))}{\Re(c^i(\omega))} = \frac{\frac{1}{2\sigma_*}\tanh(\sigma_*\frac{L}{2}) + \frac{L}{4}\sech(\sigma_*\frac{L}{2})^2}{\sigma_*\tanh(\sigma_*\frac{L}{2})} \omega + \cO(\omega^2) \quad \text{as} \quad \omega \downarrow 0,
    \end{equation*}
    so that 
    \begin{equation*}
        \lim_{\omega\downarrow 0} W^i(\omega) < 0 \quad \Leftrightarrow \quad \frac{\frac{1}{2\sigma_*}\tanh(\sigma_*\frac{L}{2}) + \frac{L}{4}\sech(\sigma_*\frac{L}{2})^2}{\sigma_*\tanh(\sigma_*\frac{L}{2})} > \frac{1}{\alpha},
    \end{equation*}
    which is the assumption. Now 
    \begin{equation*}
        \frac{\Im(c^i(\omega))}{\Re(c^i(\omega))} \sim \frac{\Im(\frac{1+i}{\sqrt{2}}\sqrt{\omega}\tanh(\frac{1+i}{\sqrt{2}}\sqrt{\omega}\frac{L}{2}))}{\Re(\frac{1+i}{\sqrt{2}}\sqrt{\omega}\tanh(\frac{1+i}{\sqrt{2}}\sqrt{\omega}\frac{L}{2}))} \to 1 \quad \text{as} \quad \omega \uparrow \infty,
    \end{equation*}
    so that 
    \begin{equation*}
        \lim_{\omega\uparrow\infty} W^i(\omega) > 0 \quad \Leftrightarrow \quad 1 - \lim_{\omega\uparrow \infty} \frac{\alpha}{\omega} > 0 \quad \Leftrightarrow \quad 1 > 0,
    \end{equation*}
    which is a tautology. 
    
    Lastly, it needs to be shown that $W^i$ is strictly monotone. However, using the Mittag-Leffler expansion of $\tanh$ it can even be derived that $W^i$ is strictly monotonically \emph{increasing}. To start, I rewrite with this expansion and $I_n := \left(\frac{\pi}{L}\frac{2n+1}{2}\right)^2$ for $n\in \bbN$
    \begin{eqnarray*}
        c^i(\omega) &=& \sqrt{\sigma_*^2 + i\omega} \; \frac{4}{L} \sum_{n=0}^\infty \frac{\sqrt{\sigma_*^2 + i\omega}}{\sigma_*^2 + i\omega + 4I_n} = \frac{4}{L} \sum_{n=0}^\infty \frac{(\sigma_*^2 + i\omega)(\sigma_*^2 + 4I_n - i\omega)}{(\sigma_*^2 + 4I_n)^2 +\omega^2} = \frac{4}{L} \sum_{n=0}^\infty \frac{\sigma_*^4+\omega^2 + 4\sigma_*^2I_n + i4\omega I_n}{(\sigma_*^2 + 4I_n)^2 +\omega^2}.
    \end{eqnarray*}
    Hence, 
    \begin{eqnarray*}
        \frac{1}{\omega} \frac{\Im(c^i(\omega))}{\Re(c^i(\omega))} = \frac{\sum_{n=0}^\infty \frac{4I_n}{(\sigma_*^2 + 4I_n)^2 +\omega^2}}{\sum_{n=0}^\infty \frac{\sigma_*^4+\omega^2 + 4\sigma_*^2I_n}{(\sigma_*^2 + 4I_n)^2 +\omega^2}},
    \end{eqnarray*}
    of which the numerator is strictly monotonically decreasing as $\omega\in(0, \infty)$ is increasing. That the denominator is strictly monotonically increasing for increasing $\omega$ can be seen by taking the derivative
    \begin{eqnarray*}
        &&\frac{d}{d(\omega^2)} \sum_{n=0}^\infty \frac{\sigma_*^2(\sigma_*^2 +4I_n) +\omega^2}{(\sigma_*^2 + 4I_n)^2 +\omega^2} = \sum_{n=0}^\infty \frac{(\sigma_*^2 + 4I_n)^2 - \sigma_*^2(\sigma_*^2 +4I_n)}{((\sigma_*^2 + 4I_n)^2 +\omega^2)^2} = \sum_{n=0}^\infty \frac{4I_n(4I_n + \sigma_*^2)}{((\sigma_*^2 + 4I_n)^2 +\omega^2)^2} > 0.
    \end{eqnarray*}
    Therefore, $\frac{1}{\omega}\frac{\Im(c^i(\omega))}{\Re(c^i(\omega))}$ is strictly monotonically decreasing, so that $W^i$ is strictly monotonically increasing on $(0, \infty)$. Throughout this proof, no complications due to poles arise, even when $\sigma_*=0$.

    On the other hand, a nontrivial anti-phase oscillation with onset frequency $\omega_*^a$ bifurcates at $\mu = \mu_*^a$ from the symmetric base steady-state $u_*$ precisely when the characteristic equation 
    \begin{equation*}
        \dm_*^a(i\omega_*^a) = i\omega_*^a - \partial_1f_\text{int}(u_*^s) - \mu_*^a c^a(\omega_*^a) = 0
    \end{equation*}
    is satisfied where $c^a(\omega) := \sqrt{i\omega + \sigma_*^2} \; \coth(\sqrt{i\omega + \sigma_*^2} \frac{L}{2})$. The same procedure can be repeated for
    \begin{equation*}
        W^a(\omega) := 1 - \frac{\alpha}{\omega} \frac{\Im(c^a(\omega))}{\Re(c^a(\omega))}
    \end{equation*}
    yielding a unique root $\omega_*^a > 0$ of $W^a$ with corresponding anti-phase Hopf bifurcation point $\mu_*^a = - \frac{\partial_1 \fint(u_*^s)}{\Re(c^a(\omega_*^a))}$.   

    Considering now the case without degradation in the bulk region with $\sigma_* =0$, it holds that
    \begin{equation*}
        W^i(\omega) = 1 - \frac{\alpha}{\omega}\frac{\sinh(L\sqrt{\frac{\omega}{2}}) + \sin(L\sqrt{\frac{\omega}{2}})}{\sinh(L\sqrt{\frac{\omega}{2}}) - \sin(L\sqrt{\frac{\omega}{2}})} \downarrow -\infty \quad \text{as} \quad \omega \downarrow 0, \qquad W^i(\omega) \uparrow 1 \quad \text{as} \quad \omega \uparrow \infty.
    \end{equation*} 
    Hence, there always exists an in-phase bifurcation point as a root of $W^i$ since $\alpha>0$. However, due to 
    \begin{equation*}
        W^a(\omega) = 1 - \frac{\alpha}{\omega}\frac{\sin(L\sqrt{\frac{\omega}{2}}) - \sinh(L\sqrt{\frac{\omega}{2}})}{-\sin(L\sqrt{\frac{\omega}{2}}) - \sinh(L\sqrt{\frac{\omega}{2}})} \to 1 - \alpha \frac{L^2}{12} \quad \text{as} \quad \omega \downarrow 0, \qquad W^i(\omega) \uparrow 1 \quad \text{as} \quad \omega \uparrow \infty,
    \end{equation*} 
    there exists an anti-phase bifurcation point only if $\alpha > \frac{12}{L^2}$ for the case $\sigma_*=0$.

    To show strict (increasing) monotonicity of $W^a$, I am using the Mittag-Leffler expansion for $\coth$ to rewrite with $A_n = \left(\frac{\pi}{L} n\right)^2$ for $n\in\bbN$
    \begin{eqnarray*}
        c^a(\omega) &=& \sqrt{\sigma^2+i\omega}\left( \frac{2}{L\sqrt{\sigma^2+i\omega}} + \frac{4}{L} \sum_{n=1}^\infty \frac{\sqrt{\sigma^2+i\omega}}{\sigma^2+i\omega + 4A_n}\right) = \frac{2}{L} + \frac{4}{L} \sum_{n=1}^\infty \frac{\sigma^2(\sigma^2 + 4A_n) + \omega^2 + i4\omega A_n}{(\sigma^2 + 4A_n)^2 + \omega^2}.
    \end{eqnarray*}
    Hence, 
    \begin{eqnarray*}
        \frac{1}{\omega}\frac{\Im(c^a(\omega))}{\Re(c^a(\omega))} = \frac{\sum_{n=1}^\infty \frac{4\omega A_n}{(\sigma^2 + 4A_n)^2 + \omega^2}}{\frac{1}{2} + \sum_{n=1}^\infty \frac{\sigma^2(\sigma^2 + 4A_n) + \omega^2}{(\sigma^2 + 4A_n)^2 + \omega^2}},
    \end{eqnarray*}
    of which the numerator is strictly monotonically decreasing as $\omega\in(0, \infty)$ is increasing. The denominator is strictly monotonically increasing for increasing $\omega$, as can be seen by taking the derivative
    \begin{eqnarray*}
        && \frac{d}{d(\omega^2)} \sum_{n=1}^\infty \frac{\sigma^2(\sigma^2 + 4A_n) + \omega^2}{(\sigma^2 + 4A_n)^2 + \omega^2} = \sum_{n=1}^\infty \frac{4A_n(4A_n+\sigma^2)}{((\sigma^2 + 4A_n)^2 + \omega^2)^2}  > 0.
    \end{eqnarray*}
    Therefore, $\frac{1}{\omega}\frac{\Im(c^a(\omega))}{\Re(c^a(\omega))}$ is strictly monotonically decreasing, so that $W^a$ is strictly monotonically increasing for all $\omega\in(0, \infty)$. Again, no complications due to poles have arisen.

    Thus, (1i), (2i), (1a), and (2a) of Assumption \ref{assu:simpledegen} hold. Since
    \begin{eqnarray*}
        \partial_1\dm^i(i\omega_*^i) = 1 - \frac{1}{2} \mu_*^i (\sigma_*^2 + i\omega_*^i)^{-\frac{1}{2}} \; \tanh(\frac{L}{2} \sqrt{\sigma_*^2 + i\omega_*^i}) - \frac{L}{4} \mu_*^i \sech(\frac{L}{2} \sqrt{\sigma_*^2 + i\omega_*^i})^2, \\
        \partial_1\dm^a(i\omega_*^a) = 1 - \frac{1}{2} \mu_*^a (\sigma_*^2 + i\omega_*^a)^{-\frac{1}{2}} \; \coth(\frac{L}{2} \sqrt{\sigma_*^2 + i\omega_*^a}) + \frac{L}{4} \mu_*^a \csch(\frac{L}{2} \sqrt{\sigma_*^2 + i\omega_*^a})^2,
    \end{eqnarray*}
    which both are real only when $\omega_*^i = 0$, the statements (3i) and (3a) of Assumption \ref{assu:simpledegen} hold as well. 

    It now remains to establish strict crossing so as to prove that there is indeed a true bifurcation at $(\mu_*^i, \omega_*^i)$ and at $(\mu_*^a, \omega_*^a)$ and not just a point where $\Re(\lambda_*)$ vanishes but does not change sign upon variation of parameter $\mu$ through the respective points. For that, set $z^2 := \sigma^2 + \lambda$, so that the roots of $\dm^i$ and $\dm^a$ respectfully solve
    \begin{eqnarray*}
        z^2 - \sigma^2 + \alpha - \mu z \tanh(z\frac{L}{2}) = 0 \quad \text{at} \quad (\mu, \lambda) = (\mu_*^i, i\omega_*^i), \\
        z^2 - \sigma^2 + \alpha - \mu z \coth(z\frac{L}{2}) = 0 \quad \text{at} \quad (\mu, \lambda) = (\mu_*^a, i\omega_*^a).
    \end{eqnarray*}
    Implicitly differentiating both equations with respect to $\mu$ gives $\partial_\mu\lambda = 2z\partial_\mu z$ at respectfully $(\mu_*^i, i\omega_*^i)$ and $(\mu_*^a, i\omega_*^a)$
    \begin{eqnarray*}
        (2z - \mu\tanh(z\frac{L}{2}) - \frac{L}{2}\mu z \sech(z\frac{L}{2})^2) \partial_\mu z &= z\tanh(z\frac{L}{2}), \\
        (2z - \mu\coth(z\frac{L}{2}) + \frac{L}{2}\mu z \csch(z\frac{L}{2})^2) \partial_\mu z &= z\coth(z\frac{L}{2}).
    \end{eqnarray*}
    Hence, 
    \begin{eqnarray*}
        \partial_\mu\lambda^i = 2z\partial_\mu z = \frac{2z^2\tanh(z\frac{L}{2})}{2z - \mu\tanh(z\frac{L}{2}) - \frac{L}{2}\mu z \sech(z\frac{L}{2})^2} \quad \text{at} \quad (\mu, \lambda) = (\mu_*^i, i\omega_*^i), \\
        \partial_\mu\lambda^a = 2z\partial_\mu z = \frac{2z^2\coth(z\frac{L}{2})}{2z - \mu\coth(z\frac{L}{2}) + \frac{L}{2}\mu z \csch(z\frac{L}{2})^2} \quad \text{at} \quad (\mu, \lambda) = (\mu_*^a, i\omega_*^a).
    \end{eqnarray*}
    Now introduce $s:= \sigma^2 + i\omega_*^i$, so that $z = \sqrt{s}$ taking the principal branch with $\Re(z) > 0$, and introduce the function 
    \begin{equation*}
        \tilde{c}^i(s) := \sqrt{s} \tanh(\sqrt{s}\frac{L}{2}),
    \end{equation*}
    which is even in $z=\sqrt{s}$. Then the in-phase growth rate change with respect to $\mu$ at $\mu_*^i = \frac{\alpha+i\omega_*^i}{\tilde{c}^i(s)}$ can be rewritten as 
    \begin{equation*}
        \partial|_{\mu=\mu_*^i} \lambda^i = \frac{(\tilde{c}^i)^2}{\tilde{c}^i - (\alpha+i\omega_*^i) (\tilde{c}^i)'}.
    \end{equation*}
    The Mittag-Leffler expansion of $\tanh(z\frac{L}{2})$ yields
    \begin{equation*}
        \tilde{c}^i(s) = \frac{4}{L} \sum_{n=0}^\infty  \frac{s}{s + 4I_n} \quad \text{and} \quad (\tilde{c}^i(s))' = \frac{16}{L}\sum_{n=0}^\infty \frac{I_n}{(s+4I_n)^2}
    \end{equation*}
    with $I_n = \left(\frac{\pi}{L} \frac{2n+1}{2}\right)^2$ when differentiated term-wise, 
    since the original series converges uniformly on $\bbC\setminus\{-4I_n\,|\,n\in\bbN\}$. In this way, the real and imaginary parts of $\tilde{c}^i$ are given by
    \begin{eqnarray*}
        \Re(\tilde{c}^i) &=& \frac{4}{L} \sum_{n=0}^\infty \frac{\Re((\sigma^2+i\omega_*^i)(\sigma^2 + 4I_n - i\omega_*^i))}{(\sigma^2 + 4I_n)^2 + (\omega_*^i)^2} = \frac{4}{L} \sum_{n=0}^\infty \frac{|s|^2 + \sigma^2 4I_n}{(\sigma^2 + 4I_n)^2 + (\omega_*^i)^2} > 0, \\
        \Im(\tilde{c}^i) &=& \frac{16}{L} \sum_{n=0}^\infty \frac{\omega_*^i I_n}{(\sigma^2 + 4I_n)^2 + (\omega_*^i)^2} > 0.
    \end{eqnarray*}
    Similarly, the real and imaginary parts of $(\tilde{c}^i)'$ read
    \begin{eqnarray*}
        \Re((\tilde{c}^i)') &=& \Re\left(\frac{16}{L} \sum_{n=0}^\infty \frac{I_n}{(\sigma^2 + 4I_n)^2 - (\omega_*^i)^2 + i8\omega_*^i I_n}\right) = \frac{16}{L} \sum_{n=0}^\infty \frac{\Re(I_n ((\sigma^2 + 4I_n)^2 - (\omega_*^i)^2 - i8\omega_*^i I_n))}{((\sigma^2+4I_n)^2 + (\omega_*^i)^2)^2} \\
        &=& \frac{16}{L} \sum_{n=0}^\infty \frac{I_n ((\sigma^2 + 4I_n)^2 - (\omega_*^i)^2)}{((\sigma^2+4I_n)^2 + (\omega_*^i)^2)^2}, \\
        \Im((\tilde{c}^i)') &=& -\frac{128}{L} \sum_{n=0}^\infty \frac{\omega_*^i I_n^2}{((\sigma^2+4I_n)^2 + (\omega_*^i)^2)^2} < 0.
    \end{eqnarray*}
    Since $\Re(\tilde{c}^i) > 0$, the equation $W^i(\omega_*^i) = 0$ is equivalent to 
    \begin{equation*}
        \omega_*^i\Re(\tilde{c}^i) = \alpha \Im(\tilde{c}^i) \quad \Leftrightarrow \quad \Im((\alpha+i\omega_*^i) \overline{\tilde{c}^i}) = 0
    \end{equation*}
    where $\overline{w}$ denotes the complex-conjugate of a $w\in\bbC$. Hence, $\tilde{c}^i$ is a real multiple $\theta\in\bbR$ of $\alpha+i\omega_*^i$, yielding
    \begin{equation} \label{eq:realmultiple}
        \alpha + i\omega_*^i = \theta \tilde{c}^i \quad \Leftrightarrow \quad \theta = \frac{\alpha}{\Re(\tilde{c}^i)} = \frac{\omega_*^i}{\Im(\tilde{c}^i)} > 0.
    \end{equation}
    Inserting this into the change of the in-phase growth rate with respect to $\mu$ at $\mu_*^i$ gives
    \begin{equation*}
        \partial_{\mu=\mu_*^i} \lambda^i = \frac{(\tilde{c}^i)^2}{\tilde{c}^i - \theta\tilde{c}^i(\tilde{c}^i)'} = \frac{\tilde{c}^i}{1 - \theta(\tilde{c}^i)'}.
    \end{equation*}
    Finally, taking the real part,
    \begin{eqnarray*}
        \Re(\partial_{\mu=\mu_*^i} \lambda^i) &=& \frac{\Re(\tilde{c}^i(1-\theta \overline{(\tilde{c}^i)'}))}{|1-\theta(\tilde{c}^i)'|^2} \\
        &=& \frac{\Re(\tilde{c}^i) - \theta(\Re(\tilde{c}^i)\Re((\tilde{c}^i)') + \Im(\tilde{c}^i)\Im((\tilde{c}^i)'))}{|1-\theta(\tilde{c}^i)'|^2} \\
        &=& \frac{\frac{\Re(\tilde{c}^i)}{\Im(\tilde{c}^i)} (\Im(\tilde{c}^i) - \omega_*^i\Re((\tilde{c}^i)') - \omega_*^i \Im((\tilde{c}^i)')}{|1-\theta(\tilde{c}^i)'|^2}
    \end{eqnarray*}
    using $\theta = \frac{\omega_*^i}{\Im(\tilde{c}^i)}$ of \eqref{eq:realmultiple} in the last step. 

    Since $\Re(\tilde{c}^i)>0$, $\Im(\tilde{c}^i)>0$, and $\Im((\tilde{c}^i)')<0$, it only remains to show that $\Im(\tilde{c}^i) - \omega_*^i\Re((\tilde{c}^i)')$ is positive. For this, rewrite
    \begin{eqnarray*}
        \Im(\tilde{c}^i) - \omega_*^i\Re((\tilde{c}^i)') = \frac{32}{L}\sum_{n=0}^\infty \frac{(\omega_*^i)^3I_n}{((\sigma^2+4I_n)^2 + (\omega_*^i)^2)^2} > 0, 
    \end{eqnarray*}
    which proves that $\partial|_{\mu=\mu_*^i}\lambda^i > 0$. The denominator $1-\theta(\tilde{c}^i)'$ of $\partial|_{\mu=\mu_*^i}\lambda^i$ is vanishing if and only if
    \begin{equation*}
        1 = \theta \Re((\tilde{c}^i)') \quad \wedge \quad \Im((\tilde{c}^i)') = 0 \quad \Leftrightarrow \quad \omega_*^i = 0 \quad\wedge \quad 1 = \theta\left(\frac{1}{2\sigma_*}\tanh(\sigma_*\frac{L}{2}) + \frac{L}{4}\sech(\sigma_*\frac{L}{2})^2\right),
    \end{equation*}
    which remains well-defined for $\sigma_*\downarrow 0$. However, when $\omega_*^i=0$, the corresponding bifurcation point is a symmetric steady-state bifurcation point whose stability property is derived in Theorem \ref{thm:asymptexpansion_stab_steady-states}.

    The proof of $\partial|_{\mu=\mu_*^a}\lambda^a>0$ can be carried out in a similar way. See Remark \ref{rmk:stbproof} for more details.
\end{proof}

\begin{remark}[Notes on the stability proof of Theorem \ref{thm:leading-orderHopf}] \label{rmk:stbproof} 
    \begin{enumerate}[label={\arabic*}]
        \item (Strictly positive contributions). The identity $\tilde{c}^i(s) = \frac{4}{L} \sum_{n=0}^\infty \frac{s}{s + 4I_n}$ with $I_n = \left(\frac{\pi}{L}\frac{2n+1}{2}\right)^2$ is the spectral (Dirichlet eigenfunction) representation of the map $s\to \sqrt{s}\tanh(\sqrt{s}\frac{L}{2})$, which is why every quantity in $\partial|_{\mu=\mu_*^i}\lambda^i$ decomposes into a sum of manifestly signed modal contributions. Each term pushes $\Re(\partial|_{\mu=\mu_*^i}\lambda^i)>0$ in the positive direction, with no cancelation possible. This is the structural reason why the result holds uniformly in $\alpha>0$, $L>0$, and $\sigma_*\geq 0$.
        \item (Usage of assumptions). The assumption $\alpha>0$ quarantees $\theta>0$ in \eqref{eq:realmultiple}, while the assumption $\sigma_*\geq 0$ guarantees that $\Re(\tilde{c}^i) > 0$. The domain length $L>0$ fixes the positive ``eigenvalues'' $I_n$, and the defining equation for $\omega_*^i$, namely $W^i(\omega_*^i) = 0$, is used only through the collinearity relation \eqref{eq:realmultiple}, which is the key step that collapses $\Re(\partial|_{\mu=\mu_*^i}\lambda^i)$ to $\frac{\tilde{c}^i}{1 - \theta (\tilde{c}^i)'}$ with a \emph{real} coefficient $\theta$.
        \item (Stability change at anti-phase Hopf point). The proof of $\Re(\partial|_{\mu=\mu_*^a}\lambda^a)>0$ carries over verbatim because the argument for the case $\Re(\partial|_{\mu=\mu_*^i}\lambda^i)>0$ rests on two structural facts: (I) the growth rate derivative can be collapsed to a resolvant-like form, $\partial|_{\mu=\mu_*^i}\lambda^i = \frac{\tilde{c}^i}{1 - \theta(\tilde{c}^i)'}$ with the mentioned real coefficient $\theta > 0$ supplied by the phase condition $W^i(\omega_*^i) = 0$, and (II) $\tilde{c}^i$ admits a Herglotz-type expansion $c_0 + \frac{4}{L}\sum_{n=0}^\infty \frac{s}{s + 4I_n}$ with $c_0\geq 0$ and positive eigenvalues $I_n$. Replacing $\tanh$ by $\coth$ for $\partial|_{\mu=\mu_*^a}\lambda^a$ only changes the boundary condition of the underlying abstract spectral problem: the Dirichlet-type eigenvalues $I_n = \left(\frac{\pi}{L}\frac{2n+1}{2}\right)^2$ become the Neumann-type  eigenvalues $A_n:= \left(\frac{\pi}{L} n\right)^2$ for $n\in\bbN$, and a constant mode $\frac{2}{L}$ appears in the representation for $\tilde{c}^a$, just as for $c^a$ in the uniqueness part of the proof. Every modal contribution to the denominator of $\Re(\partial|_{\mu=\mu_*^a}\lambda^a)$ retains its sign, so that the conclusion $\Re(\partial|_{\mu=\mu_*^a}\lambda^a)>0$ is again uniform in $\alpha>0$, $L>0$, and $\sigma_*\geq 0$.
    \end{enumerate}
\end{remark}

Let us now turn for a moment to bifurcations to steady-state branches and the implications of Theorem \ref{thm:leading-orderHopf} for those.

\begin{remark}[Bifurcations to steady-states] \label{rmk:steady-state-bif}
    Theorem \ref{thm:leading-orderHopf} extends to the case of bifurcation onset frequency $\omega_* \downarrow 0$ (in-phase and anti-phase) since this limit is unproblematic for $\sigma_*>0$. For $\sigma_* = 0$,
    \begin{equation*}
        c^i(\omega_*) = \sqrt{i\omega_*}\tanh(\frac{L}{2}\sqrt{i\omega_*}) \sim \frac{L}{2} i\omega_* \to 0 \quad \text{and} \quad c^a(\omega_*) = \sqrt{i\omega_*}\coth(\frac{L}{2}\sqrt{i\omega_*}) \to \frac{2}{L}
    \end{equation*}
    are also well-defined limits. The introduced differential operators become for steady-states
    \begin{equation*}
        D(0) = 0, \quad D(0, \sigma)^{1/2} = \sigma\coth(\sigma L), \quad \text{and} \quad D_c(0, \sigma)^{1/2} = \sigma \csch(\sigma L)
    \end{equation*}
    for all $\sigma \geq 0$. Hence, beside Hopf bifurcations there is the further possibility for bifurcations to nontrivial symmetric and asymmetric steady-states to exist, bifurcating off the (symmetric) base state branch \eqref{eq:steadystates} with $(u_*^-, u_*^+) = u_*^s \cdot (1, 1)$. Here I call a steady-state where both boundaries exhibit the same value, $u_*^s\cdot (1, 1) + v^s\cdot (1,1), v^s\in\bbR$, \emph{symmetric}, and a steady-state where both boundaries exhibit the opposite dislocation from the symmetric base state, $u_*\cdot(1, 1) + v^{as}\cdot (1, -1), v^{as}\in\bbR$, I call \emph{asymmetric}. Note that such an asymmetric state actually is a mixed (symmetric-antisymmetric) state, as found for a class of Schrödinger and Gross-Pitaevskii equations  that often describe ground states of certain particles in quantum systems in case of a simple double-well potential in \cite{kirr2008symmetry}. More to this is shortly discussed and illustrated at the end of Section \ref{sec:numericalasympt}.

    The bifurcation point $\mu_*^s$ to a symmetric steady-state different from the symmetric base steady-state $u_*$ is precisely given by
    \begin{subequations}
    \begin{equation*} 
        \dm^i(0; \mu_*^s, \sigma_*) = -\partial_1\fint(u_*^s) - \mu_*^s\sigma_*\tanh(\sigma_*\frac{L}{2}) = 0 \quad \Leftrightarrow \quad \mu_*^s = \frac{\alpha}{\sigma_*\tanh(\sigma_*\frac{L}{2})} > 0 \quad \text{if } \sigma_*> 0
    \end{equation*}
    (does not exist for $\sigma_*=0$ since $\partial_nu_* = 0$ in this case) and the bifurcation point $\mu_*^{as}$ to an asymmetric steady-state is simply given by
    \begin{equation*} 
        \dm^a(0; \mu_*^a, \sigma_*) = -\partial_1\fint(u_*^s) - \mu_*^{as}\sigma_*\coth(\sigma_*\frac{L}{2}) = 0 \quad \Leftrightarrow \quad \mu_*^{as} =
        \begin{cases}
            \frac{\alpha}{\sigma_*\coth(\sigma_*\frac{L}{2})} > 0 &\text{for } \sigma_*>0, \\
            \alpha\frac{L}{2} > 0 &\text{for } \sigma_*=0.
        \end{cases}
    \end{equation*}
    \end{subequations}
    Interesting questions now include which of these steady and oscillatory states are winning out (are preferably attained) in the system at each point in the parameter space for $(\alpha, \beta, \gamma, \sigma, L)$. To furthermore find bifurcations from antisymmetric, $(u_*^-, u_*^+) = u_*^{as}\cdot (1, -1), u_*^{as}\in\bbR$, or asymmetric, $(u_*^-, u_*^+) = u_*^{s}\cdot (1, 1) + v^{as}\cdot (1, -1), u_*^{s}, v^{as}\in\bbR$, steady-states as base states, one needs to exchange the symmetric base state taken in this study, $(u_*^-, u_*^+) = u_*^{a}\cdot (1, 1), u_*^{s}\in\bbR$, with the one of interest and derive the respective boundary-integral equations and characteristic equations.
\end{remark}

As the nonlinearities in the boundary reaction kinetics $\fint$ will generally bend the bifurcation branches, I now turn to quantifying such corrections to the Hopf ``branches'' through correcting terms for the bifurcation parameter $\mu$ and oscillation frequency $\omega$ along the bifurcating branch close to its emergence. Later on I will show numerically how the oscillation profile changes when traversing both the in-phase and anti-phase Hopf branches far from their onset from the base equilibrium branch. For such numerical considerations, the corrections to $\mu$ and $\omega$ initially, close to $(\mu_*^i, \omega_*^i)$ and close to $(\mu_*^a, \omega_*^a)$, are needed.

By Theorem \ref{thm:leading-orderHopf}, the assumptions of Theorem \ref{theorem:hopf} hold and thus there are in-phase (i) and anti-phase (a) branches of solutions with the respective expansions in the in-phase oscillation amplitude $r^i$ and anti-phase oscillation amplitude $r^a$
\begin{subequations} \label{eq:asymptexpansions}
    \begin{eqnarray} 
        &\mu = \mu_*^i + \mu_2^i \cdot (r^i)^2 + \cO((r^i)^4), \quad \mu = \mu_*^a + \mu_2^a \cdot (r^a)^2 + \cO((r^a)^4) \quad \text{(generically, $\mu_*^i\neq \mu_*^a, \mu_2^i\neq\mu_2^a$)}, \\
        &\omega = \omega_*^i + \omega_2^i \cdot (r^i)^2 + \cO((r^i)^4), \quad \omega = \omega_*^a + \omega_2^a \cdot (r^a)^2 + \cO((r^a)^4) \quad \text{(generically, $\omega_*^i\neq \omega_*^a, \omega_2^i\neq\omega_2^a$)}.
    \end{eqnarray}
\end{subequations}
When computing the expansion in the respective oscillation amplitude, the action of the linearizations and Fredholm operators, $\cL^i$ and $\cL^a$, on Fourier modes contribute key coefficients. Therefore, I define
\begin{equation} \label{eq:d_coeffs}
    \Lambda_\ell^{i|a} := \dm^{i|a} (i\omega_*^{i|a}\ell; \mu_*^{i|a}, \sigma_*), \quad \Lambda_{\ell, \mu}^{i|a} := \partial_2\dm^{i|a} (i\omega_*^{i|a}\ell; \mu_*^{i|a}, \sigma_*), \quad \Lambda_{\ell, \omega}^{i|a} := i\partial_1\dm^{i|a} (i\omega_*^{i|a}\ell; \mu_*^{i|a}, \sigma_*)
\end{equation}
where the upper index ``$i|a$'' again means that the reader should take either ``$i$'' or ``$a$'' for the whole definition.

\begin{theorem}[Expansion of the bifurcating in-phase and anti-phase Hopf branches] \label{thm:asymptcorr}
    Fix $\sigma_*\geq 0$ and consider the system \eqref{eq:sys_tocubic} respectively with $\mu\sim\mu_*^i$ and $\mu\sim\mu_*^a$. Recall the definitions of $\dm^i$ and $\dm^a$ from \eqref{eq:di} and \eqref{eq:da}. Then the bifurcating in-phase Hopf bifurcation branch and the generically different bifurcating anti-phase Hopf bifurcation branch have the expansions \eqref{eq:asymptexpansions} with 
    \begin{subequations} \label{eq:asymptmu2om2}
        \begin{equation}
            \mu_2^{i|a} = \frac{\Im(M^{i|a} \; \overline{\Lambda_{1, \omega}^{i|a}})}{\Im(\Lambda_{1, \mu}^{i|a}\; \overline{\Lambda_{1, \omega}^{i|a}})}, \quad \omega_2^{i|a} = \frac{\Im(M^{i|a} \; \overline{\Lambda_{1, \mu}^{i|a}})}{\Im(\Lambda_{1, \omega}^{i|a}\; \overline{\Lambda_{1, \mu}^{i|a}})}
        \end{equation}
        where the definitions \eqref{eq:d_coeffs} are used together with the abbreviation
        \begin{equation}
            M^{i|a} = \frac{3}{4} \gamma + \beta^2 ((\Lambda_0^{i|a})^{-1} + (2\Lambda_2^{i|a})^{-1}).
        \end{equation}
    \end{subequations}
    Again, the reader should take either ``$i$'' or ``$a$'' for the upper index ``$i|a$'' for the whole identity.
\end{theorem}

\begin{proof}
    The boundary-integral equation functions $F^i$ of \eqref{eq:bndryintegraleq_i} and $F^a$ of \eqref{eq:bndryintegraleq_a} can be expanded in the respective in-phase oscillation amplitude $r^i$ and anti-phase amplitude $r^a$ of its arguments just as in the proof of Theorem 3.2 of \cite{pelz2026oscillations}. This leads to the result and additionally to the onset profiles of the oscillatory solutions in both cases. 
\end{proof}

Let us now turn to the symmetric and antisymmetric steady-state perturbations of \eqref{eq:sys_tocubic}, $v^s\cdot (1, 1)$ and $v^{as}\cdot (1, -1)$ and derive their bifurcation branches at their respective onset at $\mu_*^s$ and $\mu_*^{as}$, along with their stability properties. 

\begin{theorem}[Expansion and stability of the bifurcating symmetric and asymmetric steady-state branches] \label{thm:asymptexpansion_stab_steady-states}
    Recall the expressions for the symmetric and asymmetric steady-state bifurcation points from Remark \ref{rmk:steady-state-bif} for system \eqref{eq:sys_tocubic} with fixed $\sigma_*\geq 0$. The bifurcating symmetric steady-state branch exists for $\gamma\neq 0$ if and only if $4\gamma(\alpha-\mu\sigma_*\tanh(\frac{L}{2}\sigma_*)) > - \beta^2$ and for $\gamma = 0$ if and only if $\beta \neq 0$. The generically different bifurcating asymmetric steady-state branch exists for $\gamma\neq 0$ if and only if $4\gamma(\alpha-\mu\sigma_*\coth(\frac{L}{2}\sigma_*)) > - \beta^2$ and for $\gamma = 0$ if and only if $\beta \neq 0$. The branches have the respective expansions
    \begin{subequations}
        \begin{eqnarray}
            \mu = \mu_*^s + \mu_1^s \cdot v^s + \mu_2^s  \cdot (v^s)^2 + \cO((v^s)^4) \quad \text{with} \quad \mu_1^s = -\frac{\beta}{\sigma_*\tanh(\frac{L}{2}\sigma_*)}, \quad \mu_2^s = -\frac{\gamma}{\sigma_*\tanh(\frac{L}{2}\sigma_*)}, \\
    \mu = \mu_*^{as} + \mu_1^{as} \cdot v^{as} + \mu_2^{as}  \cdot (v^{as})^2 + \cO((v^{as})^4)  \quad \text{with} \quad \mu_1^{as} = -\frac{\beta}{\sigma_*\coth(\frac{L}{2}\sigma_*)}, \quad \mu_2^{as} = -\frac{\gamma}{\sigma_*\coth(\frac{L}{2}\sigma_*)},
        \end{eqnarray}
        assuming $v^s$ and $v^{as}$ are small for this expansion to hold. 
        
        Furthermore, close to bifurcation onset, at least one of the symmetric steady-state(s) (up to two of count; see Theorem \ref{thm:num_stable_sol}) is (case: $\gamma=0$) / are (case: $\gamma \neq 0$) stable when the base branch is strictly unstable, so if $\mu > \mu_*^i \;\wedge \; \frac{1}{2\sigma_*}\tanh(\sigma_*\frac{L}{2}) + \frac{L}{4}\sech(\sigma_*\frac{L}{2})^2 < \frac{\sigma_*}{\alpha}\tanh(\sigma_*\frac{L}{2})$ and if $\mu < \mu_*^i \;\wedge \; \frac{1}{2\sigma_*}\tanh(\sigma_*\frac{L}{2}) + \frac{L}{4}\sech(\sigma_*\frac{L}{2})^2 > \frac{\sigma_*}{\alpha}\tanh(\sigma_*\frac{L}{2})$. At least one of the asymmetric steady-state(s) is (case: $\gamma = 0$) / are (case: $\gamma\neq 0$) stable close to bifurcation onset when the base branch is strictly unstable, so if $\mu > \mu_*^a \;\wedge \; \frac{1}{2\sigma_*}\coth(\sigma_*\frac{L}{2}) - \frac{L}{4}\csch(\sigma_*\frac{L}{2})^2 < \frac{\sigma_*}{\alpha}\coth(\sigma_*\frac{L}{2})$ and if $\mu < \mu_*^a \;\wedge \; \frac{1}{2\sigma_*}\coth(\sigma_*\frac{L}{2}) - \frac{L}{4}\csch(\sigma_*\frac{L}{2})^2 > \frac{\sigma_*}{\alpha}\coth(\sigma_*\frac{L}{2})$.
        The explicit expressions of the steady-states and explicit expansions of perturbation growth rates $\lambda$ along the base and bifurcating branches can easily be computed and are provided in this proof, so that one can read off which part of the nontrivial bifurcation branches are stable.
    \end{subequations}
\end{theorem}
\begin{proof}
    Here it is crucial to expand the bifurcation parameter more generally as
\begin{eqnarray*}
    \mu &=& \mu_*^s + \mu_1^s \cdot v^s + \mu_2^s  \cdot (v^s)^2 + \mu_3^s  \cdot (v^s)^3 + \cO((v^s)^4), \\
    \mu &=& \mu_*^{as} + \mu_1^{as} \cdot v^{as} + \mu_2^{as}  \cdot (v^{as})^2 + \mu_3^{as}  \cdot (v^{as})^3 + \cO((v^{as})^4)
\end{eqnarray*}
since no restricting information on the dependence of $\mu$ on respectively $v^s$ and $v^{as}$ is available yet. Further, inserting those expansions into the respective boundary-integral equation with $\omega^{i|a} = 0$, $F^i(v^s, \mu, 0, \sigma_*) = 0$ of \eqref{eq:bndryintegraleq_i} and $F^a(v^{as}, \mu, 0, \sigma_*) = 0$ of \eqref{eq:bndryintegraleq_a}, and solving the algebraic equations in orders of respectively $v^s$ up to and including $\cO((v^s)^4)$ and $v^{as}$ up to and including $\cO((v^{as})^4)$ yields
\begin{eqnarray*}
    \mu_1^s = -\frac{\beta}{\sigma_*\tanh(\frac{L}{2}\sigma_*)}, \quad \mu_2^s = -\frac{\gamma}{\sigma_*\tanh(\frac{L}{2}\sigma_*)}, \quad\mu_3^s = 0, \\
    \mu_1^{as} = -\frac{\beta}{\sigma_*\coth(\frac{L}{2}\sigma_*)}, \quad \mu_2^{as} = -\frac{\gamma}{\sigma_*\coth(\frac{L}{2}\sigma_*)}, \quad \mu_3^{as} = 0.
\end{eqnarray*}
Since system \eqref{eq:sys_tocubic} is shifted in the symmetric base steady-state, the nontrivial symmetric solution with boundary values $(u_*^s+v^s)\cdot(1, 1)$ can now be found from the boundary-integral equation through
\begin{eqnarray*}
    &&F^i(v^s, \mu, 0, \sigma_*) = \alpha v^s - \beta (v^s)^2 - \gamma (v^s)^3 - \mu\sigma_*\tanh(\frac{L}{2}\sigma_*) v^s = 0 \\
    &\Leftrightarrow& v^s = 0 \quad \vee \quad 
    \begin{cases}
        v^s_{1|2} = -\frac{\beta}{2\gamma} \pm \frac{1}{2\gamma}\sqrt{\beta^2 + 4\alpha\gamma - 4\gamma\mu\sigma_*\tanh(\frac{L}{2}\sigma_*)}, & \gamma \neq 0 \\
        v^s_1 = \frac{1}{\beta}(\alpha-\mu\sigma_*\tanh(\frac{L}{2}\sigma_*)), & \beta \neq 0, \gamma = 0.
    \end{cases}
\end{eqnarray*}
Hence, $v^s_{1|2}$ exist for $\gamma\neq 0$ if and only if $4\gamma(\alpha-\mu\sigma_*\tanh(\frac{L}{2}\sigma_*)) > - \beta^2$ and for $\gamma = 0$ if and only if $\beta \neq 0$.

On the other hand, the nontrivial asymmetric solution with boundary values $u_*^s\cdot (1, 1) + v^{as}\cdot(1, -1)$ can be similarly found with $F^a$ to be 
\begin{equation*}
    \begin{cases}
        v^{as}_{1|2} = -\frac{\beta}{2\gamma} \pm \frac{1}{2\gamma}\sqrt{\beta^2 + 4\alpha\gamma - 4\gamma\mu\sigma_*\coth(\frac{L}{2}\sigma_*)}, & \gamma \neq 0 \\
        v^{as}_1 = \frac{1}{\beta}(\alpha-\mu\sigma_*\coth(\frac{L}{2}\sigma_*)), & \beta \neq 0, \gamma = 0.
    \end{cases}
\end{equation*}
Thus, $v^{as}_{1|2}$ exist for $\gamma\neq 0$ if and only if $4\gamma(\alpha-\mu\sigma_*\coth(\frac{L}{2}\sigma_*)) > - \beta^2$ and for $\gamma = 0$ if and only if $\beta \neq 0$.

It remains to derive the stability properties of the base branch of the emerging steady-state branches close to each bifurcation onset and to show an exchange of stability with the base branch $u_*$. In the following, I am doing so simultaneously for the symmetric and asymmetric steady-state using 
\begin{equation*}
    c^i(\lambda) := \sqrt{\lambda+\sigma_*^2} \tanh(\sqrt{\lambda+\sigma_*^2}\, \frac{L}{2}) \quad \text{and} \quad c^a(\lambda) := \sqrt{\lambda+\sigma_*^2} \coth(\sqrt{\lambda+\sigma_*^2}\, \frac{L}{2}).
\end{equation*}
Abusing notation slightly by denoting the space-dependent emerging solution branches by $v^{s|as}(x)$, perturb the newly emerging ($s$: symmetric, $as$: asymmetric) branches $u_*+v^{s|as}$ through $u = u_*+v^{s|as} + e^{\lambda t} v$ with perturbation $v$ close to steady-state bifurcation onset, so that $|v|\ll 1$ and $\partial_tu = \lambda ve^{\lambda t}$, and linearize \eqref{eq:sys_tocubic} now about the newly emerging branches $u_*+v^{s|as}$. A similar procedure as done in Section \ref{s:2} to arrive at \eqref{eq:d} yields the characteristic equations 
\begin{eqnarray*}
    \dm^{i|a}(\lambda; \mu, \sigma) &=& \lambda - \partial_1 f(u_*^s(\mu, \sigma) + v^{s|as}(\mu,\sigma), c^i(\lambda)u_*^s + c^{i|a}(\lambda)v^{s|as}, \mu) \\
    &&- \partial_2 f(u_*^s(\mu, \sigma) + v^{s|as}(\mu,\sigma), c^i(\lambda)u_*^s + c^{i|a}(\lambda)v^{s|as}, \mu) \,\sqrt{\lambda+\sigma_*^2} \\
    &=& \lambda + \alpha - 2\beta v^{s|as} - 3\gamma (v^{s|as})^2 - \mu c^{i|a}(\lambda) = 0
\end{eqnarray*}
where again the usual notation $v^{s|as}$ for the boundary points is used. 
Expanding in $v^{s|as}$,
\begin{eqnarray*}
    \dm^{i|a}(\lambda; \mu, \sigma) &=& \frac{d\lambda}{dv^{s|as}} C^{i|a}_*v^{s|as} - \beta v^{s|as} - 2\gamma(v^{s|as})^2 \quad \text{with} \quad C^{i|a}_* := 1 - \mu_*^{s|as} (c^{i|a})'(0)
\end{eqnarray*}
since $\mu^{s|as} = \frac{\alpha}{c^{i|a}(0)} - \frac{\beta}{c^{i|a}(0)}v^{s|as} - \frac{\gamma}{c^{i|a}(0)}(v^{s|as})^2 + \cO((v^{s|as})^4)$. Then 
\begin{eqnarray*}
    \frac{d\lambda}{dv^{s|as}} = (C^{i|a}_*)^{-1} (\beta + 2\gamma v^{s|as}) + \cO((v^{s|as})^2),
\end{eqnarray*}
so that 
\begin{equation*}
    \lambda = (C^{i|a}_*)^{-1} (\beta v^{s|as} + 2\gamma (v^{s|as})^2) + \cO((v^{s|as})^3)
\end{equation*}
close to bifurcation onset.

For the base branch perturbation growth rate $\lambda_b$, it holds that $\partial_\mu|_{\mu=\mu_*^{s|as}} \lambda_b = -\partial_\mu\dm^{i|a}_* / \partial_\lambda\dm^{i|a}_* = c^{i|a}(0) / C^{i|a}_*$, which gives
\begin{equation*}
    \lambda_b = \frac{c^{i|a}(0)}{C^{i|a}_*} (\mu - \mu_*^{s|as}) + \cO((\mu - \mu_*^{s|as})^2)= (C^{i|a}_*)^{-1} (-\beta v^{s|as} - \gamma(v^{s|as})^2) + \cO((v^{s|as})^3)
\end{equation*}
close to bifurcation onset, where the last equality is stated as it leads to an expression that can be compared to the previous $\lambda$. Hence, there is an exchange of stability between the base branch and the bifurcating steady-state branches, and all stabilities close to bifurcation onset can be determined by determining the stability of the base branch at the respective $\mu$-value. Note that $C^i_*>0$ whenever $\frac{1}{2\sigma_*}\tanh(\sigma_*\frac{L}{2}) + \frac{L}{4}\sech(\sigma_*\frac{L}{2})^2 < \frac{\sigma_*}{\alpha}\tanh(\sigma_*\frac{L}{2})$, so when no in-phase Hopf bifurcation exists, and $C^i_*<0$ when $\frac{1}{2\sigma_*}\tanh(\sigma_*\frac{L}{2}) + \frac{L}{4}\sech(\sigma_*\frac{L}{2})^2 > \frac{\sigma_*}{\alpha}\tanh(\sigma_*\frac{L}{2})$, in which case an in-phase Hopf bifurcation exists at some value for $\mu$ (see Theorem \ref{thm:leading-orderHopf}). Similarly, $C^a_*>0$ whenever $\frac{1}{2\sigma_*}\coth(\sigma_*\frac{L}{2}) - \frac{L}{4}\csch(\sigma_*\frac{L}{2})^2 < \frac{\sigma_*}{\alpha}\coth(\sigma_*\frac{L}{2})$ with obvious extension to $\sigma_*\to 0$, so when no anti-phase Hopf bifurcation exists, and $C_*^a<0$ if $\frac{1}{2\sigma_*}\coth(\sigma_*\frac{L}{2}) - \frac{L}{4}\csch(\sigma_*\frac{L}{2})^2 > \frac{\sigma_*}{\alpha}\coth(\sigma_*\frac{L}{2})$, for which there is an anti-phase Hopf bifurcation.

In summary, with an exchange of stability between base branch and steady-state bifurcation branches at the respective bifurcation points, 
\begin{eqnarray*}
    \mu>\mu_*^s: \;\lambda_b (\mu)
    \begin{cases}
        > 0 \quad \text{if} \quad \frac{1}{2\sigma_*}\tanh(\sigma_*\frac{L}{2}) + \frac{L}{4}\sech(\sigma_*\frac{L}{2})^2 < \frac{\sigma_*}{\alpha}\tanh(\sigma_*\frac{L}{2}) \quad\text{ (unstable base branch)} \\
        < 0 \quad \text{if} \quad \frac{1}{2\sigma_*}\tanh(\sigma_*\frac{L}{2}) + \frac{L}{4}\sech(\sigma_*\frac{L}{2})^2 > \frac{\sigma_*}{\alpha}\tanh(\sigma_*\frac{L}{2}) \quad\text{ (stable base branch)}
    \end{cases} \\
    \mu>\mu_*^{as}:\; \lambda_b (\mu)
    \begin{cases}
        > 0 \quad \text{if} \quad \frac{1}{2\sigma_*}\coth(\sigma_*\frac{L}{2}) - \frac{L}{4}\csch(\sigma_*\frac{L}{2})^2 < \frac{\sigma_*}{\alpha}\coth(\sigma_*\frac{L}{2}) \quad\text{ (unstable base branch)} \\
        < 0 \quad \text{if}  \quad \frac{1}{2\sigma_*}\coth(\sigma_*\frac{L}{2}) - \frac{L}{4}\csch(\sigma_*\frac{L}{2})^2 > \frac{\sigma_*}{\alpha}\coth(\sigma_*\frac{L}{2}) \quad\text{ (stable base branch)}
    \end{cases}
\end{eqnarray*}
and with flipped inequalities for respectively $\mu<\mu_*^s$ and $\mu<\mu_*^{as}$. Further, for the perturbation growth rates $\lambda^s$ and $\lambda^{as}$ from respectively the non-base symmetric steady-state branch for $\mu>\mu_*^s$ and asymmetric steady-state branch for $\mu>\mu_*^{as}$,
\begin{eqnarray*}
    \beta v^s<0: \;\lambda^s (\mu)
    \begin{cases}
        < 0 \quad \text{if} \quad \frac{1}{2\sigma_*}\tanh(\sigma_*\frac{L}{2}) + \frac{L}{4}\sech(\sigma_*\frac{L}{2})^2 < \frac{\sigma_*}{\alpha}\tanh(\sigma_*\frac{L}{2}) \\
        > 0 \quad \text{if} \quad \frac{1}{2\sigma_*}\tanh(\sigma_*\frac{L}{2}) + \frac{L}{4}\sech(\sigma_*\frac{L}{2})^2 > \frac{\sigma_*}{\alpha}\tanh(\sigma_*\frac{L}{2}) 
    \end{cases} \\
    \beta v^s>0:\; \lambda^s (\mu)
    \begin{cases}
        > 0 \quad \text{if} \quad \frac{1}{2\sigma_*}\tanh(\sigma_*\frac{L}{2}) + \frac{L}{4}\sech(\sigma_*\frac{L}{2})^2 < \frac{\sigma_*}{\alpha}\tanh(\sigma_*\frac{L}{2}) \\
        < 0 \quad \text{if} \quad \frac{1}{2\sigma_*}\tanh(\sigma_*\frac{L}{2}) + \frac{L}{4}\sech(\sigma_*\frac{L}{2})^2 > \frac{\sigma_*}{\alpha}\tanh(\sigma_*\frac{L}{2})
    \end{cases}
\end{eqnarray*}
and
\begin{eqnarray*}
    \beta v^{as}<0: \;\lambda^{as} (\mu)
    \begin{cases}
        < 0 \quad \text{if} \quad \frac{1}{2\sigma_*}\coth(\sigma_*\frac{L}{2}) - \frac{L}{4}\csch(\sigma_*\frac{L}{2})^2 < \frac{\sigma_*}{\alpha}\coth(\sigma_*\frac{L}{2}) \\
        > 0 \quad \text{if}  \quad \frac{1}{2\sigma_*}\coth(\sigma_*\frac{L}{2}) - \frac{L}{4}\csch(\sigma_*\frac{L}{2})^2 > \frac{\sigma_*}{\alpha}\coth(\sigma_*\frac{L}{2})
    \end{cases} \\
    \beta v^{as}>0:\; \lambda^{as} (\mu)
    \begin{cases}
        > 0 \quad \text{if} \quad \frac{1}{2\sigma_*}\coth(\sigma_*\frac{L}{2}) - \frac{L}{4}\csch(\sigma_*\frac{L}{2})^2 < \frac{\sigma_*}{\alpha}\coth(\sigma_*\frac{L}{2}) \\
        < 0 \quad \text{if}  \quad \frac{1}{2\sigma_*}\coth(\sigma_*\frac{L}{2}) - \frac{L}{4}\csch(\sigma_*\frac{L}{2})^2 > \frac{\sigma_*}{\alpha}\coth(\sigma_*\frac{L}{2}).
    \end{cases}
\end{eqnarray*}
\end{proof}

The previous proof directly implies the following statement on the location of the steady-state fold bifurcation lines.

\begin{cor}[fold bifurcation lines] \label{cor:fold}
    If $\gamma \neq 0$, at $\mu = \mu^{if}$ with
    \begin{equation*}
        \mu^{if} = \frac{1}{\sigma_*}(\alpha + \frac{\beta^2}{4\gamma^2})\coth(\frac{L}{2}\sigma_*),
    \end{equation*}
    the symmetric steady-states $v_1^s$ and $v_2^s$ merge and create a fold bifurcation point, and, at $\mu = \mu^{af}$ with
    \begin{equation*}
        \mu^{af} = \frac{1}{\sigma_*}(\alpha +\frac{\beta^2}{4\gamma^2})\tanh(\frac{L}{2}\sigma_*),
    \end{equation*}
    the antisymmetric steady-states $v_1^{as}$ and $v_2^{as}$ merge and create a fold bifurcation point. Note that $\mu^{if} \uparrow \infty$ as $\sigma_*\downarrow 0$ and that $\mu^{af} \uparrow \frac{L}{2}(\alpha +\frac{\beta^2}{4\gamma^2})$ as $\sigma_*\downarrow 0$, so that the symmetric steady-states do not merge for $\sigma_* = 0$ while the antisymmetric steady-states still do so at a finite $\mu^{af}$ as $\sigma_*\downarrow 0$.
\end{cor}

\section{Stability and instability for up-to-cubic boundary kinetics}\label{sec:stab}

In the previous section, the stability properties of the symmetric and asymmetric steady-state bifurcation branches of the system \eqref{eq:sys_tocubic} specialized to cubically expanded reaction kinetics $\fint$ were provided in Theorem \ref{thm:asymptexpansion_stab_steady-states}. Furthermore, Theorem \ref{thm:leading-orderHopf} provided information to the stability change to instability for the symmetric base branch $u_*$ when $\mu$ is increasing respectively through the in-phase and anti-phase Hopf points. This section dives deeper into the stability properties of the Hopf branches and provides expressions for the critical curve in the parameter space where the subcritical Hopf branches turn supercritical.

I am first stating the Stability Conjecture. It remains an open problem to prove (or disprove) this conjecture from the boundary-integral equations \eqref{eq:bndryintegraleq_i} and \eqref{eq:bndryintegraleq_a} in an elegant and minimal way.

\begin{conj}[Exchange of stability with base branch] \label{conj:stability}
    I conjecture that if the symmetric base branch $u_*$ loses stability as $\mu$ increases through either the in-phase Hopf point $\mu_*^{i}$, the anti-phase Hopf point $\mu_*^a$, the symmetric steady-state bifurcation point $\mu^s_*$, or the asymmetric steady-state bifurcation point $\mu_*^{as}$, the respective non-base bifurcation branch has the opposite stability than the base branch at the same value for $\mu$. Hence, the base branch and, e.g., the emerging in-phase Hopf branch cannot be either simultaneously stable or unstable at the same point $\mu\neq \mu_*^i$ close to bifurcation onset $\mu=\mu_*^i$.
\end{conj}

\begin{proof}[Proof strategy]
    Theorem \ref{fig:steadystatebranches} has established the statement for the bifurcating steady-state branches. The strategy for the Hopf branches is to expand the respective Fredholm operator $\cL^i$ and $\cL^a$ in the respective oscillation amplitudes $r^i$ and $r^a$ and corresponding eigenvalues $\lambda$. It then remains to check that $\Re(\lambda) < 0$ for the in-phase Hopf branch whenever $\mu_2^i > 0$ and $\Re(\lambda) > 0$ whenever $\mu_2^i < 0$, and similarly for the anti-phase Hopf branch. This quite tedious computation has been done for Theorem 4.1 of \cite{pelz2026oscillations} for a single dynamic boundary at $x=0$ adjacent to a diffusive field on the infinite half line $(0, \infty)$. Therefore, I defer a similar computation for the coupled boundary system \eqref{eq:sys_tocubic}, in search for a \emph{more elegant} and \emph{minimal way} to determine the Hopf branch stability with known base branch stability.
\end{proof}

\begin{theorem}[Onset stability of in-phase and anti-phase Hopf branches] \label{thm:stabilityofbranches}
    Assuming the Stability Conjecture \ref{conj:stability} holds, generically the in-phase and anti-phase bifurcation branches are respectively spectrally stable when $\mu_2^i>0$ and $\mu_2^a>0$ and respectively spectrally unstable when $\mu_2^i<0$ and $\mu_2^a<0$. The transition between sub- and supercriticality is given by the parameter curve $\gamma = \gamma_\text{crit}^{i|a}$ with
    \begin{equation*} \label{eq:gammacrit}
        \gamma_\text{crit}^{i|a}(\alpha, \beta, L, \sigma) := \frac{4}{3}\beta^2 \left( \Im((\Lambda_0^{i|a})^{-1} + (2\Lambda_2^{i|a})^{-1}) \frac{\Re(\Lambda_{1, \omega}^{i|a})}{\Im(\Lambda_{1,\omega}^{i|a})} - \Re((\Lambda_0^{i|a})^{-1} + (2\Lambda_2^{i|a})^{-1}) \right)
    \end{equation*}
    using the coefficients introduced in \eqref{eq:d_coeffs}.
\end{theorem}

\begin{proof}
    By Theorem \ref{thm:leading-orderHopf}, the symmetric base branch $u_*$ loses stability as $\mu$ increases through $\mu_*^{i}$. Since the assumption is that the Stability Conjecture \ref{conj:stability} holds, the in-phase Hopf branch must be spectrally unstable for $\mu_2^i<0$ (in-phase Hopf branch bending toward $\mu < \mu_*^i$ where $u_*$ is stable) and spectrally stable for $\mu_2^i>0$ (in-phase Hopf branch bending toward $\mu > \mu_*^i$ where $u_*$ is unstable), except for the degenerate case of $\mu_*^i = \mu_*^a$. The same holds for the anti-phase Hopf branch that emerges at such generically different $\mu=\mu_*^a$.

    For the Hopf branches to transition near the bifurcation points ($r^i$ and $r^a$ are respectively small) between being spectrally stable to spectrally unstable, respectively $\mu_2^i = 0$ and $\mu_2^a = 0$ have to be satisfied, which amounts to 
    \begin{equation*}
        \Im(M^{i|a} \; \overline{\Lambda_{1, \omega}^{i|a}}) = \Im(M^{i|a}) \Re(\Lambda_{1, \omega}^{i|a}) - \Re(M^{i|a}) \Im(\Lambda_{1, \omega}^{i|a}) \overset{!}{=} 0.
    \end{equation*}
    Noticing that the first summand of $\Re(M^{i|a})$ is $\frac{3}{4}\gamma$ proves the claim.
\end{proof}

\section{Competition between attractors for up-to-cubic boundary kinetics} \label{sec:attractorcompetition}

In Remark \ref{rmk:steady-state-bif}, I have provided the location of the bifurcation points to symmetric and asymmetric steady-states branching off the symmetric base state $u_*$ that exist for wide parameter ranges. This means that in the simple example system \eqref{eq:sys_tocubic} with a cubically expanded nonlinearity $\fint$ at both boundaries multiple dynamical behaviors could be possible. In this section the attempt is made to provide insight into which of the stable solutions that potentially coexist are selected by the system, so as to know which solution the system will converge to for large times. 

\begin{theorem}[Number of stable solutions close to bifurcation onset] \label{thm:num_stable_sol}
    The functions $\dm^i(\lambda; \mu, \sigma)$ defined in \eqref{eq:di} and $\dm^a(\lambda; \mu, \sigma)$ defined in \eqref{eq:da} are \emph{Evans functions} for system \eqref{eq:sys_tocubic} and have maximally two roots each, $(\lambda_1^i, \lambda_2^i)$ and $(\lambda_1^a, \lambda_2^a)$, with positive real parts. Therefore, at no point an in-phase oscillation is competing or mixing with a nontrivial (non-base) symmetric steady-state and at no point an anti-phase oscillation is competing or mixing with an asymmetric steady-state close to their onset at the base branch. By \emph{mixing} a state is meant that is a linear combination of two existent attracting states with nearly the same real part magnitude of the respective eigenvalue $\lambda$ that represent their growth rate away from the base state $u_*$.
\end{theorem}

\begin{proof}
    Let $\Gamma_{-\varepsilon, R} := \{z = -\varepsilon+Re^{i\varphi} \;|\; \varphi\in [-\pi, \pi]\} \cup \{z = -\varepsilon+iR-\tau i2R \;|\; \tau\in(0, 1)\}$ for $0<\varepsilon\ll 1$ and $R>0$ be a closed curve parametrized in counterclockwise direction and with $\Re(\lambda)>-\varepsilon$ for all $z\in \mathrm{int}(\Gamma_{-\varepsilon,R})$, where $\mathrm{int}(\Gamma)$ denotes the open interior of a closed curve $\Gamma$. Let $\Gamma_{-\varepsilon, R}^C := \{z = -\varepsilon+Re^{i\varphi} \;|\; \varphi\in [-\pi, \pi]\}$ be the half-circle arc of $\Gamma_{-\varepsilon, R}$ and let $\Gamma_{-\varepsilon, R}^I := \{z = -\varepsilon+iR-\tau i2R \;|\; \tau\in(0, 1)\}$ be the remaining closing line on the imaginary axis of $\Gamma_{-\varepsilon, R}$.

    Since $\tanh(\sqrt{\sigma_*^2+\lambda}\frac{L}{2})$ only has poles at $\{\lambda_k\}$ with 
    \begin{equation*}
        i\frac{2k+1}{2}\pi = \sqrt{\sigma_*^2+\lambda}\frac{L}{2} \quad \text{for} \quad k\in\bbZ \quad \Leftrightarrow \quad \lambda_k = -\frac{(2k+1)^2}{L^2}\pi^2 - \sigma_*^2 < 0 \quad \text{for} \quad k\in\bbN,
    \end{equation*}
    the Evans function $\dm^i$ has poles neither along the curve $\Gamma_{-\varepsilon,R}$ nor in the closure $\overline{\mathrm{int}(\Gamma_{-\varepsilon,R})}$, even for $\varepsilon\downarrow 0$. Thus, the number $Z^i$ of $\lambda$ with $\Re(\lambda)\geq 0$ that are zeros of $\dm^i$ can be computed using the argument principle by
    \begin{equation*}
        Z^i = \lim_{\varepsilon\downarrow 0}\lim_{R\uparrow\infty} \frac{1}{2\pi i} \int_{\Gamma_{-\varepsilon, R}} \frac{\partial_\lambda \dm^{i}(\lambda; \mu, \sigma_*)}{\dm^{i}(\lambda; \mu, \sigma_*)} \; d\lambda = \lim_{\varepsilon\downarrow 0}\lim_{R\uparrow \infty} \Delta_{{\Gamma_{-\varepsilon, R}}} \mathrm{arg}(\dm^{i}(\lambda; \mu, \sigma_*))
    \end{equation*}
    where the latter right-hand side denotes the change of the argument of $\dm^i$ along ${\Gamma_{-\varepsilon, R}}$, which is also called the winding number of $\dm^i$ along ${\Gamma_{-\varepsilon, R}}$.
    Instead of dealing with $\partial_\lambda \dm^i / \dm^i$ explicitly, it is easier to proceed with analyzing the behavior of the argument of $\dm^i$. I will do so in the following. Let $z := \sqrt{\sigma_*^2+\lambda}$, taking the principal branch so that $\Re(z)\geq0$. Then 
    \begin{equation*}
        \dm^i(\lambda) = \lambda + \alpha - \mu c^i(\lambda) \quad \text{with} \quad c^i(\lambda) = \sqrt{\sigma_*^2+\lambda} \;\tanh(\sqrt{\sigma_*^2+\lambda}\frac{L}{2}) = z \tanh(z\frac{L}{2}).
    \end{equation*}
    Recall from the proof of Theorem \eqref{thm:leading-orderHopf} that, using the Mittag-Leffler expansion of $\tanh$,
    \begin{equation*}
        c^i(\lambda) = \frac{4}{L} \sum_{n=0}^\infty \frac{z^2}{z^2 + 4I_n} \quad \text{with} \quad I_n := \left(\frac{\pi}{L}\frac{2n+1}{2}\right)^2 \quad \forall n\in \bbN
    \end{equation*}
    and with summands of size $\cO(I_n^{-1}) = \cO(n^{-2})$. Hence, the series converges locally uniformly away from the points $\{z^2 = -4I_n \;|\; n\in \bbN\}$ (the poles of $\tanh(z\frac{L}{2})$), so that $\dm^i$ is analytic on $\{\Re(\lambda) > -4I_0 = \frac{\pi^2}{L^2}\} \supset \{\Re(\lambda) \geq 0\}$, its only singularities being the poles at $\lambda = -\sigma_*^2 - 4I_n < 0$.

    I will use the following identities repeatedly, valid for $I > 0$:
    \begin{equation} \label{eq:ReIm_ci}
        \Re\left(\frac{z^2}{z^2+I}\right) = \frac{|z^2|^2 + I\Re(z^2)}{|I + z^2|^2}, \quad \Im\left(\frac{z^2}{z^2+I}\right) = \frac{I\Im(z^2)}{|I + z^2|^2}.
    \end{equation}
    All zeros of $\dm^i$ in $\{\Re(\lambda)\geq 0\}$ lie in a bounded set. One can see that by noticing that for $\Re(\lambda)\geq 0$ and $|\lambda|\to\infty$ one has $\mathrm{arg}(\sigma_*^2+\lambda) \in [-\frac{\pi}{2}, \frac{\pi}{2}]$, $\mathrm{arg}(z) \in [-\frac{\pi}{4}, \frac{\pi}{4}]$, and $\Re(z) \geq \frac{|z|}{\sqrt{2}} \to \infty$, so that $\tanh(z\frac{L}{2})\to 1$ uniformly. If $\dm^i(\lambda)=0$, then $|\lambda|\leq \alpha + |\mu||z||\tanh(z\frac{L}{2})| \leq \alpha + 2|\mu|\sqrt{\sigma_*^2+|\lambda|}$ for large $|\lambda|$, which is impossible for large $|\lambda|$. Therefore, zeros in the closed right half-plane lie in a bounded set and are finitely many. The limit as $R\to \infty$ therefore stabilizes at a $Z^i = \#\{\text{zeros of $\dm^i$ in $\{\Re(\lambda)\geq 0\}$}\}$.

    The same asymptotics give on the arc $\Gamma_{-\varepsilon, R}^C$ with $|\mathrm{\arg(\lambda)|\leq \frac{\pi}{2}}$ that $\dm^i(\lambda) \sim \lambda(1-\cO(R^{-1/2}))$ as $\varepsilon\downarrow 0$, so that the change of $\mathrm{arg}(\dm^i)$ along $\Gamma_{-\varepsilon, R}^C$ tends to $\pi$ as $R\uparrow \infty$ and as $\varepsilon\downarrow 0$.

    I am now regarding the sign structure of $\Im(\dm^i)$ generally on vertical lines in $\bbC$. With fixed $\varepsilon \in [0, \frac{\pi^2}{L^2})$, consider the line $\{\lambda = -\varepsilon+i\tau\;|\; \tau\in\bbR\}$, with $\Re(z^2 + I_n) \geq I_0 - \varepsilon > 0$ for all $n\in\bbN$. By \eqref{eq:ReIm_ci}, 
    \begin{equation*}
        \Im(\dm^i(-\varepsilon+i\tau)) = \tau(1-\mu F_\varepsilon(\tau)) \quad \text{where}\quad F_\varepsilon := \frac{16}{L} \sum_{n=0}^\infty \frac{I_n}{(\sigma_*^2-\varepsilon+I_n)^2 + \tau^2}.
    \end{equation*}
    The function 
    \begin{subequations} \label{note:Feps}
    \begin{equation} \label{note:increasingFeps}
        \text{$F_\varepsilon$ is finite, continuous, strictly decreasing in $\tau\geq 0$, and $F_\varepsilon(\tau) \downarrow 0$ as $\tau\to\infty$.}
    \end{equation}
    Furthermore, the map
    \begin{equation} \label{note:onezero_posImline}
        \tau \mapsto 1-\mu F_\varepsilon(\tau) \text{ has at most one zero on $(0, \infty)$}.
    \end{equation}
    \end{subequations}
    For $\mu > 0$, this map is strictly increasing and, for $\mu\leq 0$, it is greater than 1. Consequently, $\Im(\dm^i(-\varepsilon+i\tau))$ changes sign at most once for $\tau\in(0, \infty)$, and, for large $\tau$, $\Im(\dm^i(-\varepsilon+i\tau)) \sim \tau > 0$. Due to the conjugation symmetry $\dm^i(\overline{\lambda}) = \overline{\dm^i(\lambda)}$, the behavior of $\dm^i$ for $\tau<0$ mirrors that for $\tau>0$ and $\dm^i$ is real on the real axis. Now all foundational information is gathered to prove the claim.

    Along the counterclockwise contour $\Gamma_{-\varepsilon, R}$ with small $\varepsilon>0$ to avoid the poles of $\dm^i$, I collect all zeros of $\dm^i$ in $\{\Re(\lambda) \geq -\varepsilon\}$. The contribution of the arc $\Gamma_{-\varepsilon, R}^C$ to the winding number of $\dm^i$ is $\pi + o(1)$ as $R\to \infty$. Now specifically focus on the upper half of the vertical line $\Gamma_{-\varepsilon, R}^I$ with $\Im(\lambda)$ ranging from $R$ to 0. Choosing the continuous branch of $\arg(\dm^i)$ with $\arg(\dm^i)\to \frac{\pi}{2}$ as $R\to \infty$ at the top of $\Gamma_{-\varepsilon, R}^I$, by \eqref{note:onezero_posImline} there is at most one $\tau_*\in(0, \infty)$ where $\Im(\dm^i)$ vanishes. Let us split the path there. On $(\tau_*, \infty)$, it holds by \eqref{note:Feps} that $\Im(\dm^i)>0$, so that the continuous argument stays in the open interval $(0, \pi)$. It starts at $\frac{\pi}{2}$ and ends at a real value for $\dm^i$ at $\tau_*$ with argument 0 or $\pi$. Hence, there is at most a $\frac{\pi}{2}$ argument change in absolute value possible on $(\tau_*, \infty)$. On $(0, \tau_*)$, in turn, it holds by \eqref{note:Feps} that $\Im(\dm^i)<0$, so that the argument stays in an interval of length $\pi$, namely $(\pi, 2\pi)$ (traversed left to right) or $(-\pi, 0)$ (traversed right to left), depending on the branch value of $\mathrm{arg}(\dm^i)$ at $\tau_*$, and ends at the real value $\dm^i(-\varepsilon)$. Thus, the argument change along the upper vertical segment of $\Gamma_{-\varepsilon, R}^I$ with $\Im(\lambda)\geq0$ is at most $\frac{3\pi}{2}$ in absolute value. If $\tau_*$ does not exist, the bound improves to $\frac{\pi}{2}$. The argument change along the lower segment of $\Gamma_{-\varepsilon, R}^I$ with $\Im(\lambda)\leq 0$ contributes the same amount due to $\dm^i(\overline{\lambda}) = \overline{\dm^i(\lambda)}$, again at most $\frac{3\pi}{2}$.

    Adding up all the contributions, 
    \begin{equation*}
        Z^i \leq \frac{1}{2\pi} \left(\pi + \frac{3\pi}{2} + \frac{3\pi}{2} + o(1)\right) = 2 + o(1) \quad \text{as} \quad \varepsilon\to 0 \quad \text{and} \quad R\to \infty.
    \end{equation*}
    Therefore, $Z^i \leq 2$ in $\{\Re(\lambda)\geq 0\}$.

    The same proof structure can be carried out for 
    \begin{equation*}
        \dm^a(\lambda) = \lambda + \alpha - \mu c^a(\lambda) \quad \text{with} \quad c^a(\lambda) = \sqrt{\sigma_*^2+\lambda} \; \coth(\sqrt{\sigma_*^2+\lambda}\frac{L}{2}) = z\coth(z\frac{L}{2}).
    \end{equation*}
    However, the contour $\Gamma_{\varepsilon, R}$ this time with small $+\varepsilon > 0$ needs to be used to exclude the poles $\{\lambda_k\}$ of $\coth(z\frac{L}{2})$ and therefore of $\dm^a$ with 
    \begin{equation*}
        ik\pi = \frac{L}{2}\sqrt{\sigma_*^2+\lambda} \quad \text{for} \quad k\in \bbZ \quad \Leftrightarrow \quad \lambda_k = -\left(\frac{\pi}{L} 2k\right)^2 - \sigma_*^2<0 \quad \text{for} \quad k\in\bbN.
    \end{equation*}
    Thus, the result $Z^a \leq 2$ follows in $\{\Re(\lambda)>0\}$ excluding the imaginary line $\{\Re(\lambda) = 0\}$.  
\end{proof}

\begin{remark} \label{rmk:num_stable_sol}
    \begin{enumerate}[label={\arabic*}]
        \item ($\mu\leq 0$ gives $Z^{i|a} = 0$). For $\Re(\lambda)\geq 0$, it holds that $\Re(z^2) = \Re(\sigma_*^2 + \lambda) > 0$ with $c^i(\lambda) = 0$ when $\sigma_*=\lambda=0$, so that by \eqref{eq:ReIm_ci} also $\Re(c^{i|a}(\lambda)) > 0$ hold. Hence, 
        \begin{equation*}
            \mu \leq 0: \quad\Re(\dm^{i|a}) = \Re(\lambda) + \alpha - \mu \Re(c^{i|a}(\lambda)) \geq \alpha > 0,
        \end{equation*}
        and $\dm^{i|a}$ thus have no zero in the closed right half-plane $\{\Re(\lambda)\geq 0\}$. 
        \item (Natural bound $Z^{i|a} \leq 2$ cannot be improved). On the real ray $\{\lambda\in(-\infty, -\sigma_*^2)\}$, the Mittag-Leffler expansions of $c^{i|a}(\lambda)$ give $(c^{i|a})'>0$ and $(c^{i|a})''<0$, so that, for $\mu > 0$, the functions $\dm^{i|a}(\cdot) = \cdot + \alpha - \mu c^{i|a}(\cdot)$ are strictly convex there. A strictly convex function with $\lim_{\lambda\to \infty} \dm^{i|a}(\lambda) = \infty$ for $\lambda\in\bbR$, in turn, has at most two zeros on $(0, \infty)$, and, for suitable parameters for which $\dm^{i|a}(0)>0$ while $\min(\dm^{i|a})<0$ at a real $\lambda > 0$, exactly two positive real zeros occur, giving $Z^{i|a} = 2$. Example parameter sets for the linearized dynamics for which $\dm^i$ and $\dm^a$ each have exactly two roots with $\Re(\lambda)>0$ can be read off directly from the first row of Figure \ref{fig:Relambda_pos_i}.
        \item (Evans functions $\dm^{i|a}$). The functions $\dm^{i|a}$ are the characteristic Evans-type functions of the eigenvalue problem
        $V_0'' = (\sigma_*^2+\lambda)V_0$ of \eqref{eq:pertdiff} with a $\lambda$-dependent interface condition, the in-phase and anti-phase conditions, respectively, at the dynamic boundaries $V_0(0) = V_0^- \;\wedge \;  V_0(L) = V_0^+$ of \eqref{eq:pertbc}. The result of Theorem \ref{thm:num_stable_sol} says that the number of destabilizing eigenvalues, representing growth rates of eigenperturbations, of each $\dm^i$ and $\dm^a$ never exceeds 2.
    \end{enumerate}
\end{remark}

The nonlinearities in the boundary kinetics $\fint$, however, generically bend the steady-state and Hopf branches with respective symmetry, so that Theorem \ref{thm:num_stable_sol} only holds close to onset from the base branch (linearized stability). It would be interesting to study the fully nonlinear competition of the steady-state with the Hopf oscillation sharing its symmetry. It would further remain to analyze whether at a given point in parameter space the non-base symmetric steady-state solution is selected over the asymmetric steady-state solution and the in-phase Hopf oscillation is selected over the anti-phase Hopf oscillation and vice versa, for system \eqref{eq:sys_tocubic} with cubically expanded internal reaction kinetics $\fint$. For this, the magnitude of the real part of the \emph{nonlinear} growth rate (in contrast to the linear growth rate $\lambda$ of $\dm^i(\lambda) = 0$ and $\dm^a(\lambda) = 0$ that was regarded in Theorem \ref{thm:asymptexpansion_stab_steady-states} and of Figures \ref{fig:Relambda_pos_i} and \ref{fig:Relambda_pos_a}) to each coexisting stable solution $u_*+v$, namely $\Re(\lambda_{\text{nl}})>0$, would need to be analyzed. However, I am choosing to not go further along this path for the present study and instead go over to numerical asymptotics and illustrations in the next section. 

\section{Numerical asymptotics and homoclinic limits at finite amplitude} \label{sec:numericalasympt}

In this section, solutions of the boundary-value problem \eqref{eq:sys_tocubic} are explored. A special focus will be devoted to illustrate the results on the existence and location of the in-phase and anti-phase Hopf bifurcations of Theorems \ref{thm:leading-orderHopf}, on the existence and location of the symmetric and asymmetric (here antisymmetric due to the base equilibrium at $u_*\equiv 0$) steady-state bifurcation points of Remark \ref{rmk:steady-state-bif}, on the bifurcation branch shapes and sub- and supercriticality, and on the growth rates $\lambda$ with $\Re(\lambda)>0$ to each such solution in parameter space.

\paragraph{Bifurcation lines and spectra of linearized operators.}
Let us start with taking a look at the spectrum of the linearized boundary-integral function $D\vF$ that combines both the in-phase Hopf Fredholm operator $\cL^i$ of \eqref{eq:Fredholm_operatori} and the anti-phase Hopf Fredholm operator $\cL^a$ of \eqref{eq:Fredholm_operatora} when evaluated at their respective bifurcation points and onset frequencies. Expressions for the bifurcation points $\mu_*^i$ (in-phase Hopf), $\mu_*^a$ (anti-phase Hopf), $\mu_*^{s}$ (symmetric steady-state), and $\mu_*^{as}$ (antisymmetric steady-state) are readily available from Theorem \ref{thm:leading-orderHopf} and Remark \ref{rmk:steady-state-bif} and are plotted for their ranges of validity in Figures \ref{fig:Relambda_pos_i} and \ref{fig:Relambda_pos_a} for a large domain with moderate bulk degradation rate ($L=10$ and $\sigma=0.1$; left column) and with small degradation rate ($L=10$ and $\sigma=0.01$; center column) and for a moderate domain size with moderate degradation rate ($L=1$ and $\sigma=0.1$; right column). One can see that both the in- and anti-phase Hopf bifurcation lines are emerging from the respective symmetric and antisymmetric steady-state bifurcation lines. The reason for that is shown in the plots of $\Re(\lambda)$ for a fixed value for $\mu$ ($\mu = 1.4$ for both left and center column, $\mu = 8.7$ for the right column; each at the location of the yellow dot in the plot in the respective column of the first row) and varying $\alpha$ in the fourth row of Figures \ref{fig:Relambda_pos_i} and \ref{fig:Relambda_pos_a}. Confirmed by Theorem \ref{thm:num_stable_sol}, the Fredholm operators $\cL^i$ and $\cL^a$ both have maximally two eigenvalues $\lambda$ with $\Re(\lambda) > 0$ in the parameter range shown, and hence the corresponding Evans functions $\dm^i$ and $\dm^a$ have maximally two roots with non-negative real part that are precisely those $\lambda$. As $\alpha$ increases, two real linearized growth rates $\lambda$ (of either $\dm^i$ or $\dm^a$) collide and form a complex conjugate pair with increasing imaginary parts and decreasing real parts. Such a collision happens in the complex plane $\bbC$ below the real line if $\mu$ is smaller than the point where respectively the symmetric steady-state bifurcation line meets the in-phase Hopf bifurcation line and the antisymmetric steady-state bifurcation line meets the anti-phase Hopf bifurcation line. Furthermore, the second and third row of both figures show the biggest $\Re(\lambda)$ value if $\Re(\lambda) > 0$ and otherwise show zero at the corresponding point in the $(\alpha, \mu)$-plane.

Decreasing the bulk degradation rate $\sigma$ toward zero at fixed domain size $L$ most prominently amounts to decreasing the slope of the symmetric steady-state bifurcation line toward zero. At $\sigma = 0$, the antisymmetric steady-state bifurcation line vanishes suddenly (see first row of Figures \ref{fig:Relambda_pos_i} and \ref{fig:Relambda_pos_a}). Increasing $L$ for fixed $\sigma >0$ again decreases the slope of the symmetric steady-state bifurcation line toward zero but, more prominently increases the slope of the antisymmetric steady-state bifurcation line and the in-phase and anti-phase Hopf bifurcation lines toward infinity ($90^\circ$). At the same time, the intersection point of the steady-state bifurcation lines with the Hopf bifurcation lines increases in $\mu$-direction and, for the antisymmetric case, also in $\alpha$-direction. Hence, it appears that the base state $u_*\equiv 0$ is linearly unstable for a bigger fraction of the $(\alpha, \mu) \in \bbR_+^2$ parameter plane as $L\to \infty$ for fixed $\sigma>0$. At this point it is important to note that all intuition which Figures \ref{fig:Relambda_pos_i} and \ref{fig:Relambda_pos_a} give only account for linearized stability. Any nonlinear term appearing in $\vF$ where the linearizations $\cL^i$ and $\cL^a$ stem from, likely bends the respective emerging bifurcation branches in a sub- or supercritical way. I will show the shape of the nonlinear bifurcation branches in Figures \ref{fig:in-phase_Hopfbranches} and \ref{fig:anti-phase_Hopfbranches} below.

Regarding the further computational details behind the bifurcation line plots in Figures \ref{fig:Relambda_pos_i} and \ref{fig:Relambda_pos_a}, the numbers $Z^{i|a}$ of eigenvalues $\lambda_{i|a}$ of respectively $\cL^{i|a}$ with $\Re(\lambda)>0$ provided in the legend of the plots in the first row of Figures \ref{fig:Relambda_pos_i} and \ref{fig:Relambda_pos_a} at each scattered point, has been computed numerically with a discretized version of the argument principle of the proof of Theorem \ref{thm:num_stable_sol},
\begin{equation*}
    Z^{i|a} = \lim_{\varepsilon\downarrow 0}\lim_{R\uparrow \infty}\frac{1}{2\pi i} \int_{\Gamma _{\varepsilon, R}}\frac{\partial_\lambda \dm^{i|a}(\lambda; \mu, \sigma)}{\dm^{i|a}(\lambda; \mu, \sigma)} \; d\lambda = \lim_{\varepsilon\downarrow 0}\lim_{R\uparrow \infty} \Delta_{{\Gamma_{\varepsilon, R}}} \mathrm{arg}(\dm^{i}(\lambda; \mu, \sigma_*)
\end{equation*}
using the complex contour of the $\varepsilon$-shifted closed half circle $\Gamma_{\varepsilon, R} := \{z = \varepsilon+Re^{i\varphi} \;|\; \varphi\in [-\pi, \pi]\} \cup \{z = \varepsilon+iR-\tau i2R \;|\; \tau\in(0, 1)\}$ parametrized in counterclockwise direction and the limit $\Gamma = \lim_{\varepsilon\downarrow 0}\lim_{R\uparrow\infty} \Gamma_{\varepsilon, R}$. The heatmaps in the second and third rows of Figures \ref{fig:Relambda_pos_i} and \ref{fig:Relambda_pos_a} where computed by finding the biggest $\Re(\lambda) >0$ using the Newton algorithm on $\dm^i$ and $\dm_a$ in the complex plane for $\lambda$. If there is no such $\Re(\lambda) >0$, the value zero has been plotted. For the finely followed paths of the two roots of either $\dm^i$ or $\dm^a$ shown in the fourth row of respectively Figure \ref{fig:Relambda_pos_i} and Figure \ref{fig:Relambda_pos_a}, again the Newton algorithm has been used directly on $\dm^i$ and $\dm^a$ together with the result of Theorem \ref{thm:num_stable_sol} that there cannot be more than two roots.

\begin{figure}
    \centering
    \begin{subfigure}{0.28\textwidth}
    \includegraphics[width=1\textwidth]{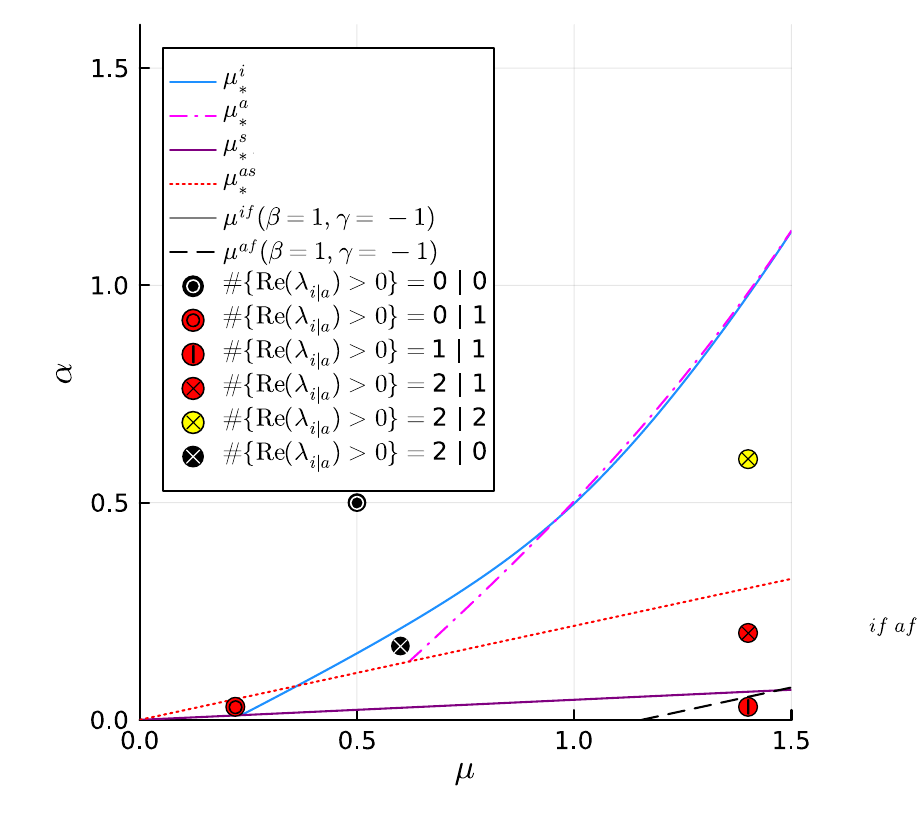}
    \end{subfigure}
    \begin{subfigure}{0.28\textwidth}
    \includegraphics[width=1\textwidth]{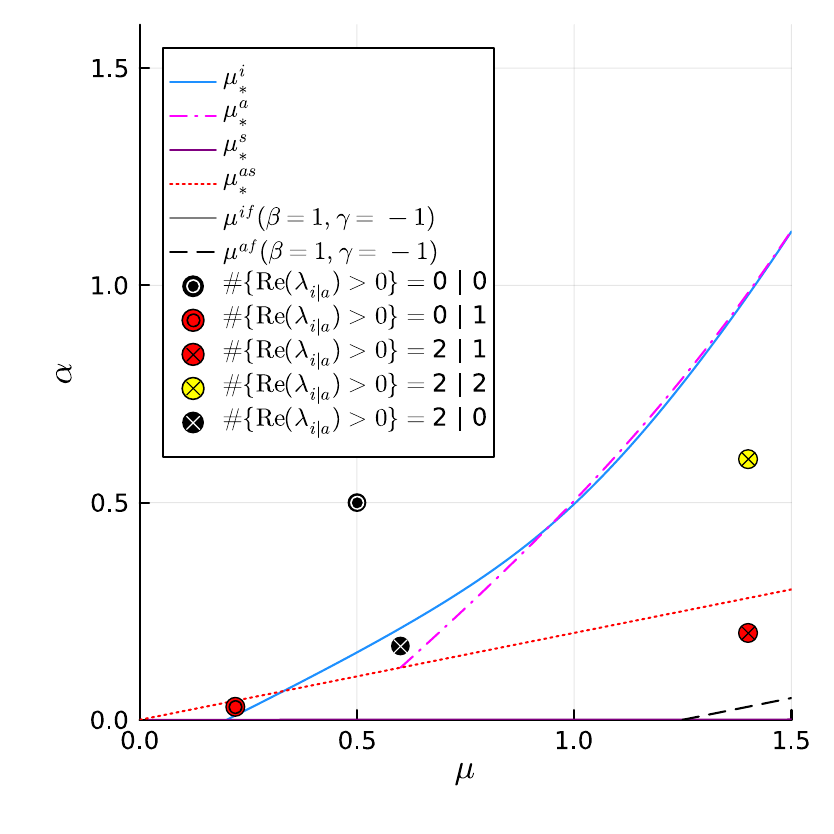}
    \end{subfigure}
    \begin{subfigure}{0.28\textwidth}
    \includegraphics[width=1\textwidth]{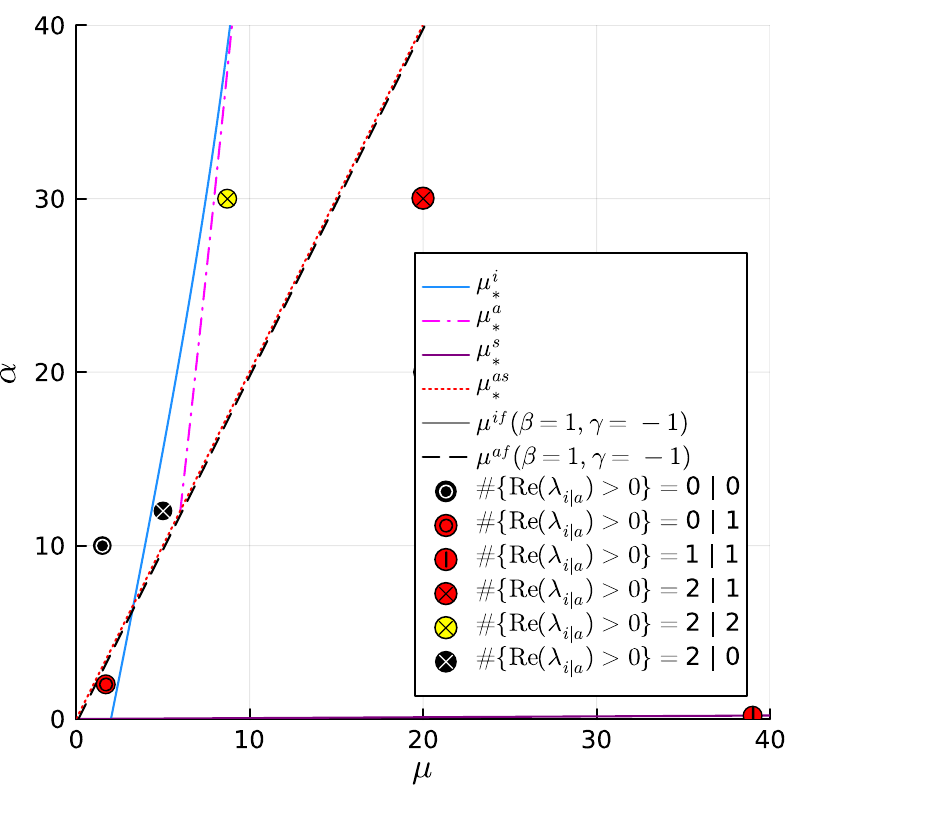}
    \end{subfigure}
    \begin{subfigure}{0.28\textwidth}
    \includegraphics[width=1\textwidth]{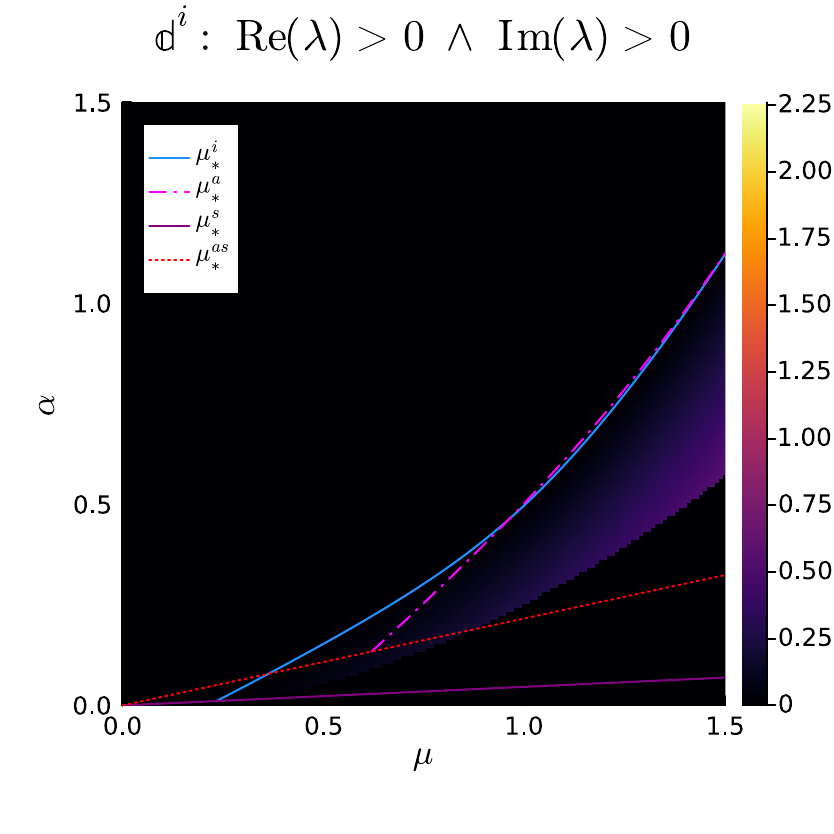}
    \end{subfigure}
    \begin{subfigure}{0.28\textwidth}
    \includegraphics[width=1\textwidth]{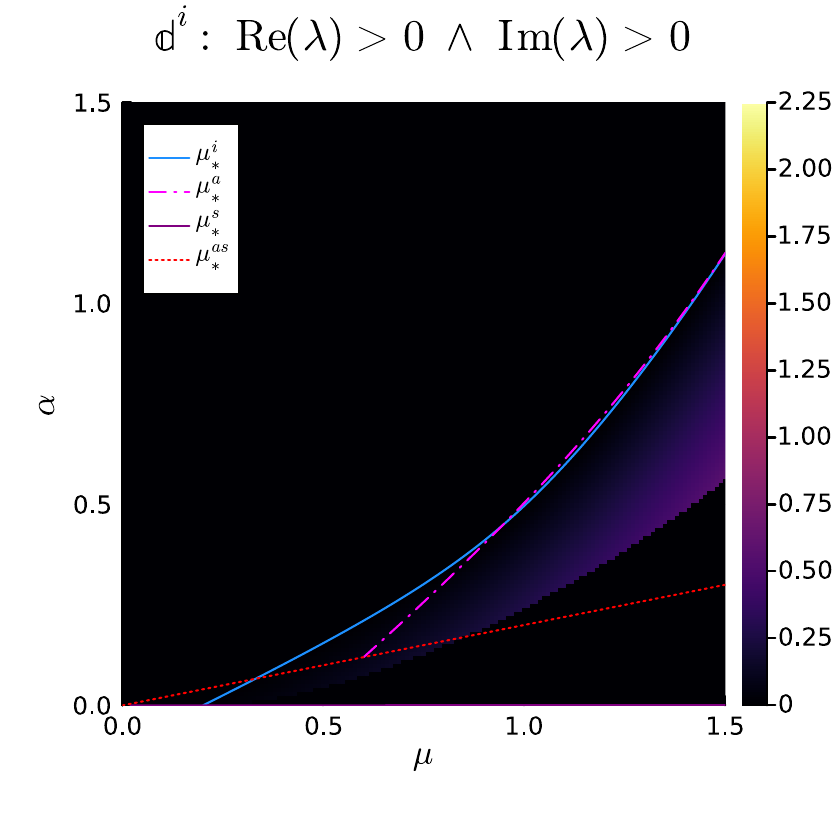}
    \end{subfigure}
    \begin{subfigure}{0.28\textwidth}
    \includegraphics[width=1\textwidth]{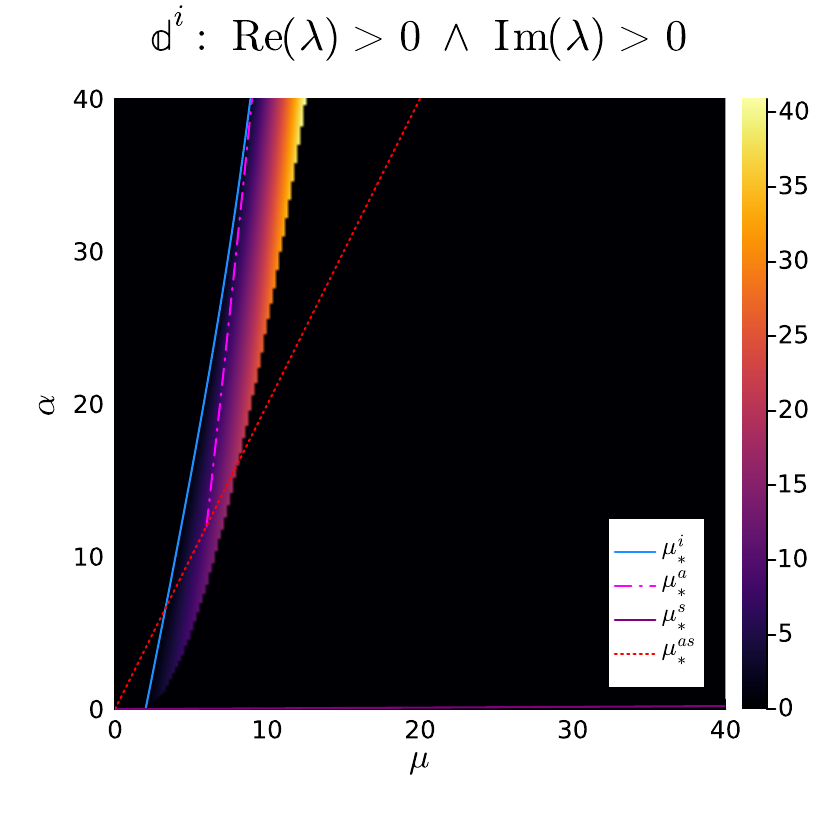}
    \end{subfigure}
    \begin{subfigure}{0.28\textwidth}
    \includegraphics[width=1\textwidth]{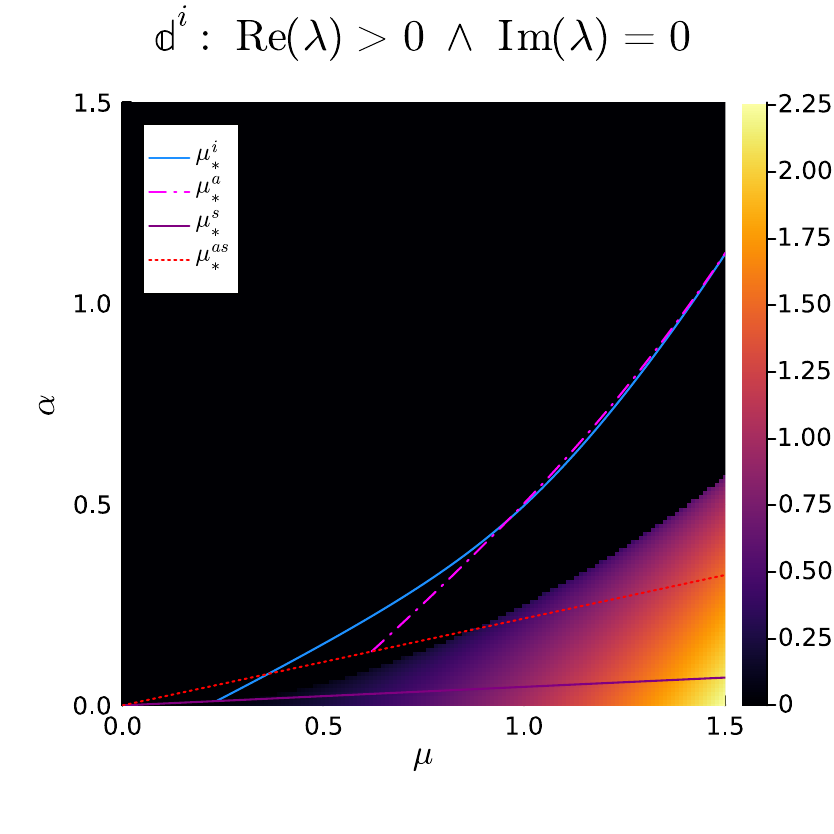}
    \end{subfigure}
    \begin{subfigure}{0.28\textwidth}
    \includegraphics[width=1\textwidth]{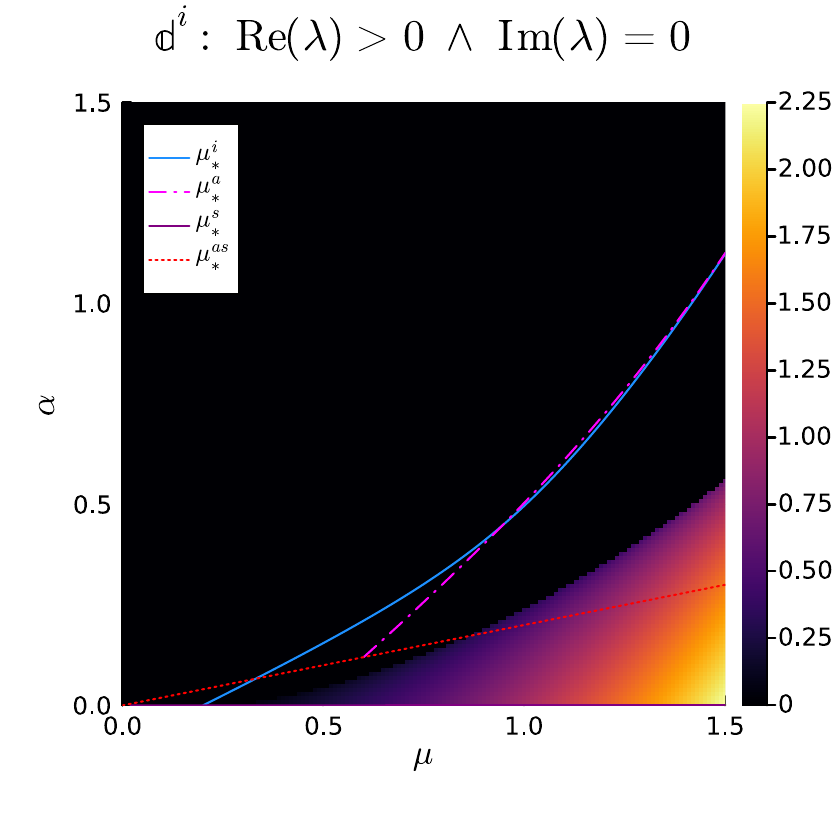}
    \end{subfigure}
    \begin{subfigure}{0.28\textwidth}
    \includegraphics[width=1\textwidth]{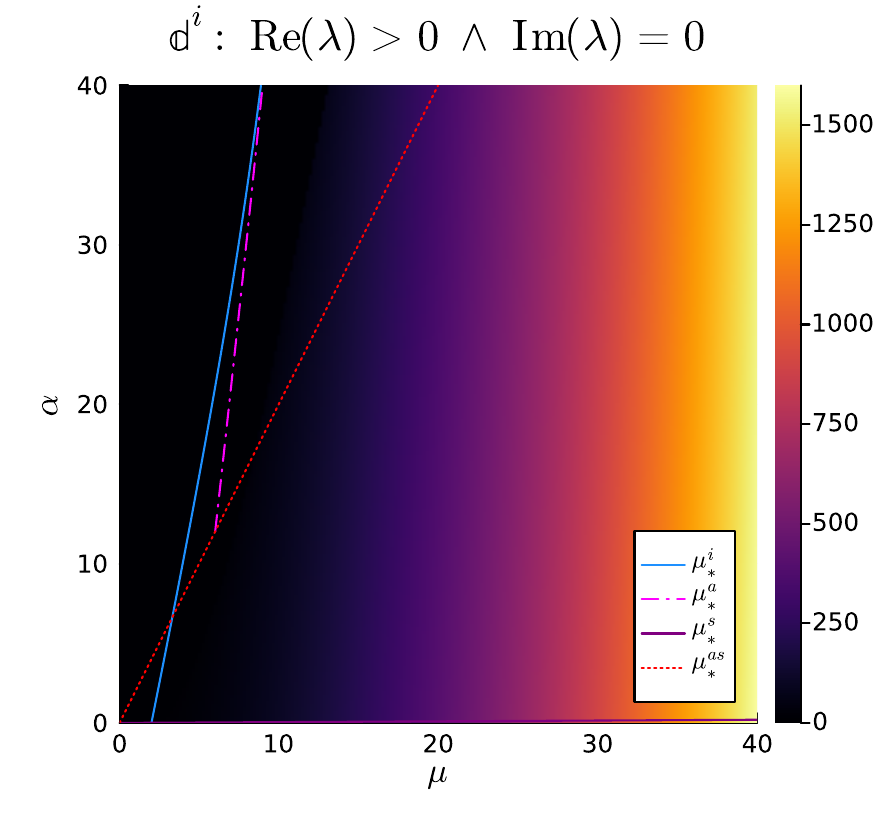}
    \end{subfigure}
    \begin{subfigure}{0.28\textwidth}
    \includegraphics[width=1\textwidth]{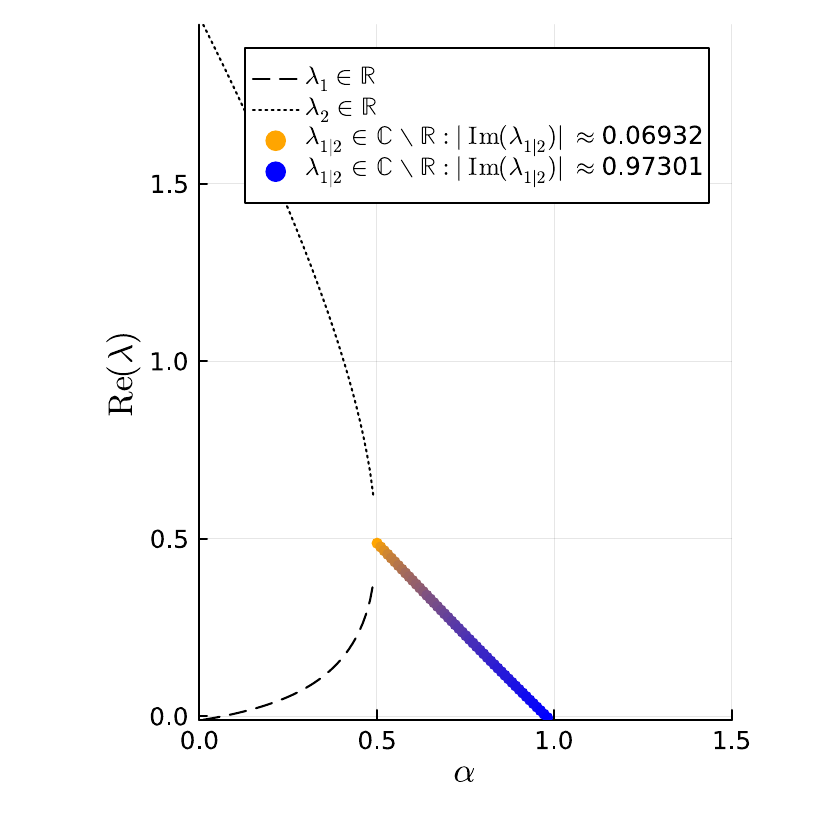}
    \end{subfigure}
    \begin{subfigure}{0.28\textwidth}
    \includegraphics[width=1\textwidth]{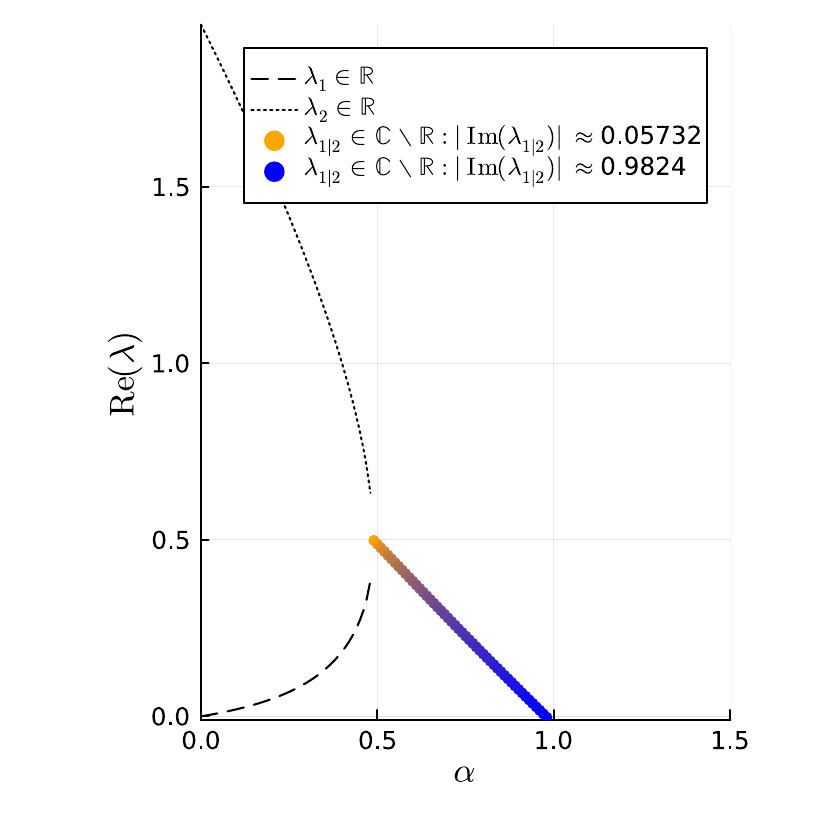}
    \end{subfigure}
    \begin{subfigure}{0.28\textwidth}
    \includegraphics[width=1\textwidth]{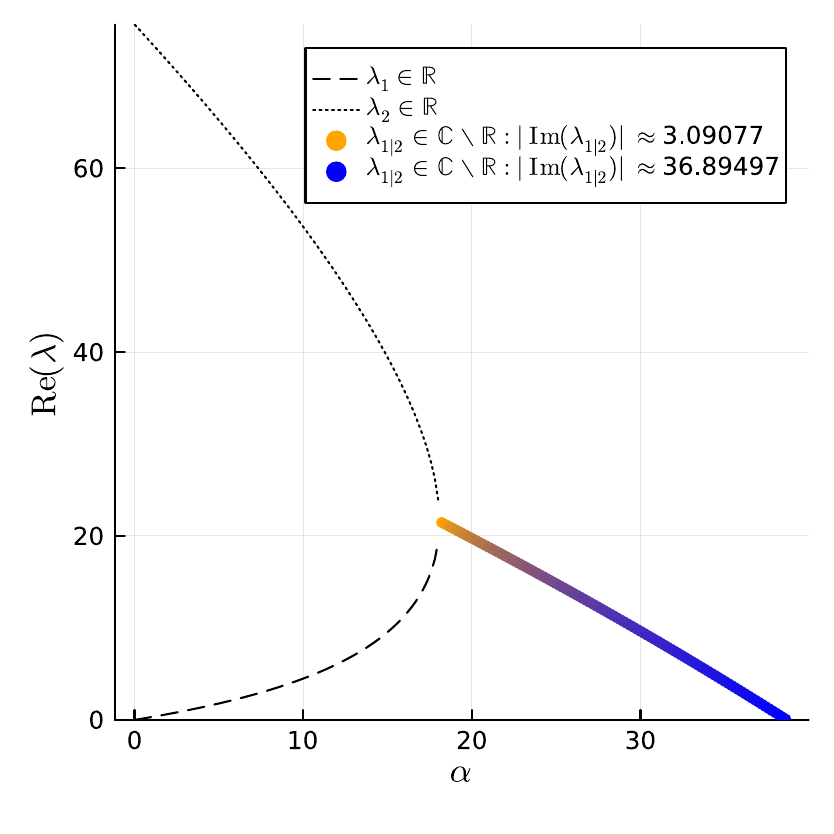}
    \end{subfigure}
    \caption{Steady-state and Hopf bifurcation lines branching off $u_*\equiv 0$, numbers $Z^{i|a} = \#\{\lambda\,|\,\Re(\lambda_{i|a})>0\}$ of the linearizations and Evans functions, $\dm^{i|a}$, respectively, and the maximal $\Re(\lambda) > 0$. Parameter values: large domain size $L=10$ and moderate degradation $\sigma=0.1$ for the left column, large domain size $L=10$ and small degradation $\sigma=0.01$ for the center column, moderate domain size $L=1$ and moderate degradation $\sigma=0.1$ for the right column. First row: in-phase Hopf bifurcation line $\mu_*^i$ and anti-phase Hopf bifurcation line $\mu_*^a$ of Theorem \ref{thm:leading-orderHopf}, symmetric steady-state bifurcation line $\mu_*^s$ and antisymmetric steady-state bifurcation line $\mu_*^{as}$ of Remark \ref{rmk:steady-state-bif}, and examples of \emph{nonlinear} ($\beta=1, \gamma=-1$) symmetric steady-state fold bifurcation line $\mu^{if}$ and antisymmetric steady-state fold bifurcation line $\mu^{af}$ of Corollary \ref{cor:fold}. Second row: Maximal $\Re(\lambda)>0$ for in-phase growth rate $\lambda$ with $\dm^i(\lambda) = 0$ and $\Im(\lambda)>0$. Third row: Maximal $\Re(\lambda)>0$ for symmetric steady-state growth rate $\lambda$ with $\dm^i(\lambda) = 0$ and $\Im(\lambda)=0$.}
    \label{fig:Relambda_pos_i}
\end{figure}

\begin{figure}[!ht]
\ContinuedFloat
\centering
\caption[]{Fourth row: varying $\alpha \in (0, 1.5]$ ($\alpha \in (0, 38]$ for right column) while fixing $\mu=1.4$ ($\mu=8.7$ for right column) in a vertical slice of the first row's bifurcation line plot, one can observe the transition to linearized growing in-phase oscillatory behavior where the two symmetric steady-state growth rates $\lambda_i\in\bbR$ merge to create two complex-conjugate in-phase Hopf growth rates $\lambda_i\in\bbR$. The magnitude of $|\Im(\lambda)|$ can also be read off approximately through the color-coding, and the in-phase Hopf bifurcation point can be observed.}
\end{figure}

\begin{figure}
    \centering
    \begin{subfigure}{0.28\textwidth}
    \includegraphics[width=1\textwidth]{figures/mustarplot_L10.0_sig0.1.pdf}
    \end{subfigure}
    \begin{subfigure}{0.28\textwidth}
    \includegraphics[width=1\textwidth]{figures/mustarplot_L10.0_sig0.01.pdf}
    \end{subfigure}
    \begin{subfigure}{0.28\textwidth}
    \includegraphics[width=1\textwidth]{figures/mustarplot_L1.0_sig0.1.pdf}
    \end{subfigure}
    \begin{subfigure}{0.28\textwidth}
    \includegraphics[width=1\textwidth]{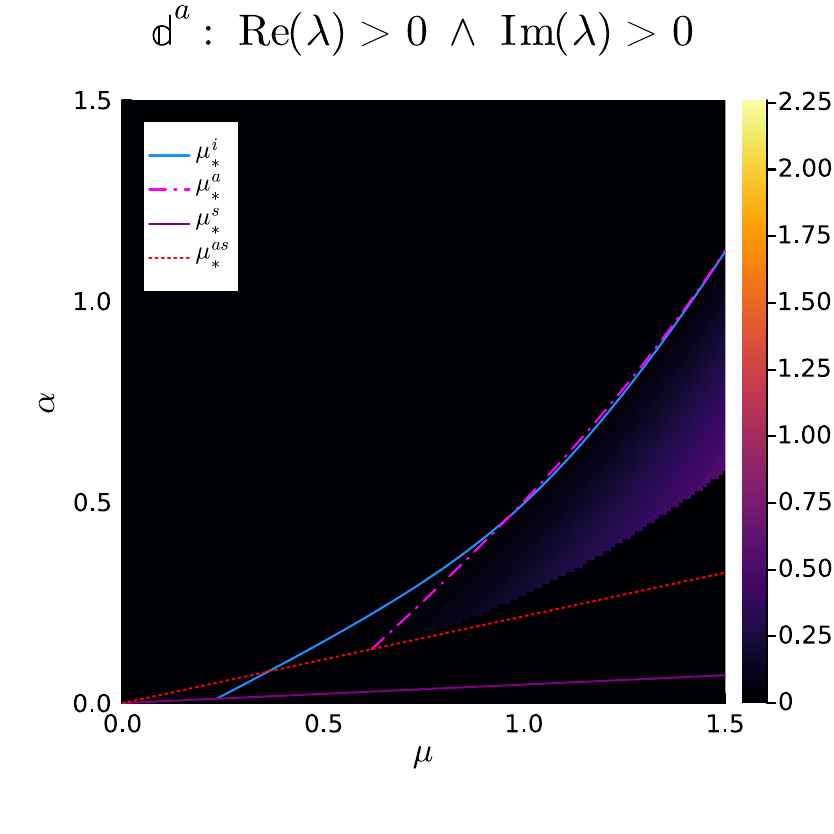}
    \end{subfigure}
    \begin{subfigure}{0.28\textwidth}
    \includegraphics[width=1\textwidth]{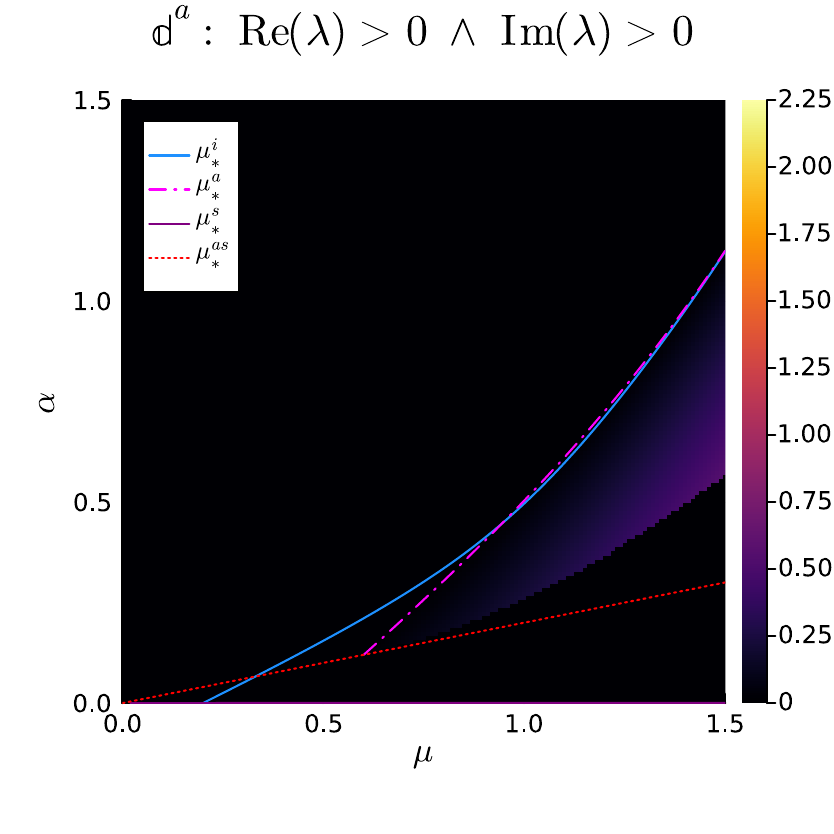}
    \end{subfigure}
    \begin{subfigure}{0.28\textwidth}
    \includegraphics[width=1\textwidth]{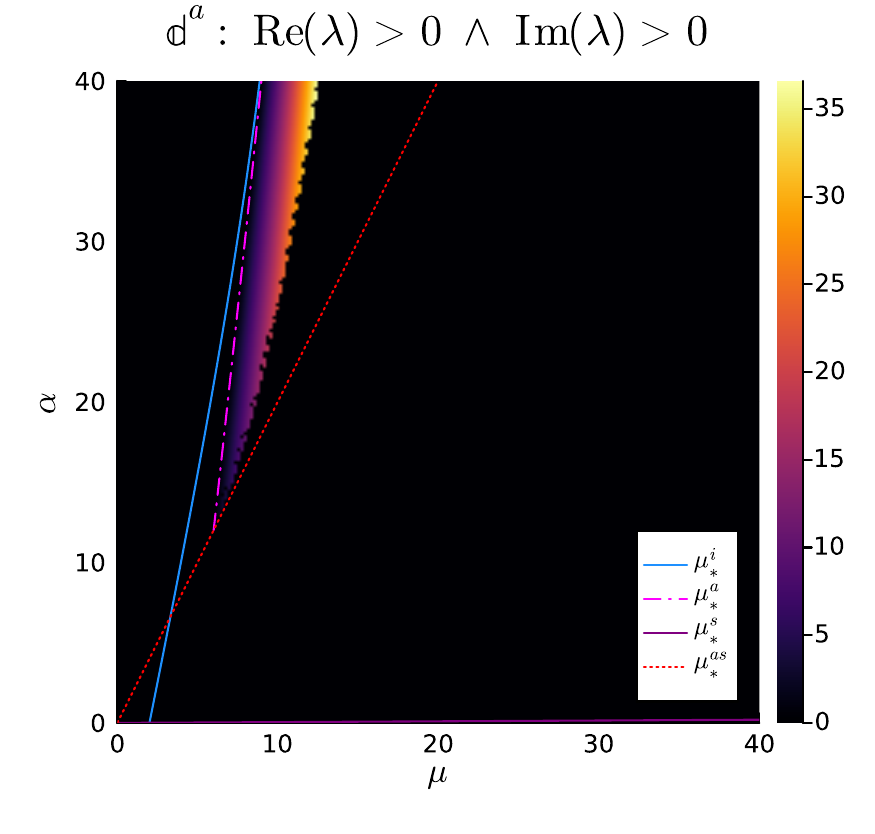}
    \end{subfigure}
    \begin{subfigure}{0.28\textwidth}
    \includegraphics[width=1\textwidth]{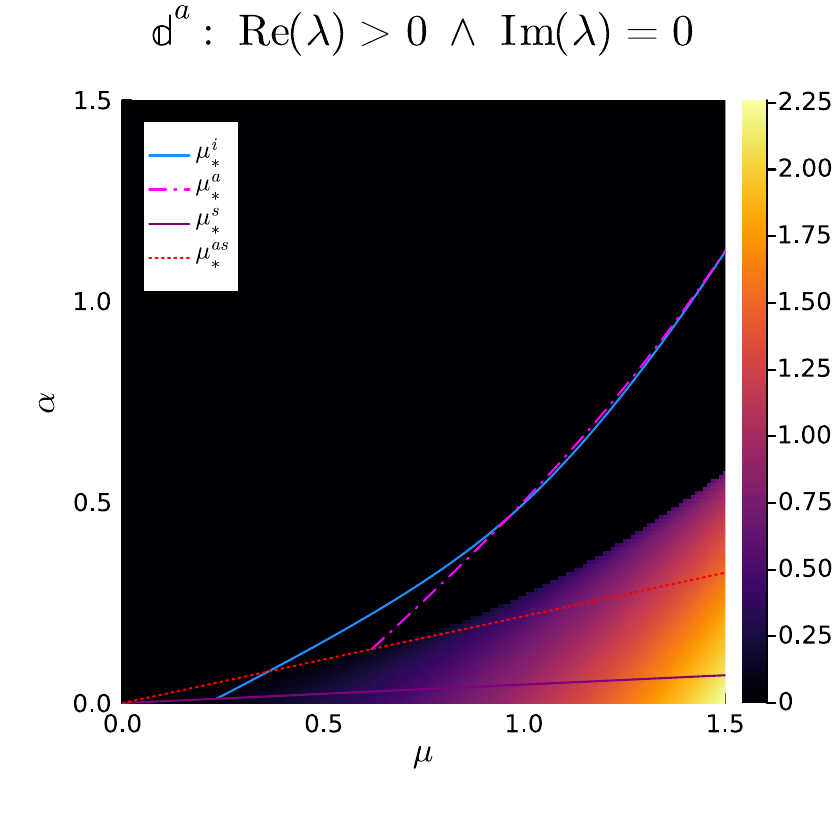}
    \end{subfigure}
    \begin{subfigure}{0.28\textwidth}
    \includegraphics[width=1\textwidth]{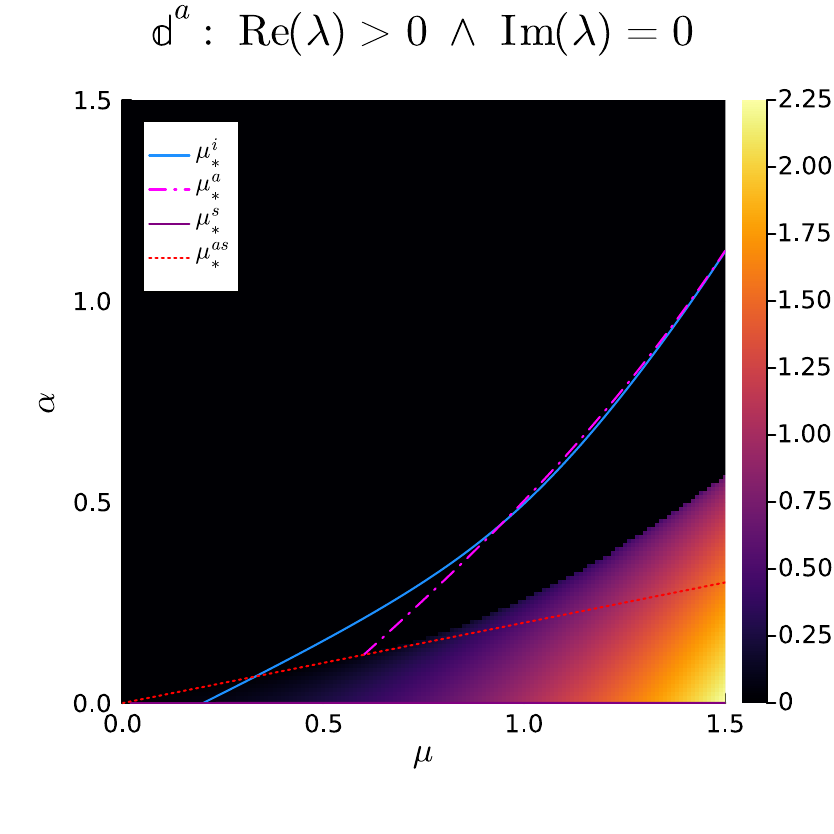}
    \end{subfigure}
    \begin{subfigure}{0.28\textwidth}
    \includegraphics[width=1\textwidth]{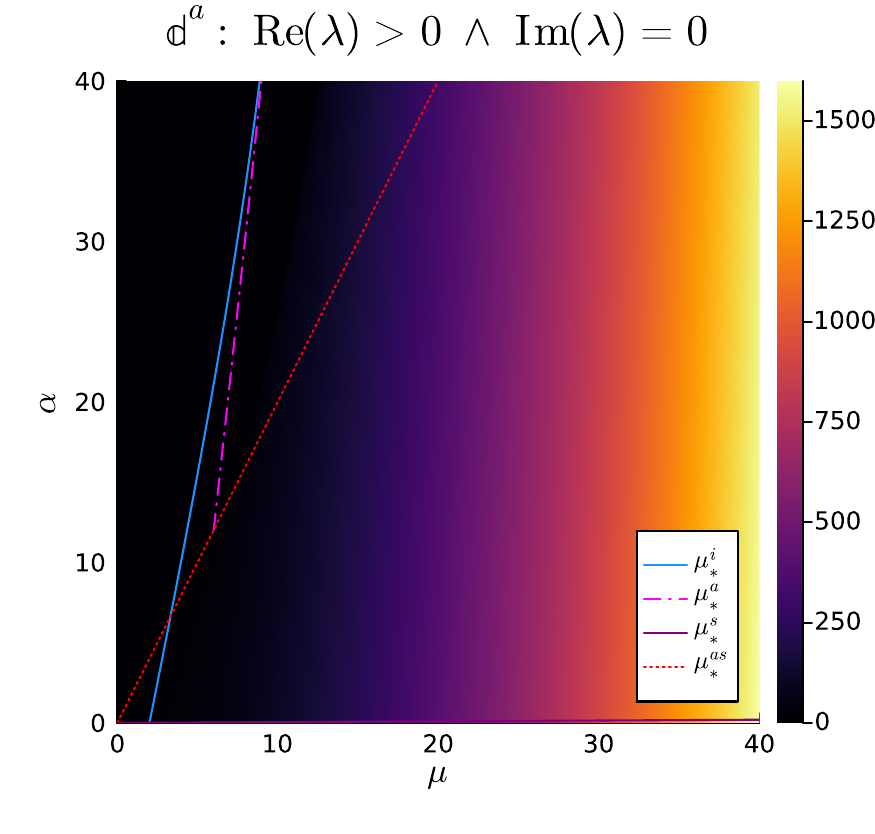}
    \end{subfigure}
    \begin{subfigure}{0.28\textwidth}
    \includegraphics[width=1\textwidth]{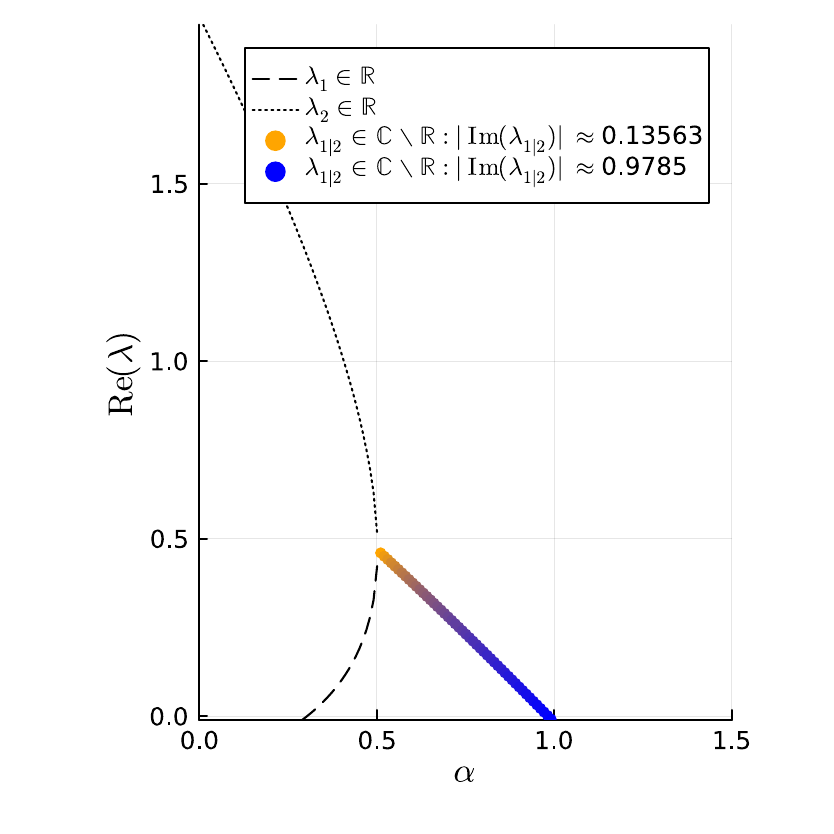}
    \end{subfigure}
    \begin{subfigure}{0.28\textwidth}
    \includegraphics[width=1\textwidth]{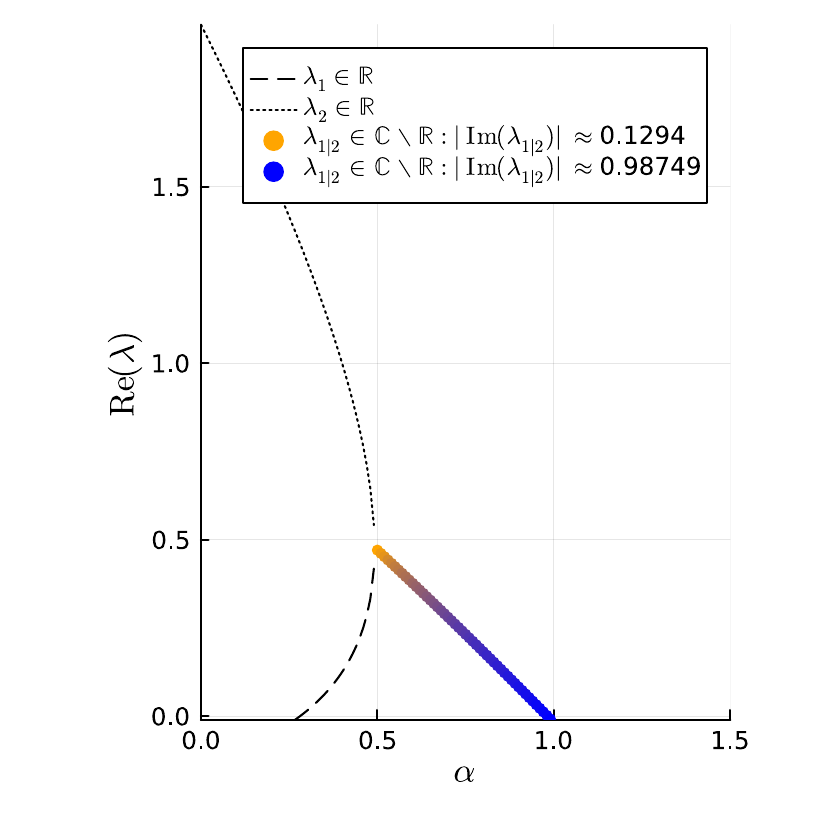}
    \end{subfigure}
    \begin{subfigure}{0.28\textwidth}
    \includegraphics[width=1\textwidth]{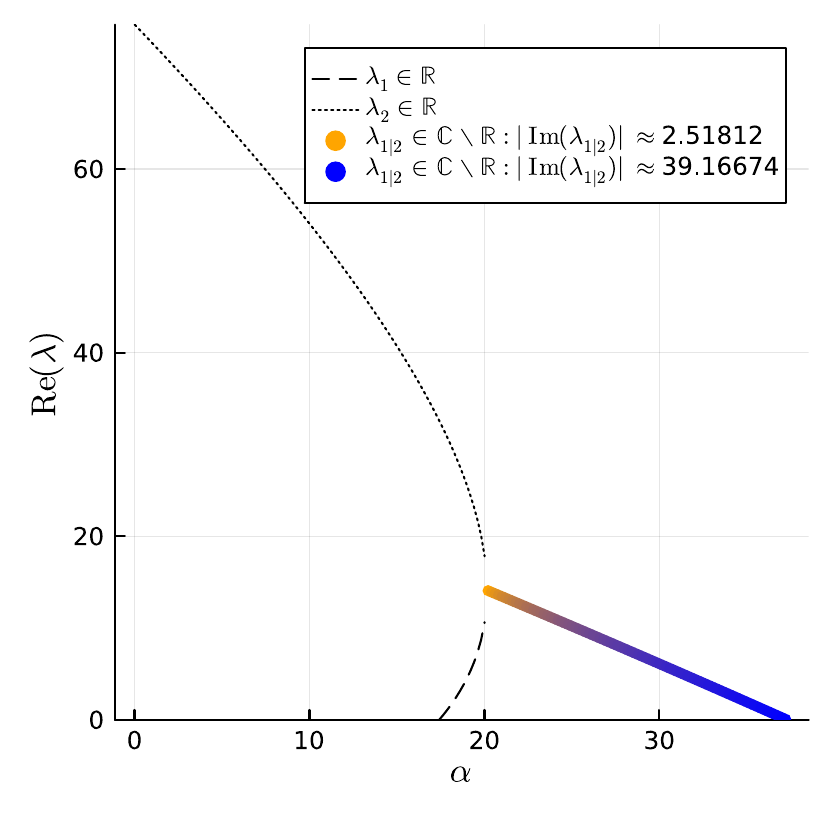}
    \end{subfigure}
    \caption{First row: Same as for Figure \ref{fig:Relambda_pos_i} for illustrative convenience. Second row: Maximal $\Re(\lambda)>0$ for anti-phase growth rate $\lambda$ with $\dm^a(\lambda) = 0$ and $\Im(\lambda)>0$. Third row: Maximal $\Re(\lambda)>0$ for antisymmetric steady-state growth rate $\lambda$ with $\dm^a(\lambda) = 0$ and $\Im(\lambda)=0$. Fourth row: varying $\alpha \in (0, 1.5]$ ($\alpha \in (0, 38]$ for right column) while fixing $\mu=1.4$ ($\mu=8.7$ for right column) in a vertical slice of the first row's bifurcation line plot, one can observe the transition to linearized growing anti-phase oscillatory behavior where the two antisymmetric steady-state growth rates $\lambda_a\in\bbR$ merge to create two complex-conjugate anti-phase Hopf growth rates $\lambda_a\in\bbR$. The magnitude of $|\Im(\lambda)|$ can also be read off approximately through the color-coding, and the anti-phase Hopf bifurcation can be observed.}
    \label{fig:Relambda_pos_a}
\end{figure}

\paragraph{Bending of bifurcation branches due to nonlinearities.}
Let us now turn to the effect the quadratic and cubic nonlinearities in $\fint$ have on the shape of the bifurcation branches. Figure \ref{fig:critcurve} shows the critical curve $\gamma_\text{crit}$ marking the transition of sub- (initially unstable) to supercriticality (initially stable) of Theorem \ref{thm:stabilityofbranches} for the three domain size $L$ and bulk degradation rate $\sigma$ pairs $(10, 0.1)$ (large domain, moderate degradation rate), $(10, 0.01)$ (large domain, small degradation rate), and $(1, 0.1)$ (moderate domain size, moderate degradation rate). While decreasing $\sigma\downarrow 0$ seems to give supercriticality for general $\gamma < 0$, decreasing the domain size $L\downarrow 0$ seems to narrow the parameter set of supercriticality until it vanishes in the $(\beta, \gamma)$-plane.

\begin{figure}
    \centering
    \begin{subfigure}{0.32\textwidth}
    \includegraphics[width=1\textwidth]{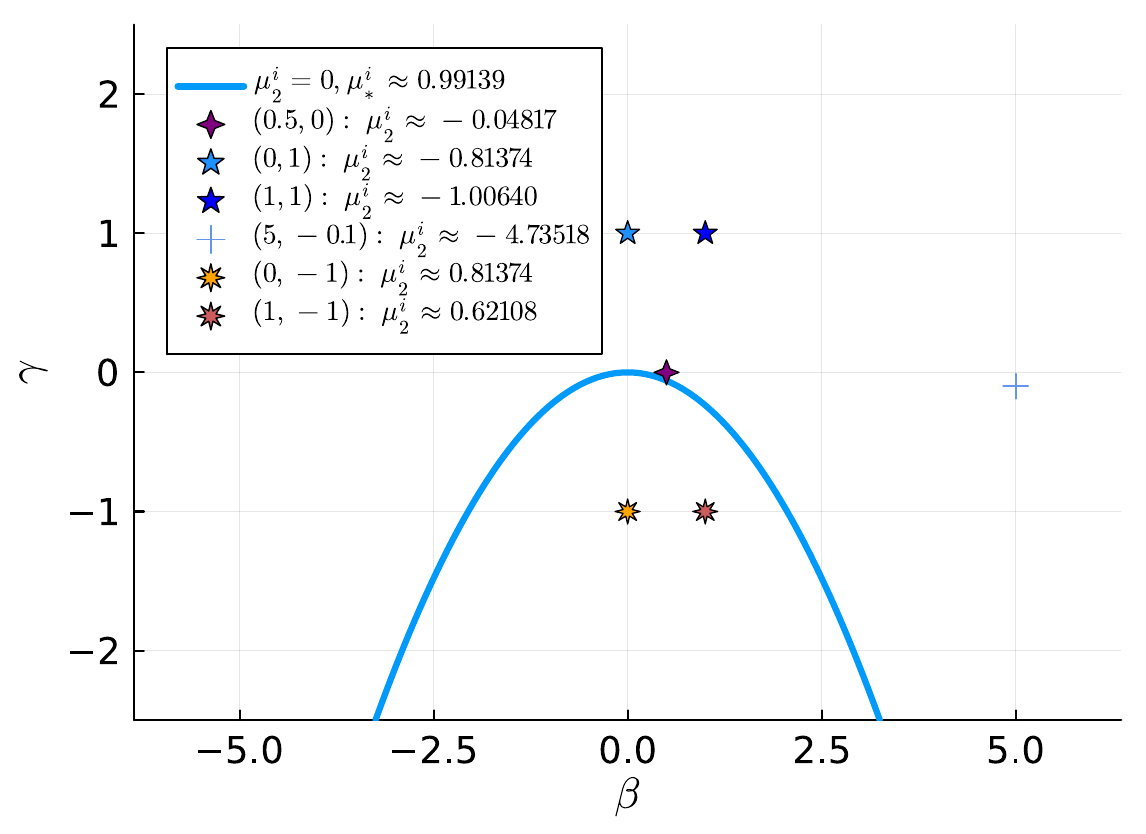}
    \end{subfigure}
    \begin{subfigure}{0.32\textwidth}
    \includegraphics[width=1\textwidth]{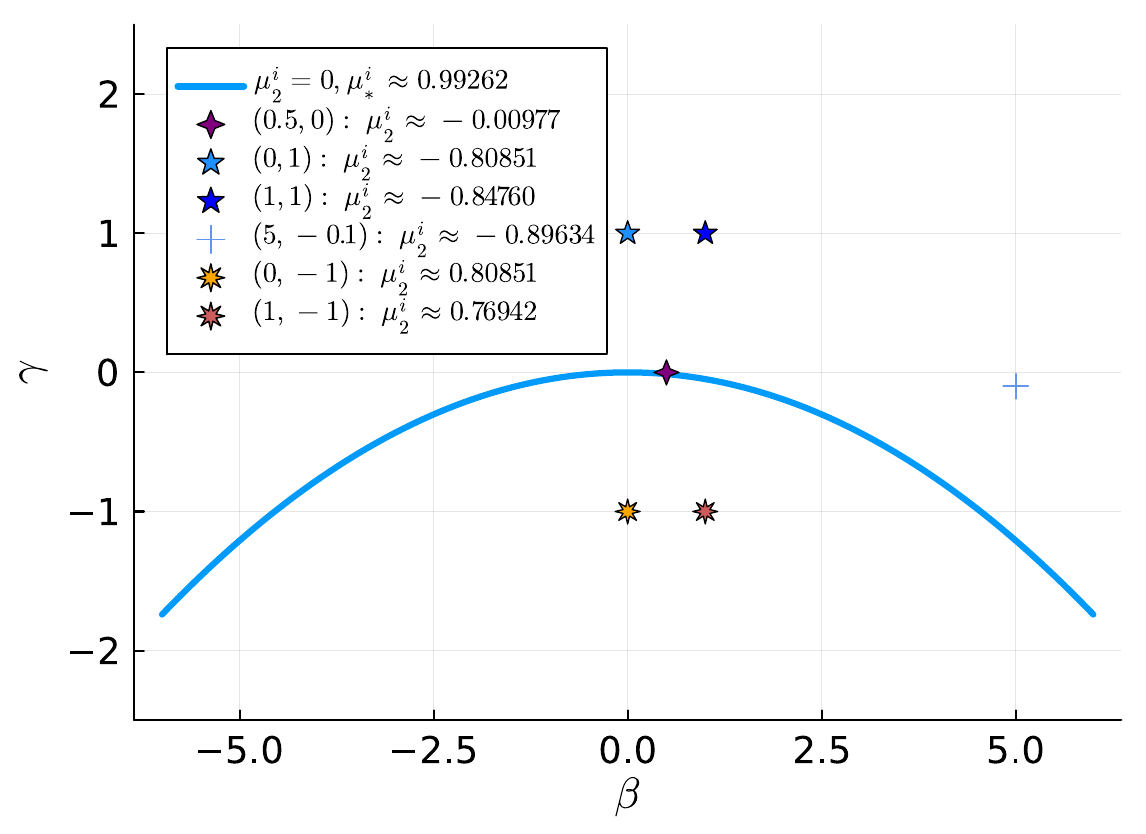}
    \end{subfigure}
    \begin{subfigure}{0.32\textwidth}
    \includegraphics[width=1\textwidth]{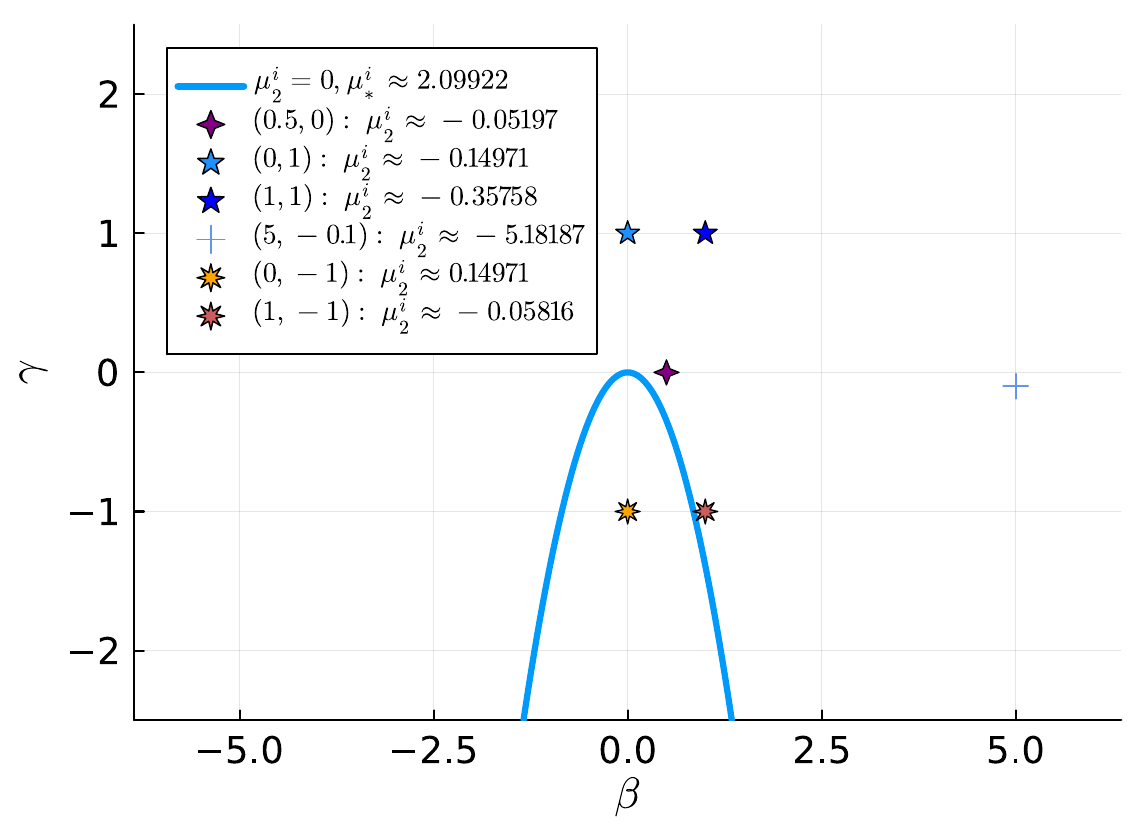}
    \end{subfigure}
    \begin{subfigure}{0.32\textwidth}
    \includegraphics[width=1\textwidth]{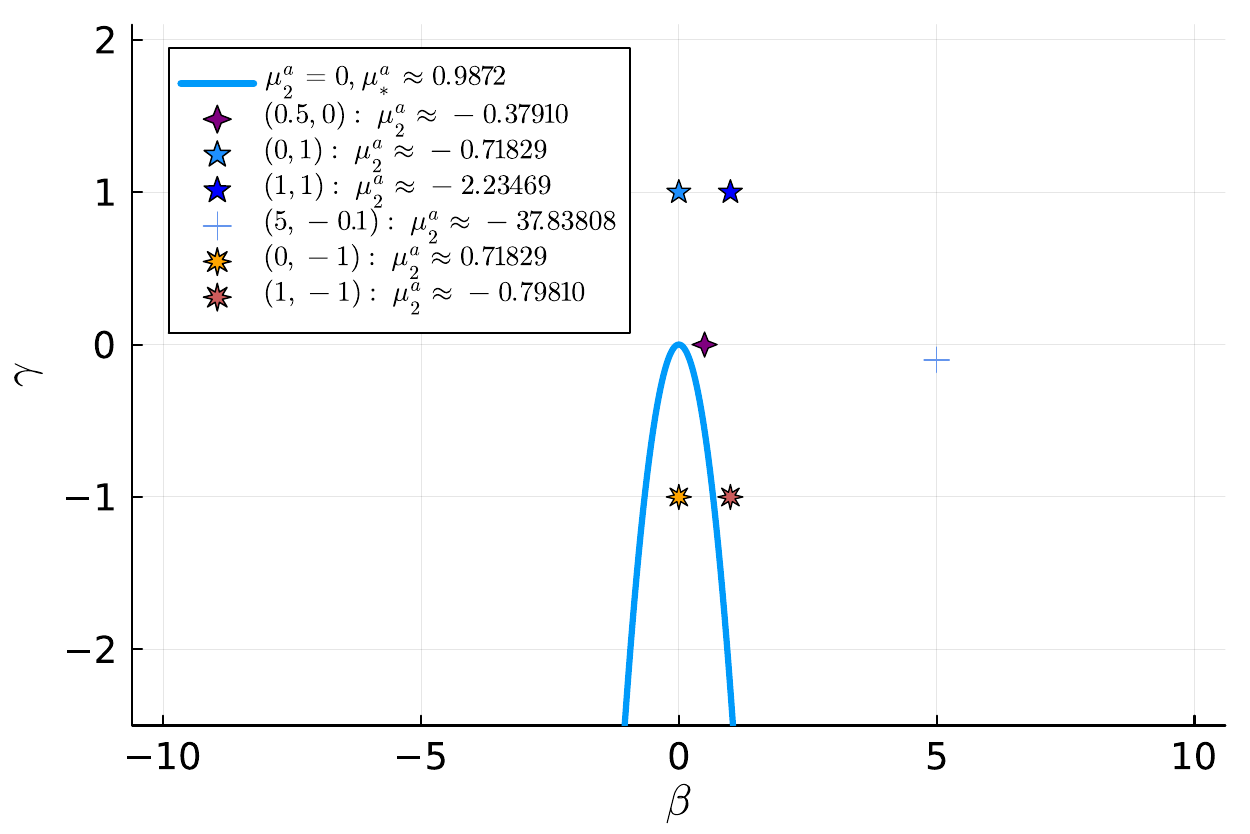}
    \end{subfigure}
    \begin{subfigure}{0.32\textwidth}
    \includegraphics[width=1\textwidth]{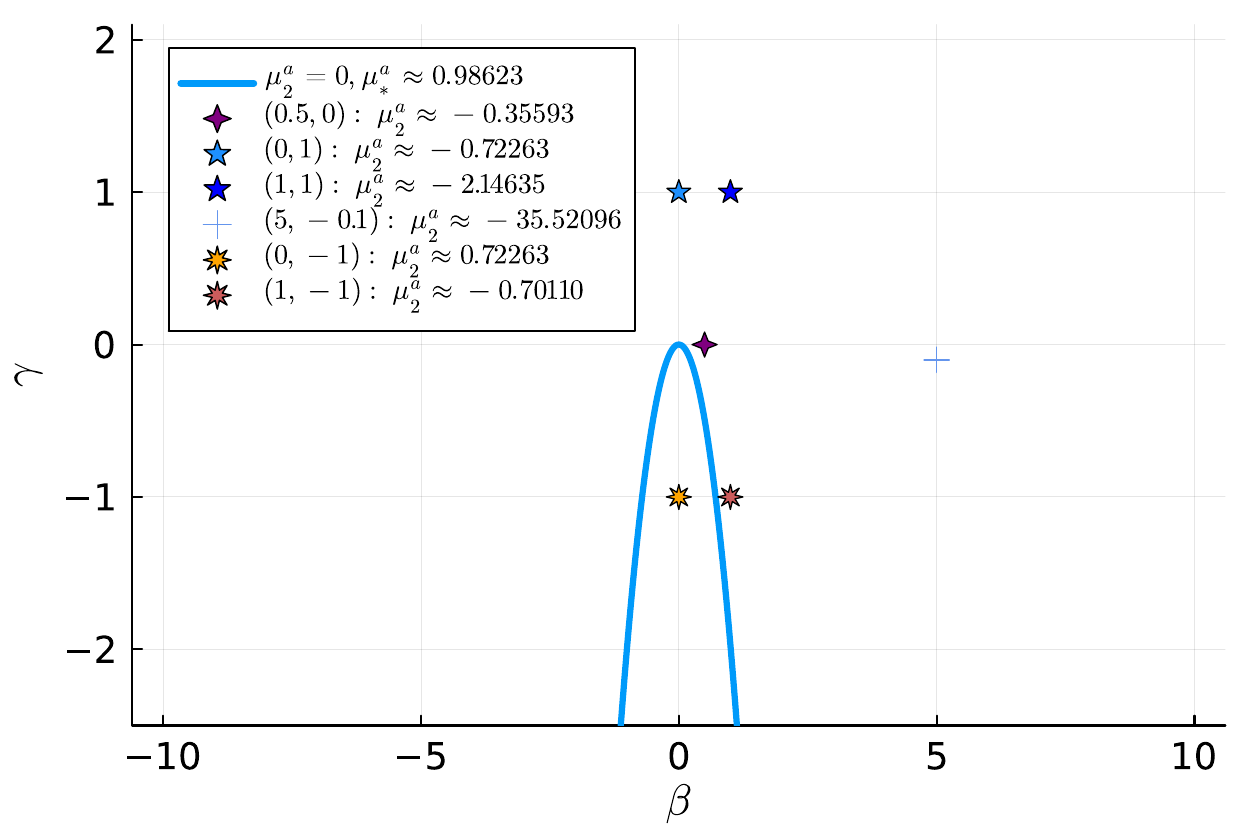}
    \end{subfigure}
    \begin{subfigure}{0.32\textwidth}
    \includegraphics[width=1\textwidth]{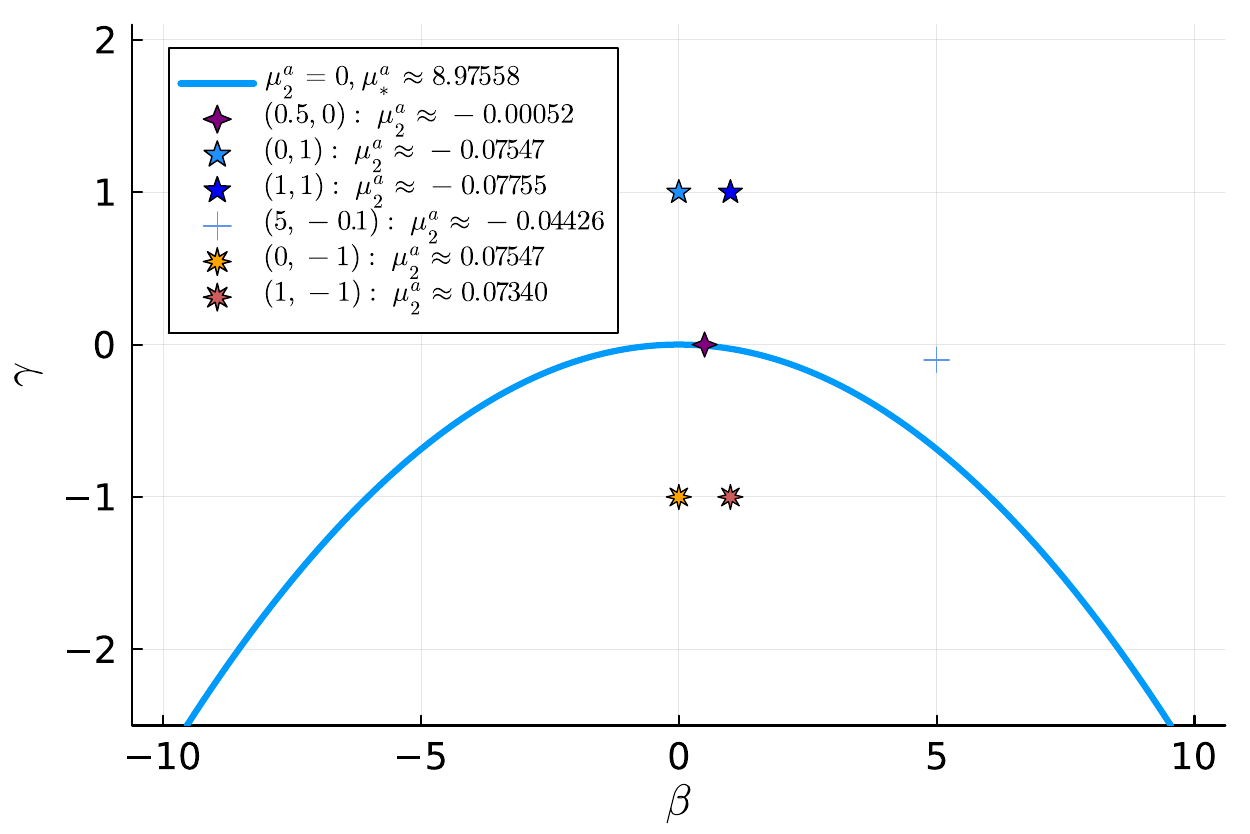}
    \end{subfigure}
    \caption{Critical curve $\gamma_{\text{crit}}$ where the sub-critical branches turn supercritical of Theorem \ref{thm:stabilityofbranches}, all with $\alpha = 1$. Below the quadratic curve, the branches are stable at bifurcation onset (supercritical), whereas they are unstable at bifurcation onset above the curve (subcritical). Left column: large domain and moderate degradation rate with $(L, \sigma) = (10, 0.1)$, center column: large domain and small degradation rate with$(L, \sigma) = (10, 0.01)$, right column: moderate domain size and moderate degradation rate with $(L, \sigma) = (1, 0.1)$. First row: for in-phase Hopf bifurcation, second row: for symmetric steady-state bifurcation, third row: for anti-phase Hopf bifurcation, fourth row: for asymmetric (for the chosen $\fint$ antisymmetric) steady-state bifurcation. For $\sigma\downarrow 0$, all parameter sets with $\gamma<0$ seem to exhibit supercritical bifurcation branches, and, for $L\downarrow 0$, the parameter range for supercriticality narrows until it vanishes.}
    \label{fig:critcurve}
\end{figure}

For the case of the bulk domain with large $L=10$ and moderate $\sigma=0.1$, the geometry of the bifurcation branches along with the shape of the nonlinear oscillatory or nontrivial steady-state solution far from bifurcation onset can be observed in Figure \ref{fig:in-phase_Hopfbranches} for the in-phase Hopf bifurcation branch, at each scattered point of the plots in the first row of Figure \ref{fig:critcurve}. These branches should be connected with the bifurcation line plots in Figure \ref{fig:Relambda_pos_i}, especially as the nonlinear oscillation frequency $\omega_{\text{nl}}^i$ seems to converge to zero for all branches except for the middle one in the bottom row for $(\beta, \gamma) = (0, -1)$ that is supercritical. 
The bifurcation branches for in- and anti-phase oscillations are highly similar at each respective point in the parameter plane. The final oscillation profiles, however, have a distinct swing for $\beta\neq 0$ at the end of the anti-phase Hopf branches that is not there at the end of the in-phase Hopf branches.

\begin{figure}
    \centering
    \begin{subfigure}{0.32\textwidth}
    \includegraphics[width=1\textwidth]{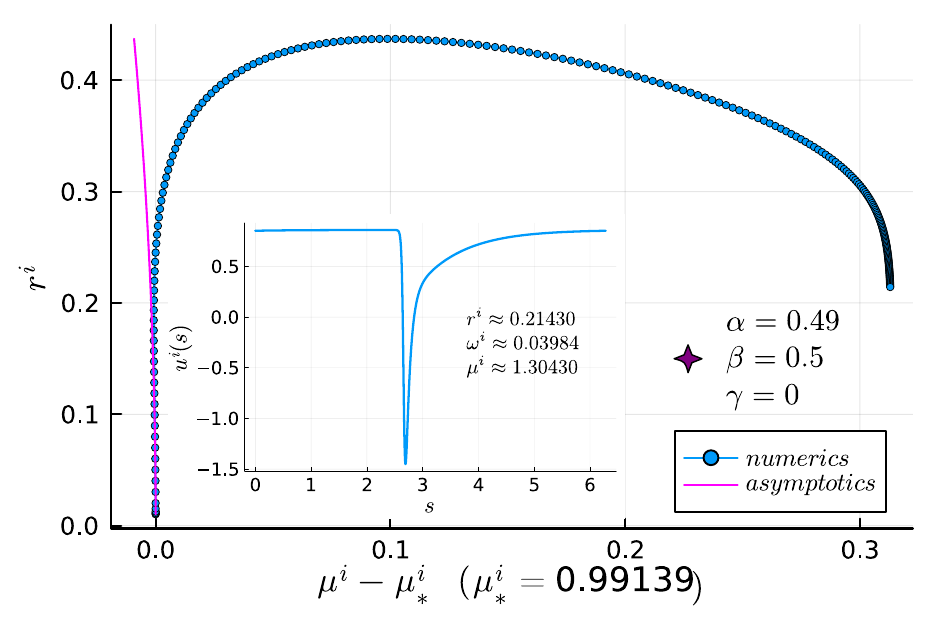}
    \end{subfigure}
    \begin{subfigure}{0.32\textwidth}
    \includegraphics[width=1\textwidth]{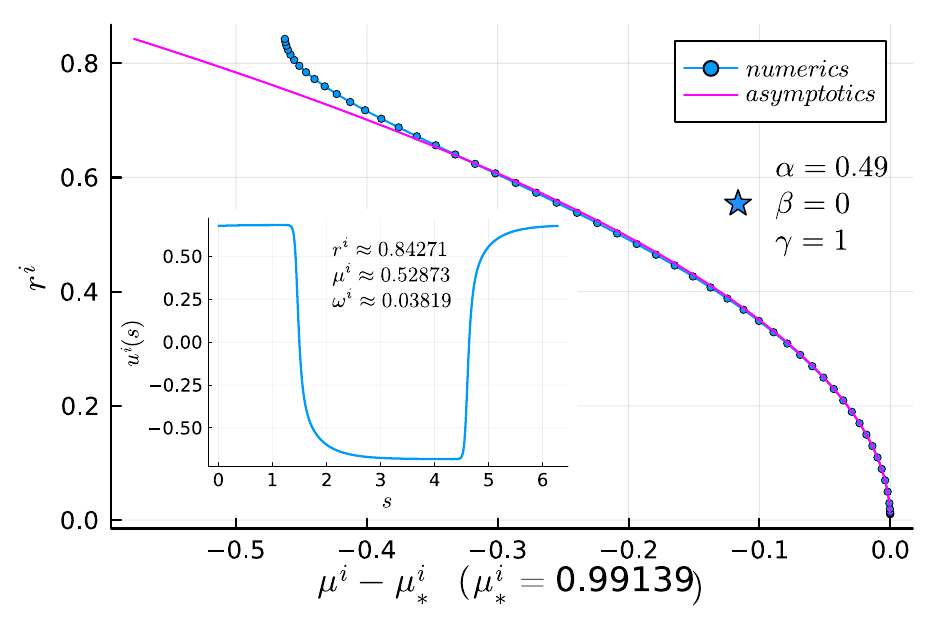}
    \end{subfigure}
    \begin{subfigure}{0.32\textwidth}
    \includegraphics[width=1\textwidth]{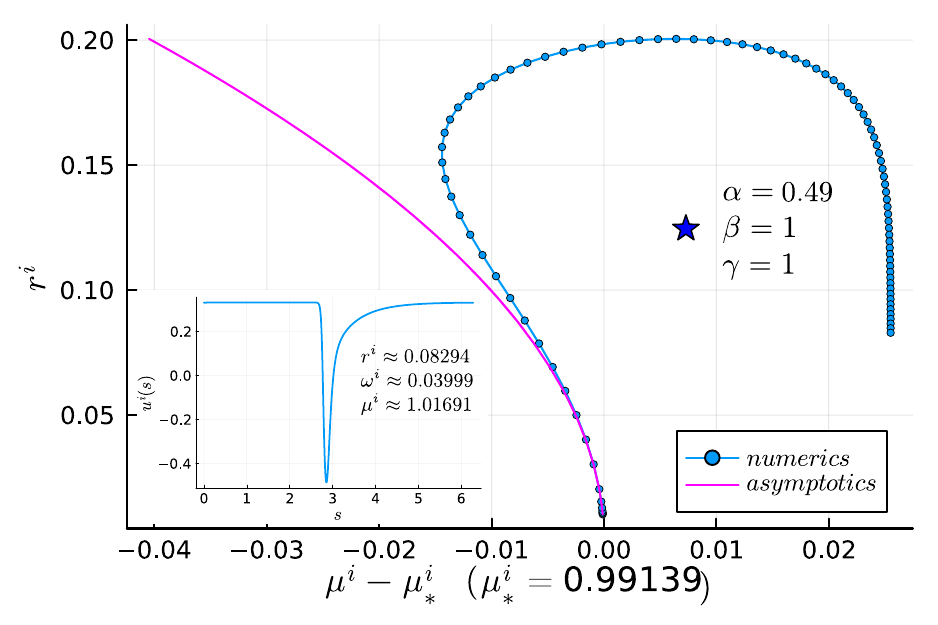}
    \end{subfigure}
    \begin{subfigure}{0.32\textwidth}
    \includegraphics[width=1\textwidth]{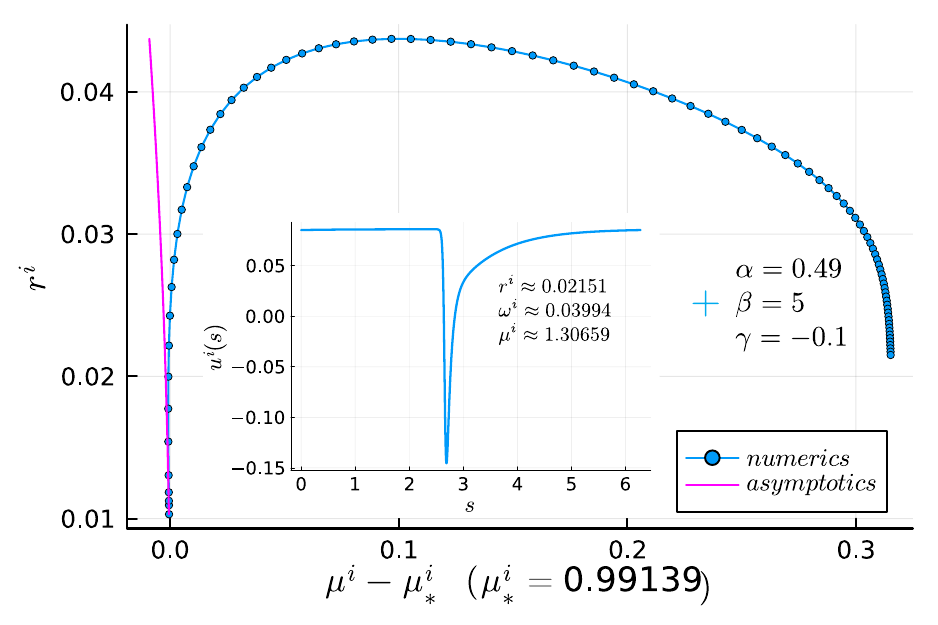}
    \end{subfigure}
    \begin{subfigure}{0.32\textwidth}
    \includegraphics[width=1\textwidth]{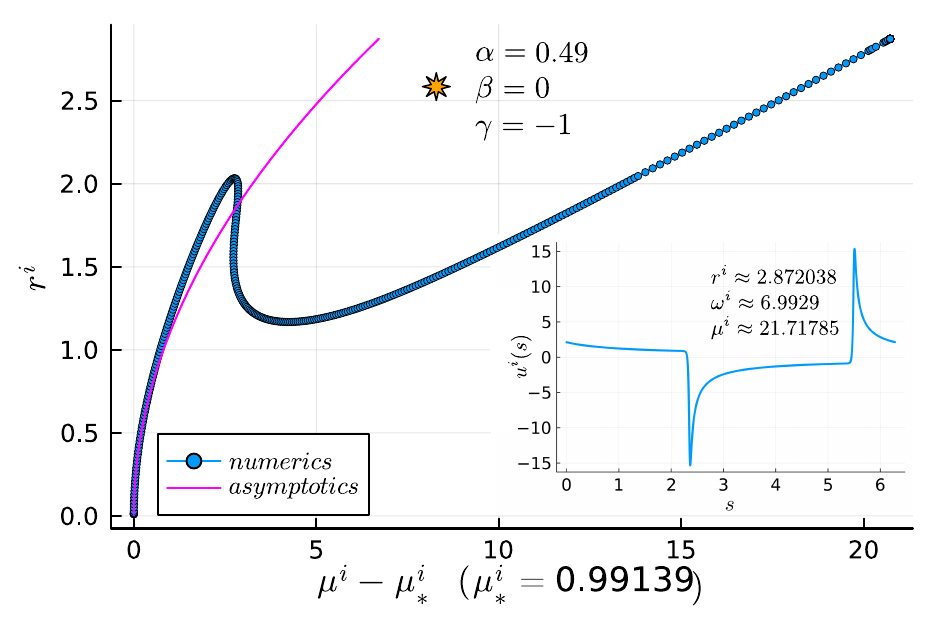}
    \end{subfigure}
    \begin{subfigure}{0.32\textwidth}
    \includegraphics[width=1\textwidth]{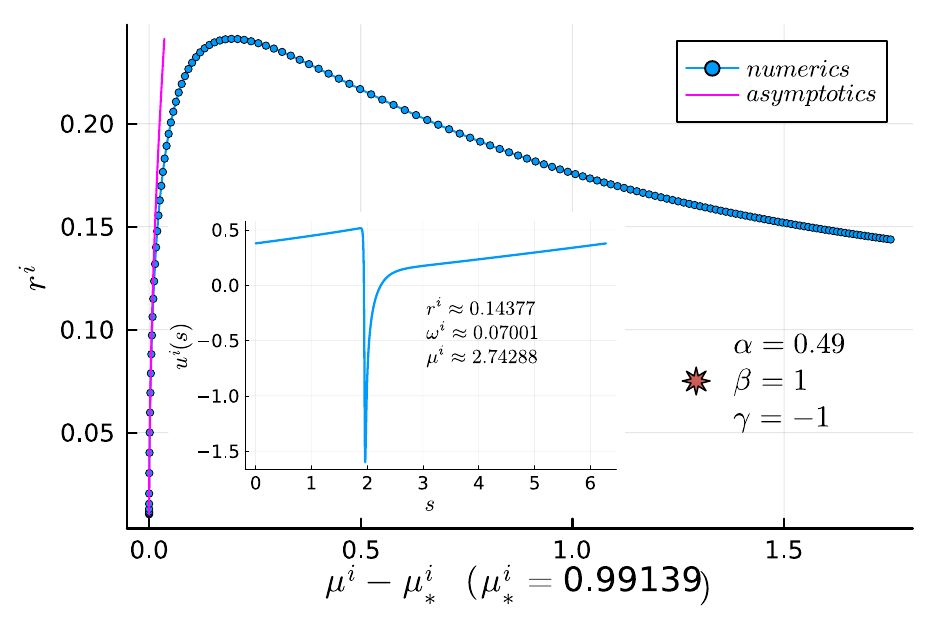}
    \end{subfigure}
    \caption{From left to right and top to bottom, in-phase Hopf bifurcation branches at the points in the $(\beta, \gamma)$-parameter plane of the first row's left plot in Figure \ref{fig:critcurve} in the same order and with $\alpha = 0.49$. Inserted are the nonlinear oscillation shapes that started with a $\cos(s)$ at $\mu^i-\mu_*^i=0$ along with the corresponding amplitude $r^i$, nonlinear oscillation frequency $\omega^i$, and bifurcation parameter $\mu^i$. Note the homoclinic and double heteroclinic limits for all cases except for $(\alpha, \beta, \gamma) = (0.49, 0, -1)$, as $\omega^i$ seems to decrease to 0.}
    \label{fig:in-phase_Hopfbranches}
\end{figure}

Similarly, the anti-phase Hopf bifurcation branches are shown in Figure \ref{fig:anti-phase_Hopfbranches}.

\begin{figure}[!h]
    \centering
    \begin{subfigure}{0.32\textwidth}
    \includegraphics[width=1\textwidth]{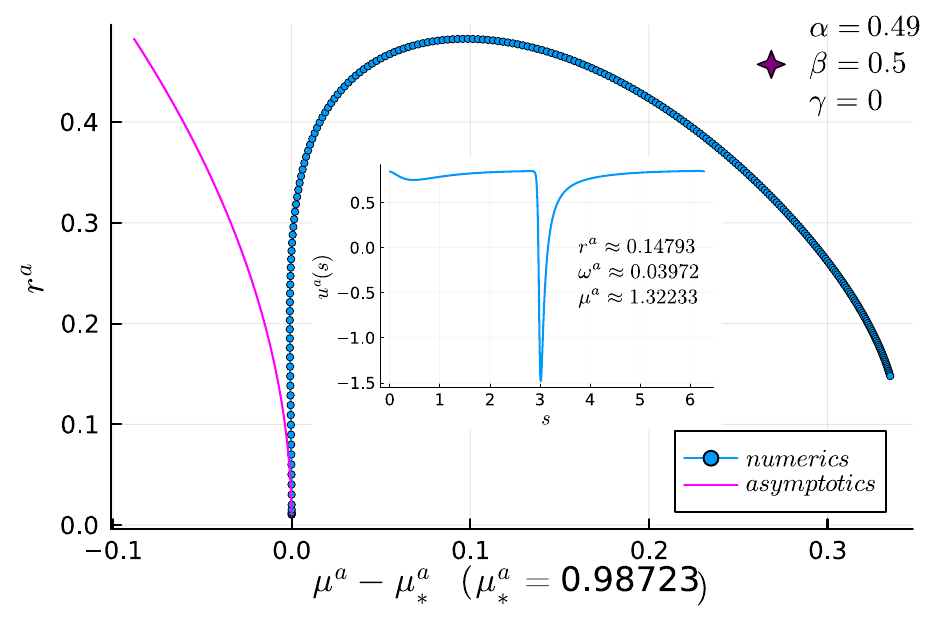}
    \end{subfigure}
    \begin{subfigure}{0.32\textwidth}
    \includegraphics[width=1\textwidth]{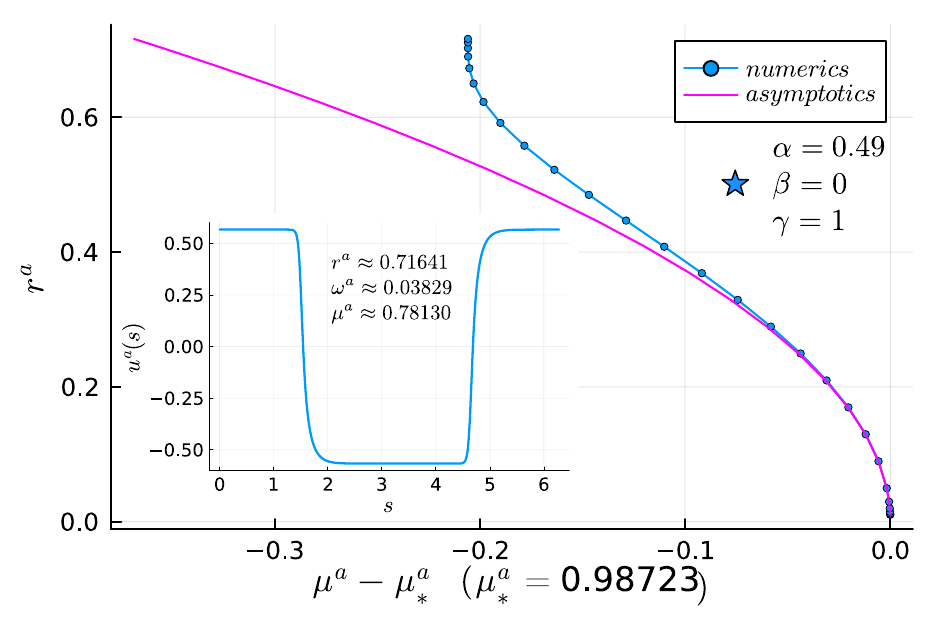}
    \end{subfigure}
    \begin{subfigure}{0.32\textwidth}
    \includegraphics[width=1\textwidth]{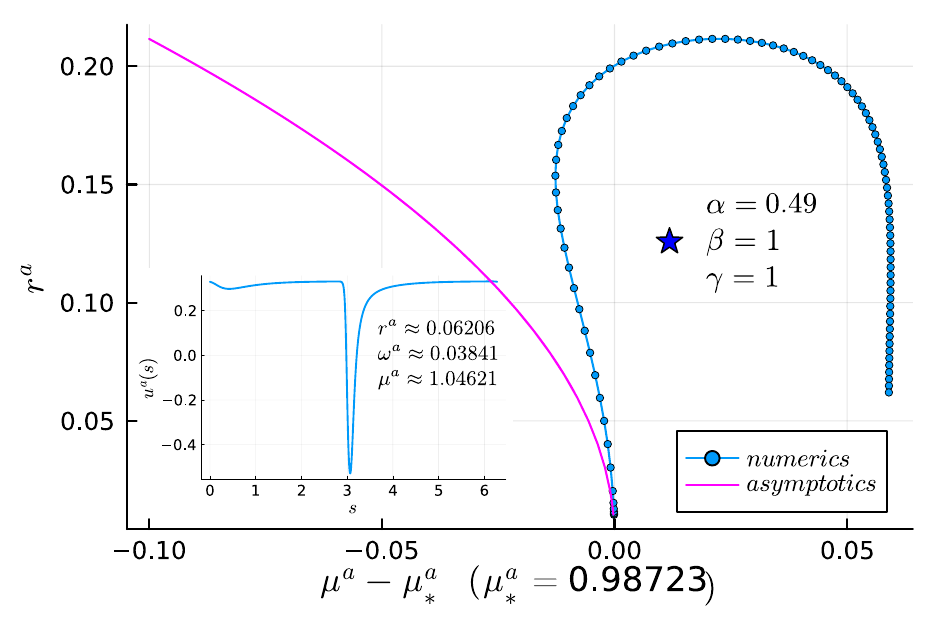}
    \end{subfigure}
    \begin{subfigure}{0.32\textwidth}
    \includegraphics[width=1\textwidth]{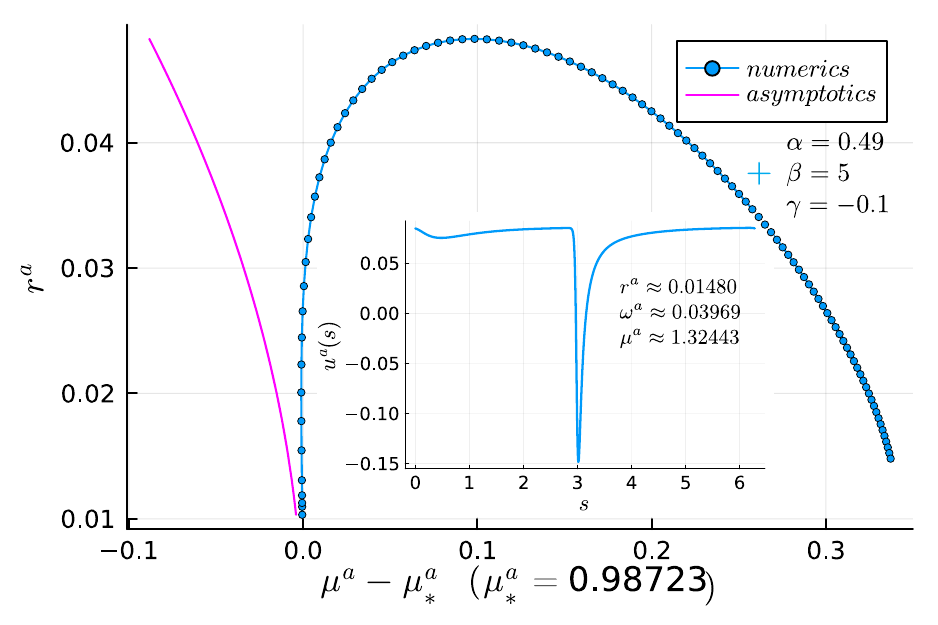}
    \end{subfigure}
    \begin{subfigure}{0.32\textwidth}
    \includegraphics[width=1\textwidth]{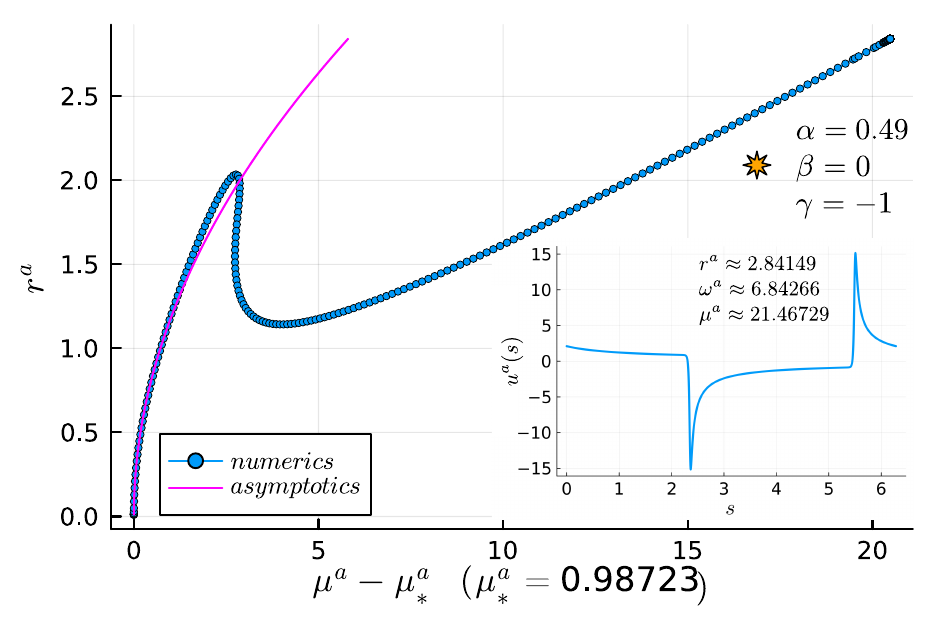}
    \end{subfigure}
    \begin{subfigure}{0.32\textwidth}
    \includegraphics[width=1\textwidth]{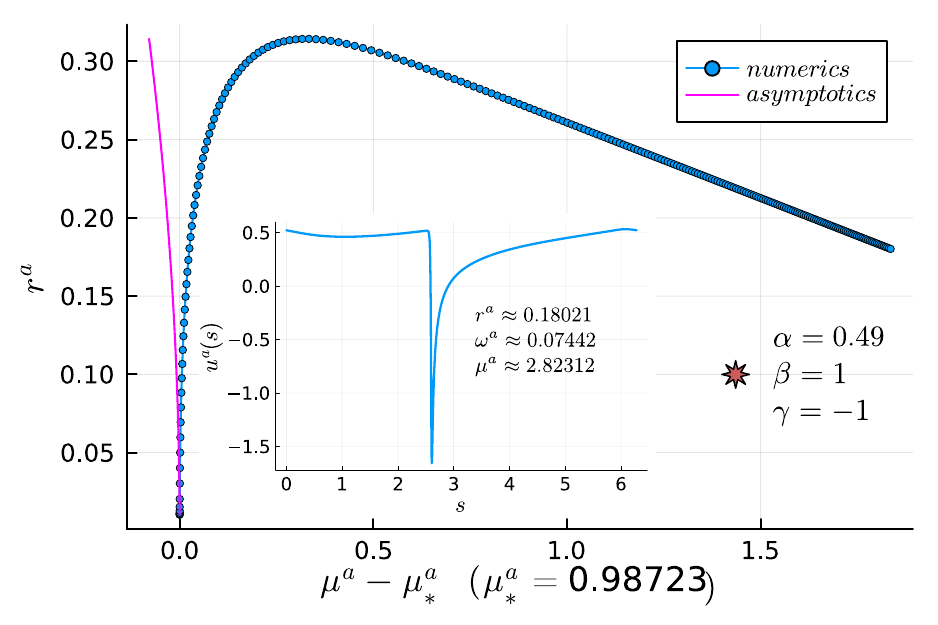}
    \end{subfigure}
    \caption{From left to right and top to bottom, anti-phase Hopf bifurcation branches at the points in the $(\beta, \gamma)$-parameter plane of the first row's left plot in Figure \ref{fig:critcurve} in the same order and with $\alpha = 0.49$. Note that the special point $(\beta, \gamma) = (1, -1)$ is selected, where interesting temporal defects appear in the full numerical solution shown in Figure \ref{fig:timesol_beta1_part2} below. Inserted are the nonlinear oscillation shapes that started with a $\cos(s)$ at $\mu^a-\mu_*^a=0$ along with the corresponding amplitude $r^a$, nonlinear oscillation frequency $\omega^a$, and bifurcation parameter $\mu^a$. Note the homoclinic and double heteroclinic limits for all cases except for $(\alpha, \beta, \gamma) = (0.49, 0, -1)$, as $\omega^a$ seems to decrease to 0.}
    \label{fig:anti-phase_Hopfbranches}
\end{figure}

The seemingly quadratic bifurcation branches for the symmetric and antisymmetric steady-states are exhibited in Figure \ref{fig:steadystatebranches}.
\begin{figure}[!h]
    \centering
    \begin{subfigure}{0.32\textwidth}
    \includegraphics[width=1\textwidth]{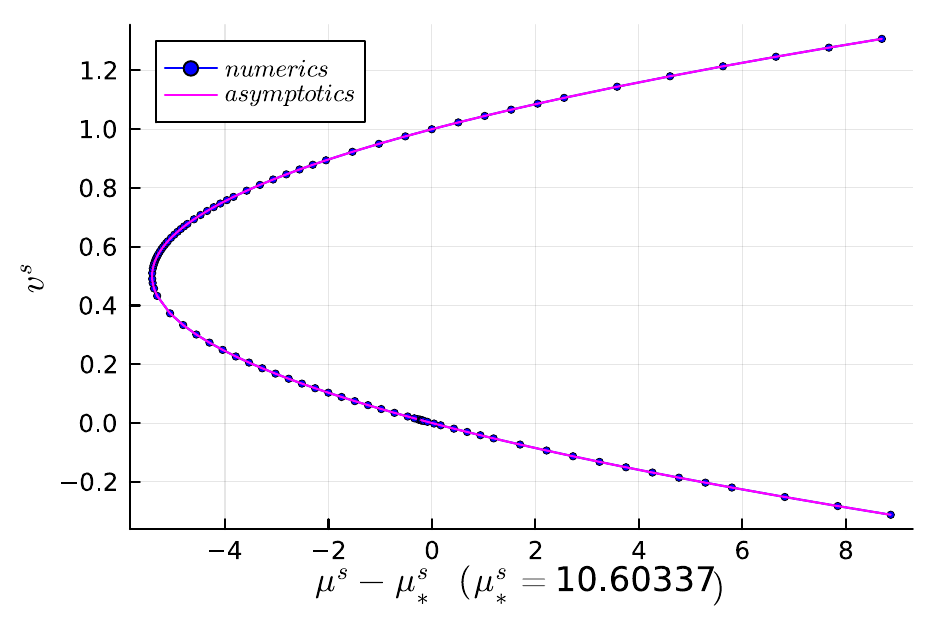}
    \end{subfigure}
    \begin{subfigure}{0.32\textwidth}
    \includegraphics[width=1\textwidth]{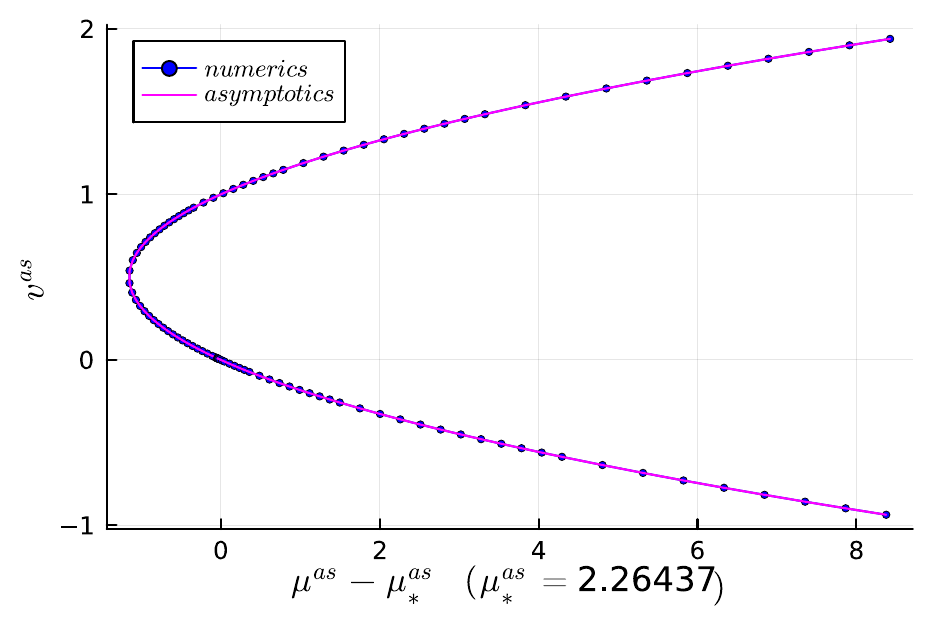}
    \end{subfigure}
    \begin{subfigure}{0.32\textwidth}
    \includegraphics[width=1\textwidth]{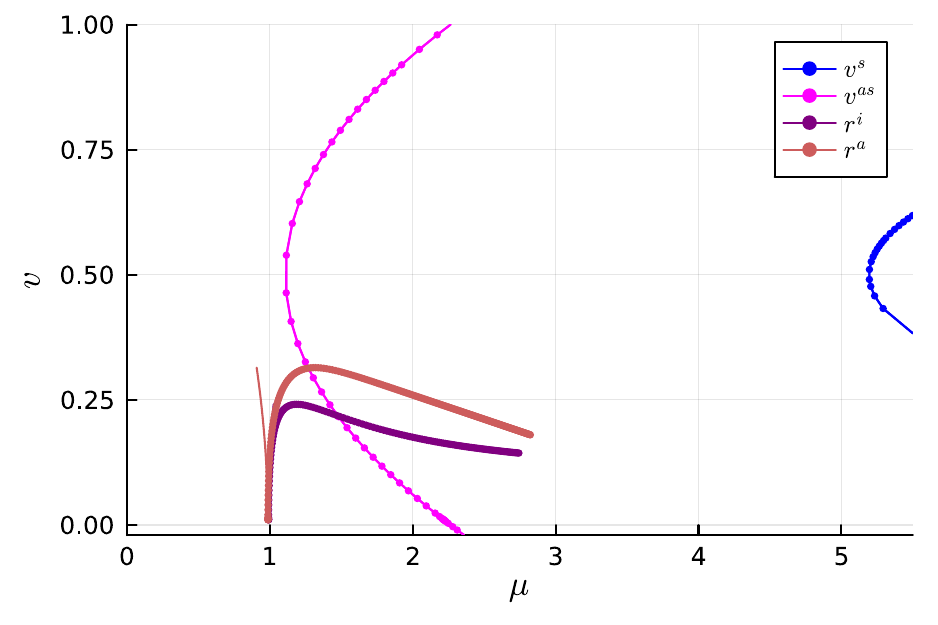}
    \end{subfigure}
    \caption{Left: symmetric steady-state branch, center: antisymmetric steady-state branch, right: all (terminating) Hopf and steady-state branches bifurcating off $u_*\equiv 0$ with $(\alpha, \beta, \gamma, L, \sigma) = (0.49, 1, -1, 10, 0.1)$ since this parameter set exhibits interesting mixing of in-phase and anti-phase Hopf oscillation as shown in Figure \ref{fig:timesol_beta1_part2} below.}
    \label{fig:steadystatebranches}
\end{figure}

\paragraph{Full numerical solutions in space and time.}

The full numerical spatiotemporal solutions have been computed for this study on a uniform time grid with discrete temporal grid points $\{t_j\}_{j=1}^M$ and uniform step size $\Delta t$ with a simple implicit-explicit method similar to the ones from Appendices A and B of \cite{pelz2023emergence}, using the Backward-Euler method for the linear terms and the classical fourth-order Runge-Kutta method for the nonlinear terms. Lumping the values of $u(t, x)$ at the discrete spatial grid points $\{x_k\}_{k=1}^N$ with uniform grid point distance $\Delta x = \frac{L}{N-1}$ into a vector denoted by $\vu(t)\in\bbR^N$, system \eqref{eq:sys_tocubic} can be rewritten as 
\begin{equation*}
    \frac{d}{dt}\vu(t) = \mathbb{L} \vu + (\fint(\vu_1) - \fint'(0)), 0, ..., 0, \fint(\vu_N) - \fint'(0))^T + \cO(\Delta x^3)
\end{equation*}
where the linear operator $\mathbb{L}$ includes a centered second-order discretization of the Laplacian $\partial_{xx} u$ on $\{x_k\}_{k=2}^{N-1}$ (namely $\partial_{xx}|_{x=x_k} u =\frac{1}{\Delta x^2}(\vu_{k+1} - 2 \vu_{k} + \vu_{k-1}) + \cO(\Delta x^3)$), a non-centered second-order discretization of the gradient $\partial_n u$ at $x_1$ (namely $\partial_n|_{x=x_1} u = \frac{1}{\Delta x}(\frac{3}{2}\vu_1 - 2 \vu_{2} + \frac{1}{2}\vu_{3}) + \cO(\Delta x^3)$) and at $x_n$ (namely $\partial_n|_{x=x_1} u =\frac{1}{\Delta x}(\frac{3}{2}\vu_N - 2 \vu_{N-1} + \frac{1}{2}\vu_{N+2}) + \cO(\Delta x^3)$), linear degradation $\{-\sigma^2 \vu_k\}_{k=1}^{N-1}$, and the internal reaction kinetics $\fint$ linearized at the base steady-state $u_*\equiv 0$, namely $\fint'(0)$, at $x=x_1$ and at $x=x_N$.

The spectrum of the linear operator $\mathbb{L}$ shows a good approximation to the combined spectra of the in-phase and anti-phase Hopf Fredholm operators $\cL^i$ and $\cL^a$. The numerical solution to \eqref{eq:sys_tocubic} at each scattered point of Figure \ref{fig:Relambda_pos_i} along with the spectrum of $\mathbb{L}$ is provided in Figures \ref{fig:timesol_beta1_part1} and \ref{fig:timesol_beta1_part2} for $(\alpha, \beta, \gamma) = (0.49, 1, -1)$. It can be seen that the nonlinearities cause the dynamics to deviate from the dynamics the linearized spectrum hints at. I would like to state here that the numerical solution for $(\alpha, \beta, \gamma) = (0.49, 1.53, -1)$ exhibits dynamics close to the linearized dynamics. 

It appears that there are temporal defects in Figure \ref{fig:timesol_beta1_part2}, especially for the parameter choice $(\alpha=0.49, \beta = 1, \gamma = -1)$ corresponding to the bottom-right Hopf branch in Figure \ref{fig:anti-phase_Hopfbranches}. This suggests that the in-phase and anti-phase Hopf oscillation have growth rates $\lambda\in \bbC$ with similar $\Re(\lambda)>0$, so that in-phase and anti-phase oscillations are mixing with each other.

\begin{figure}[]
    \centering
    \begin{subfigure}{0.32\textwidth}
    \includegraphics[width=1\textwidth]{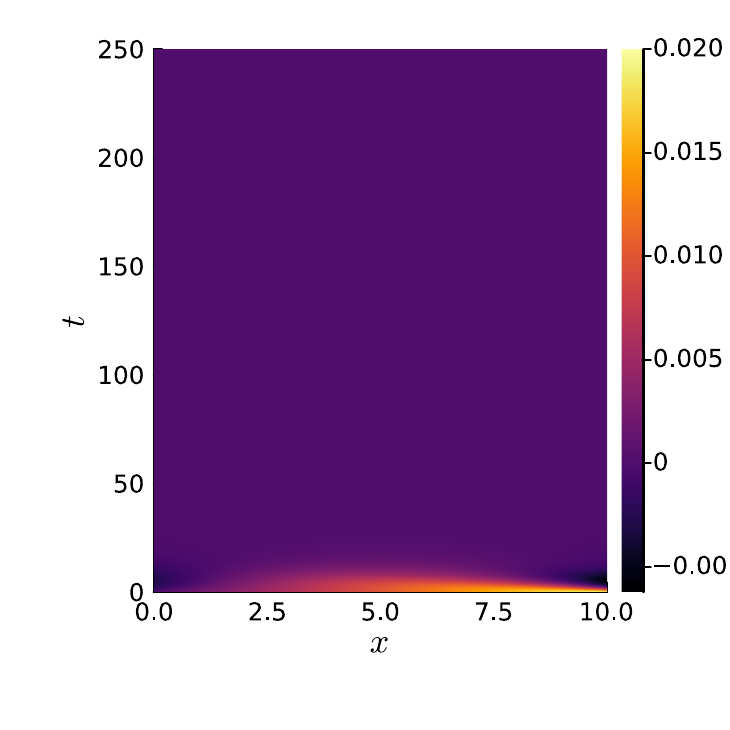}
    \end{subfigure}
    \begin{subfigure}{0.32\textwidth}
    \includegraphics[width=1\textwidth]{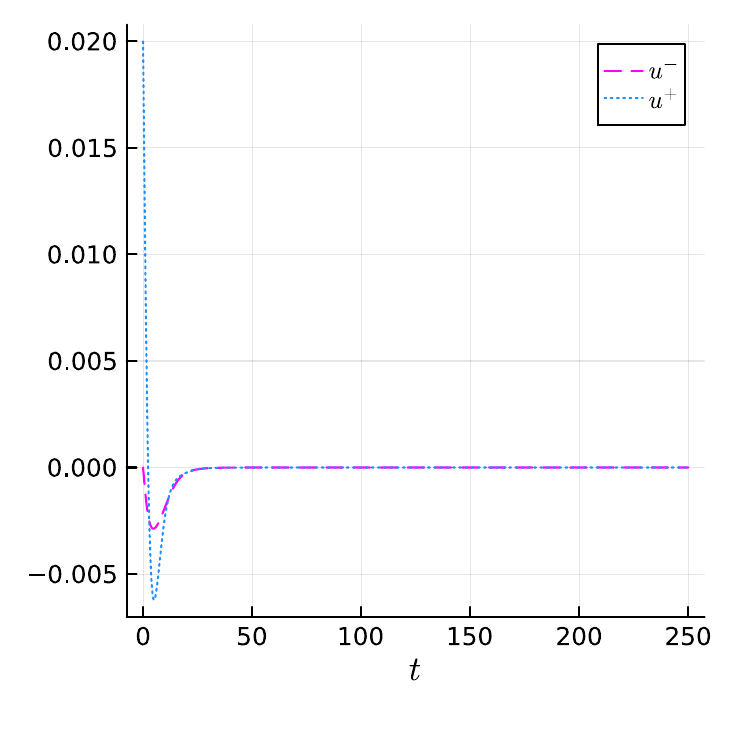}
    \end{subfigure}
    \begin{subfigure}{0.32\textwidth}
    \includegraphics[width=1\textwidth]{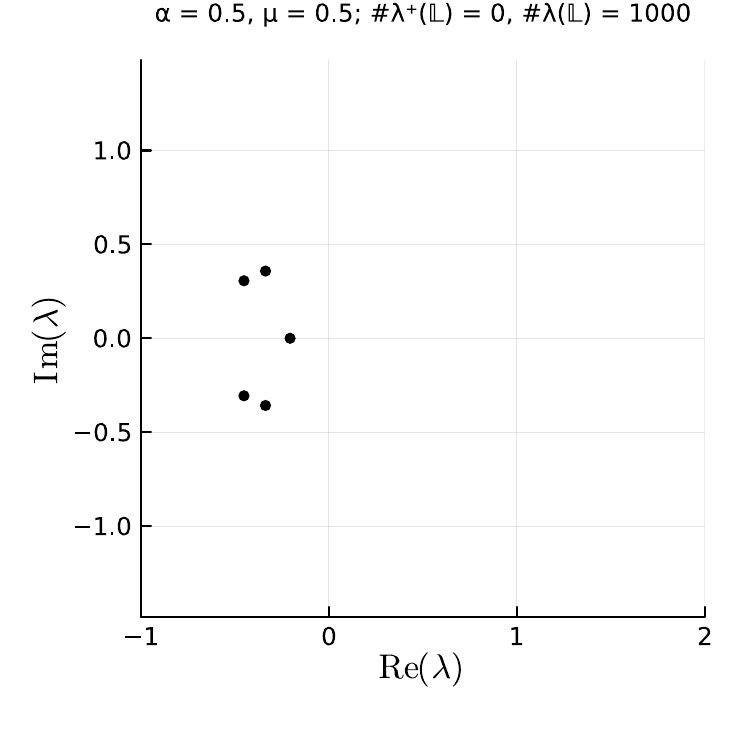}
    \end{subfigure}
    \begin{subfigure}{0.32\textwidth}
    \includegraphics[width=1\textwidth]{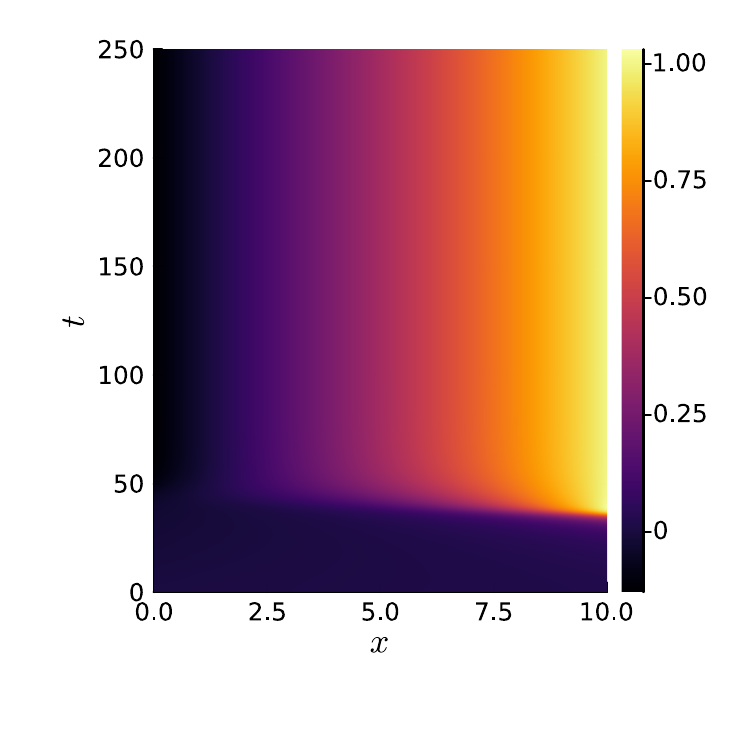}
    \end{subfigure}
    \begin{subfigure}{0.32\textwidth}
    \includegraphics[width=1\textwidth]{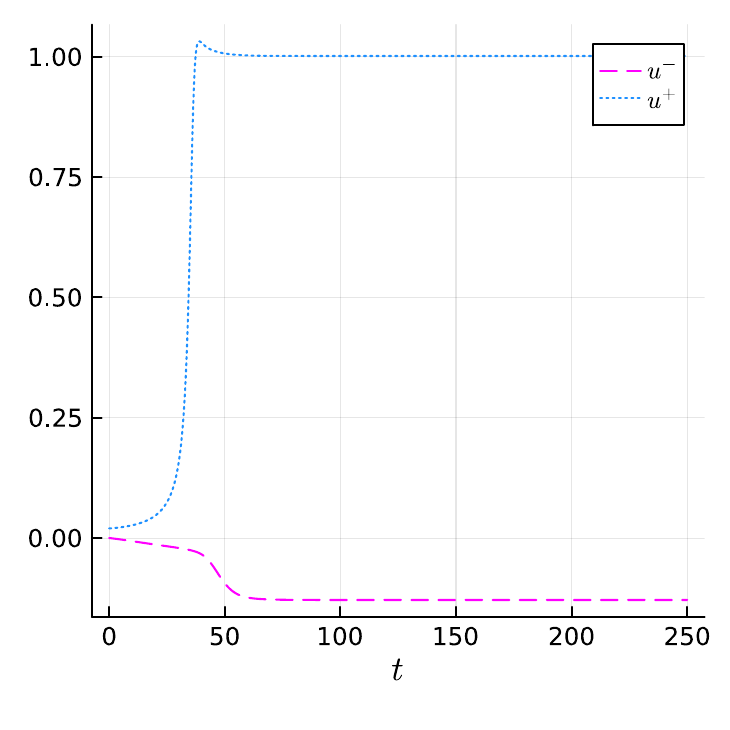}
    \end{subfigure}
    \begin{subfigure}{0.32\textwidth}
    \includegraphics[width=1\textwidth]{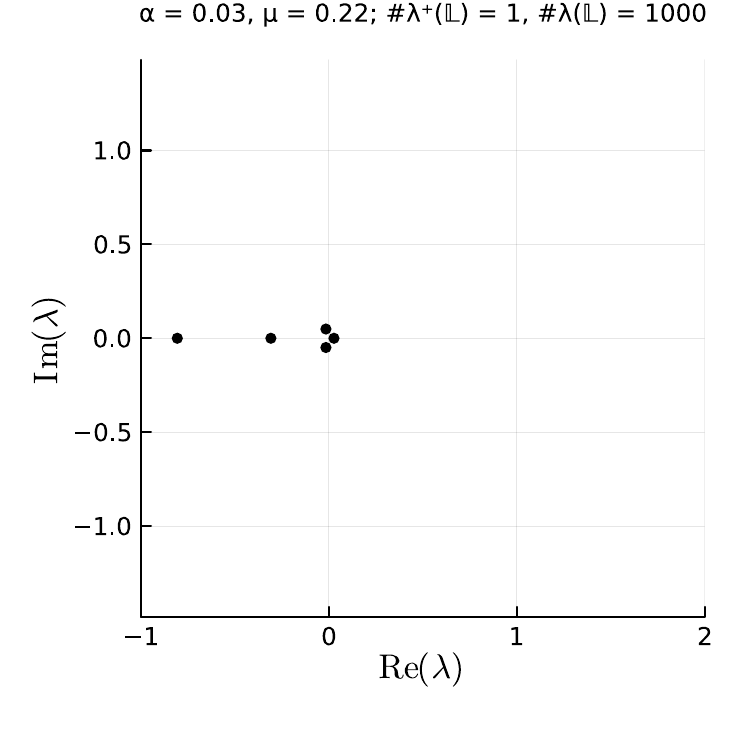}
    \end{subfigure}
    \begin{subfigure}{0.32\textwidth}
    \includegraphics[width=1\textwidth]{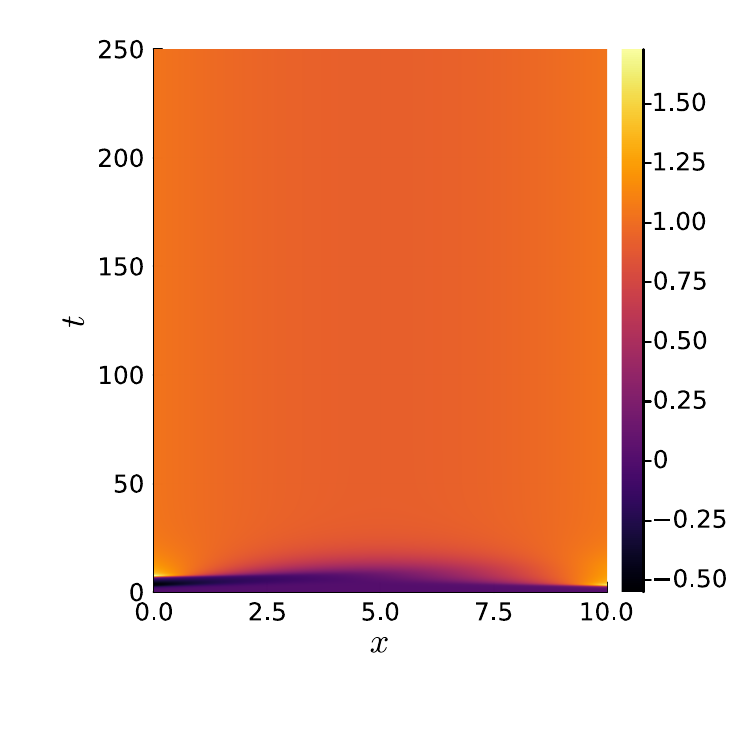}
    \end{subfigure}
    \begin{subfigure}{0.32\textwidth}
    \includegraphics[width=1\textwidth]{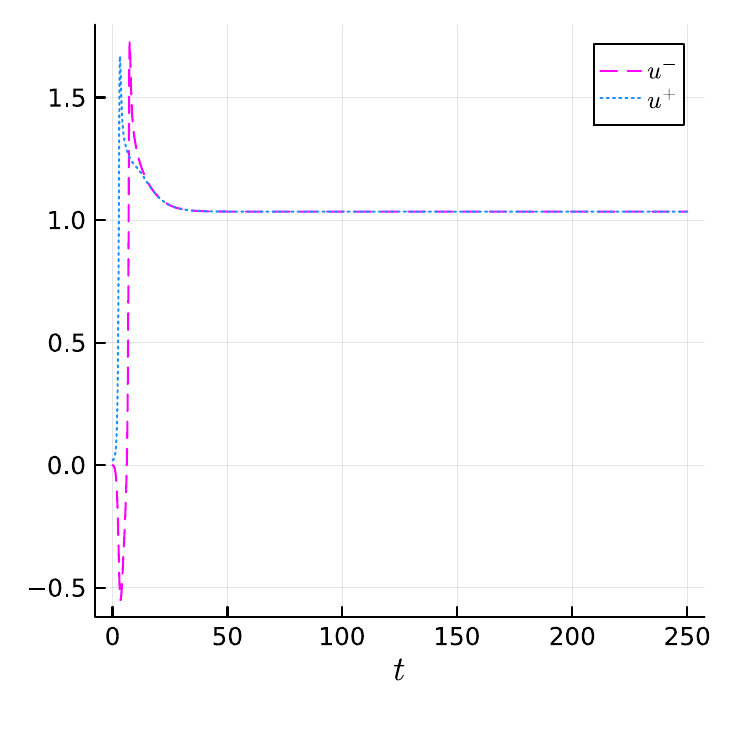}
    \end{subfigure}
    \begin{subfigure}{0.32\textwidth}
    \includegraphics[width=1\textwidth]{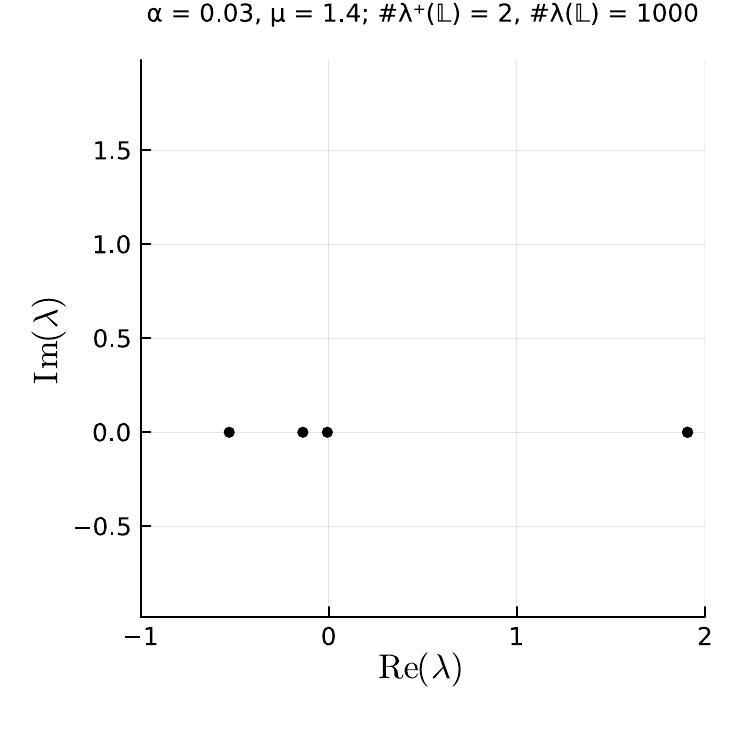}
    \end{subfigure}
    \begin{subfigure}{0.32\textwidth}
    \includegraphics[width=1\textwidth]{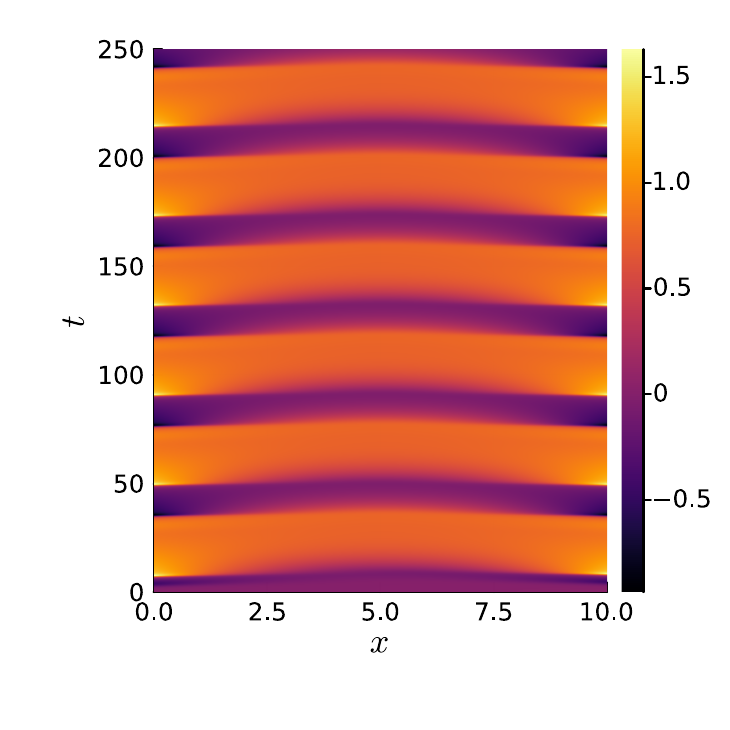}
    \end{subfigure}
    \begin{subfigure}{0.32\textwidth}
    \includegraphics[width=1\textwidth]{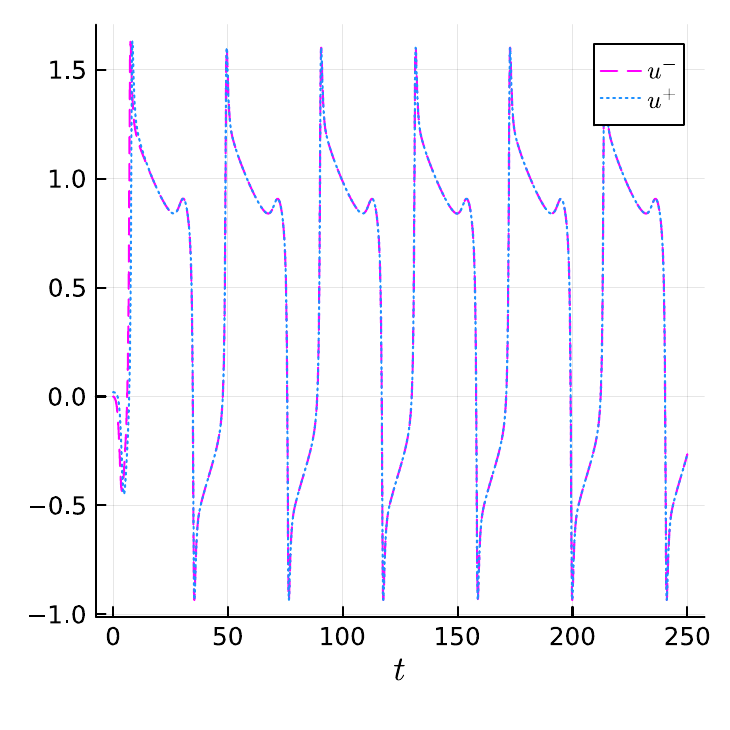}
    \end{subfigure}
    \begin{subfigure}{0.32\textwidth}
    \includegraphics[width=1\textwidth]{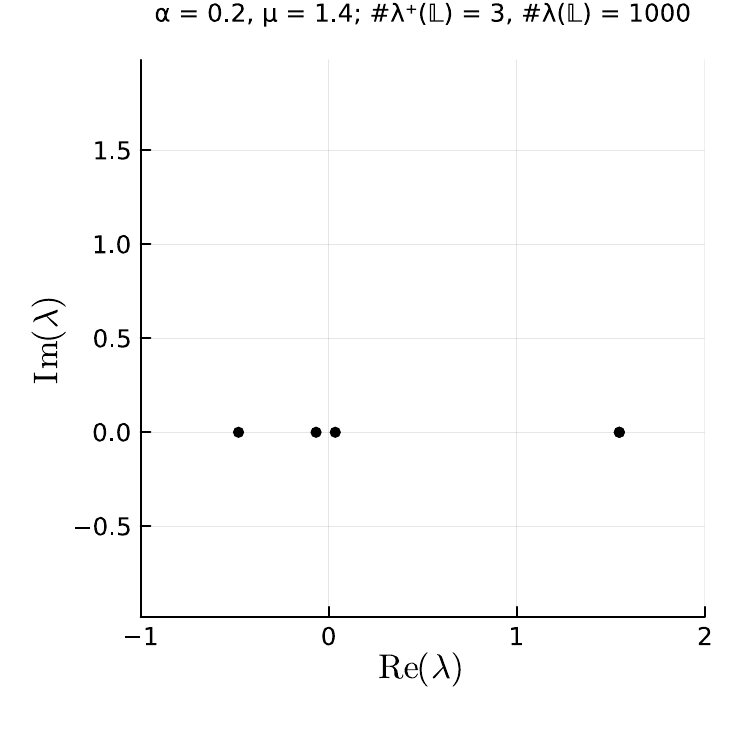}
    \end{subfigure}
    \caption{For $(\beta, \gamma, L, \sigma) =(1, -1, 10, 0.1)$, numerical spatiotemporal solution in the bulk domain $(0, L)$ ($u$; left column) and at the boundary points ($u(t,0)=u^-$ and $u(t,L)=u^+$; center column) and the eigenvalues of the linear numerical operator $\mathbb{L}$ with biggest real part ($\{\lambda\in\mathrm{spec}(\mathbb{L})\,|\, \Re(\lambda) > -1\}$; right column). The number of eigenvalues $\lambda$ of $\mathbb{L}$ with $\Re(\lambda)>0$ is denoted by $\#\lambda^+(\mathbb{L})$, while the overall number of eigenvalues of $\mathbb{L}$ is denoted by $\#\lambda(\mathbb{L})$.}
    \label{fig:timesol_beta1_part1}
\end{figure}

\begin{figure}[H]
\ContinuedFloat
\centering
\caption[]{From top to bottom, starting from  the black scattered point with white circle of Figure \ref{fig:Relambda_pos_i} and going counterclockwise: $(\alpha, \mu) = (0.5, 0.5)$, $(\alpha, \mu) = (0.03, 0.22)$, $(\alpha, \mu) = (0.03, 1.4)$, $(\alpha, \mu) = (0.2, 1.4)$. The initial condition is the symmetric base steady-state now with $u_*^s=0$ perturbed by the linear combination of a symmetric and antisymmetric perturbation with boundary values $\delta\cdot (1, 1) + \delta\cdot(1, -1)$ with $\delta=0.01$ and linear interpolation in between the boundary points. The number of grid points used were $N = 10^4$ for the spatial grid and  $M = 10^3$ for the temporal grid with final time $T = 250$.}
\end{figure}

\begin{figure}[]
    \centering
    \begin{subfigure}{0.32\textwidth}
    \includegraphics[width=1\textwidth]{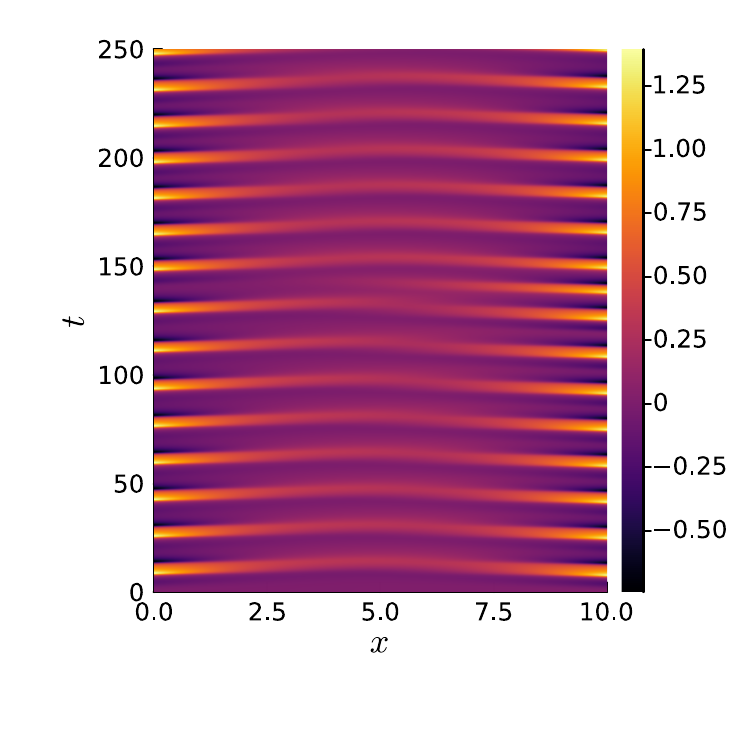}
    \end{subfigure}
    \begin{subfigure}{0.32\textwidth}
    \includegraphics[width=1\textwidth]{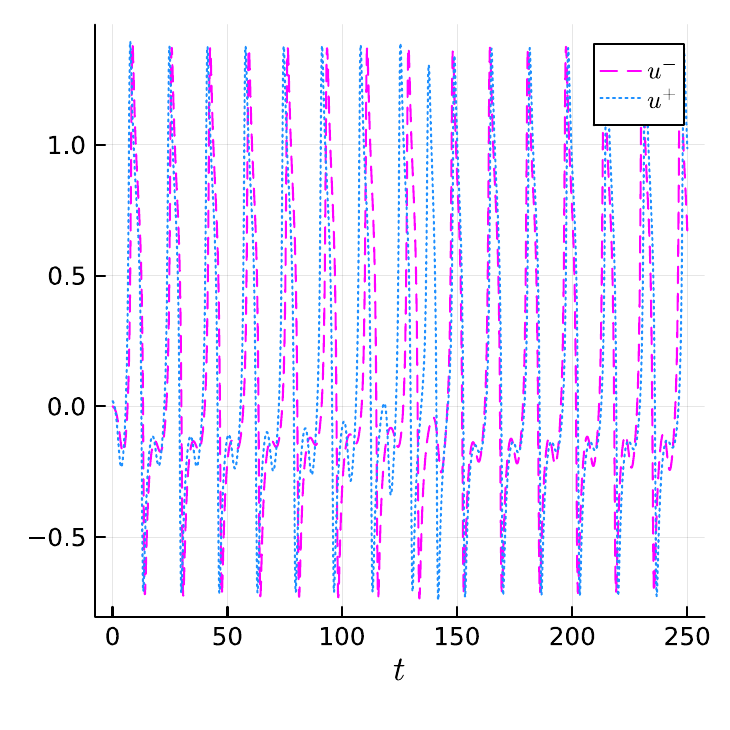}
    \end{subfigure}
    \begin{subfigure}{0.32\textwidth}
    \includegraphics[width=1\textwidth]{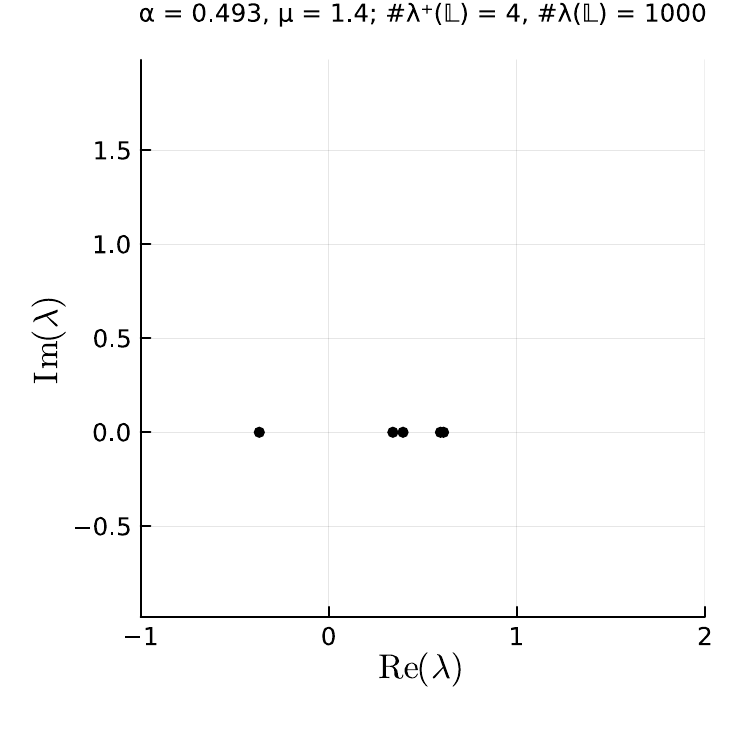}
    \end{subfigure}
    \begin{subfigure}{0.32\textwidth}
    \includegraphics[width=1\textwidth]{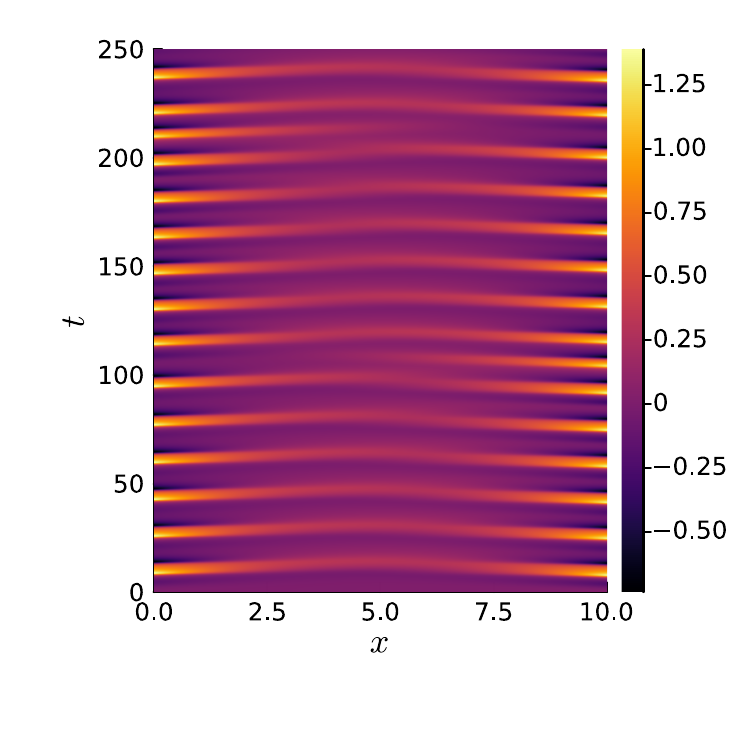}
    \end{subfigure}
    \begin{subfigure}{0.32\textwidth}
    \includegraphics[width=1\textwidth]{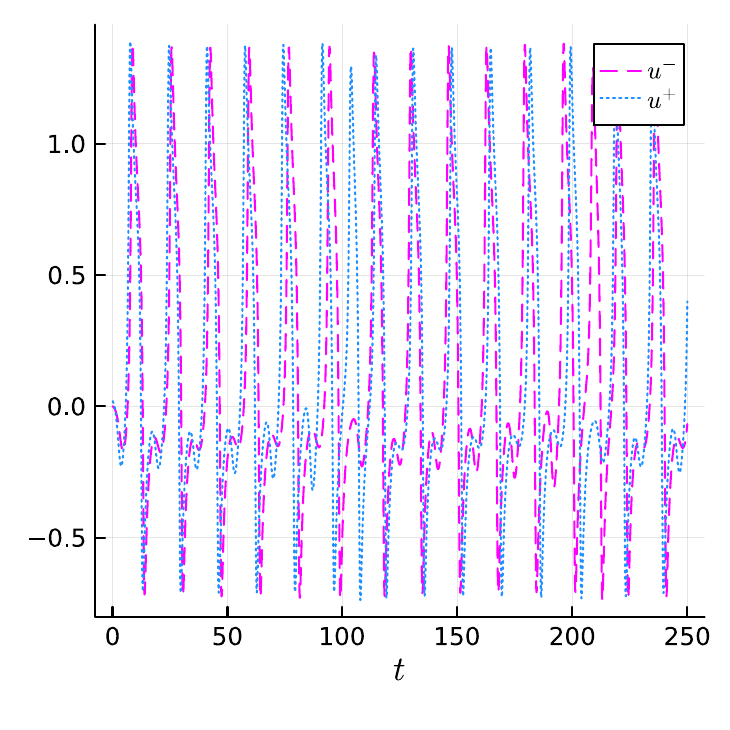}
    \end{subfigure}
    \begin{subfigure}{0.32\textwidth}
    \includegraphics[width=1\textwidth]{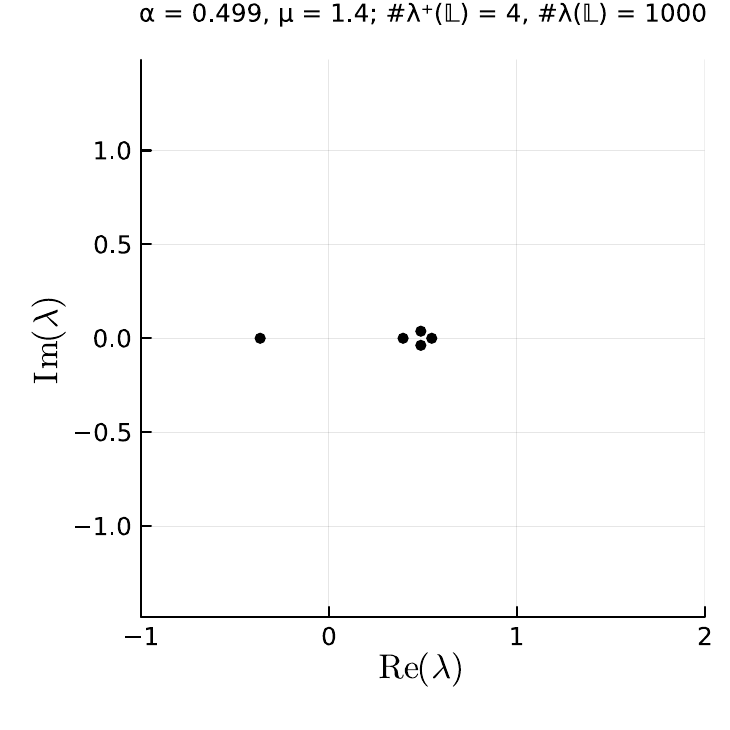}
    \end{subfigure}
    \begin{subfigure}{0.32\textwidth}
    \includegraphics[width=1\textwidth]{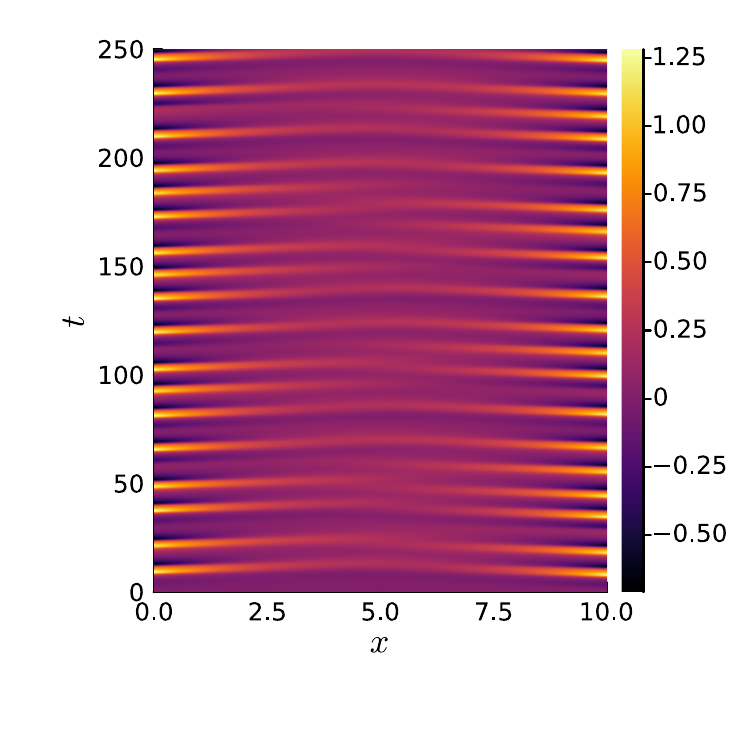}
    \end{subfigure}
    \begin{subfigure}{0.32\textwidth}
    \includegraphics[width=1\textwidth]{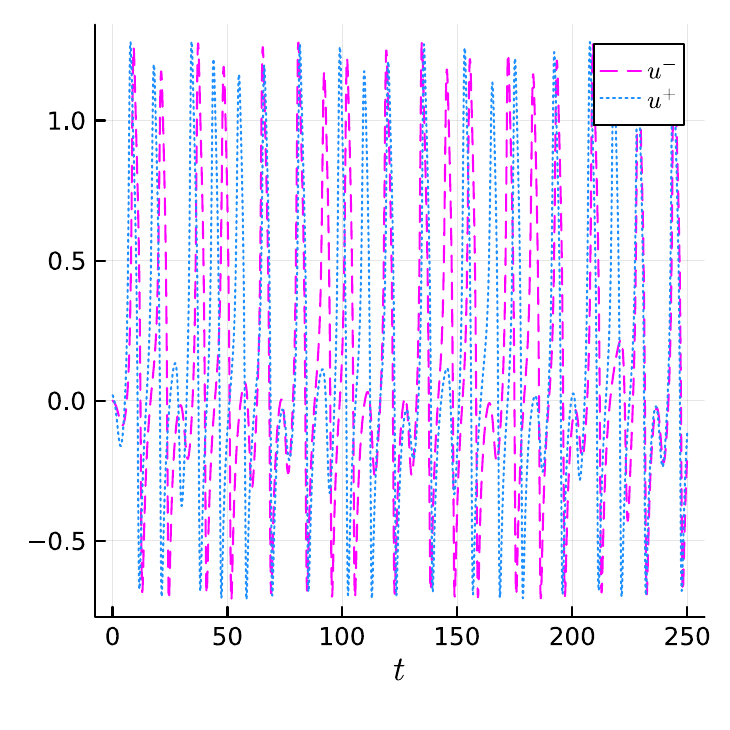}
    \end{subfigure}
    \begin{subfigure}{0.32\textwidth}
    \includegraphics[width=1\textwidth]{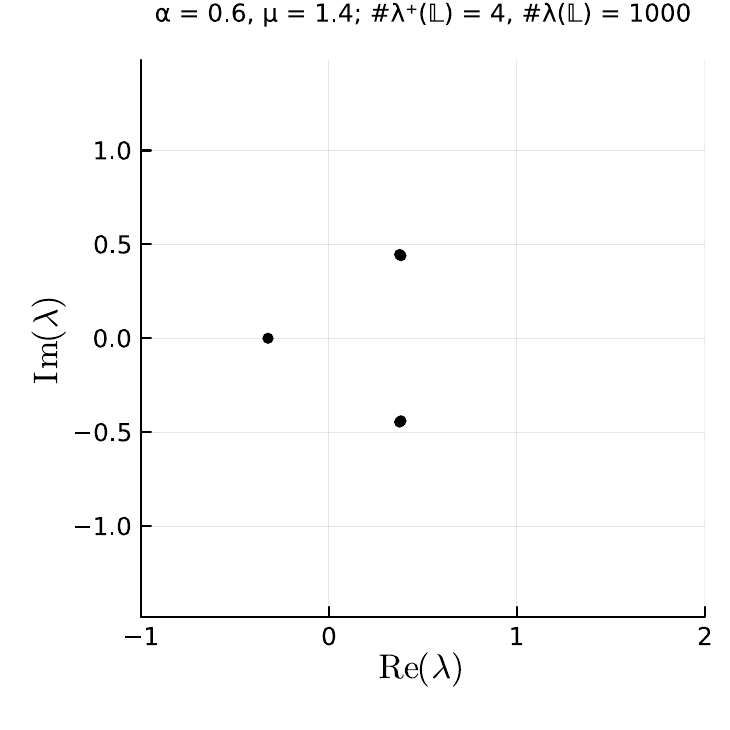}
    \end{subfigure}
        \begin{subfigure}{0.32\textwidth}
    \includegraphics[width=1\textwidth]{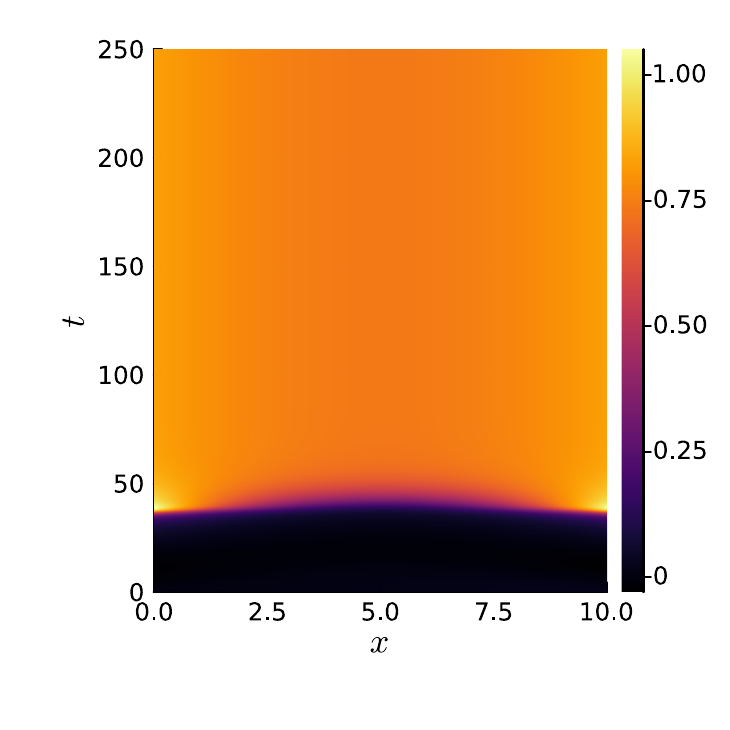}
    \end{subfigure}
    \begin{subfigure}{0.32\textwidth}
    \includegraphics[width=1\textwidth]{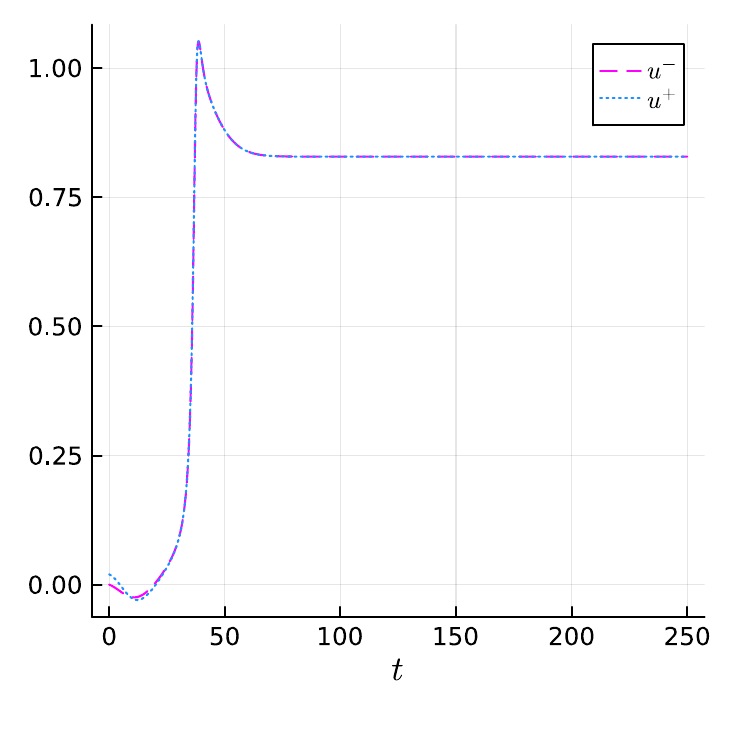}
    \end{subfigure}
    \begin{subfigure}{0.32\textwidth}
    \includegraphics[width=1\textwidth]{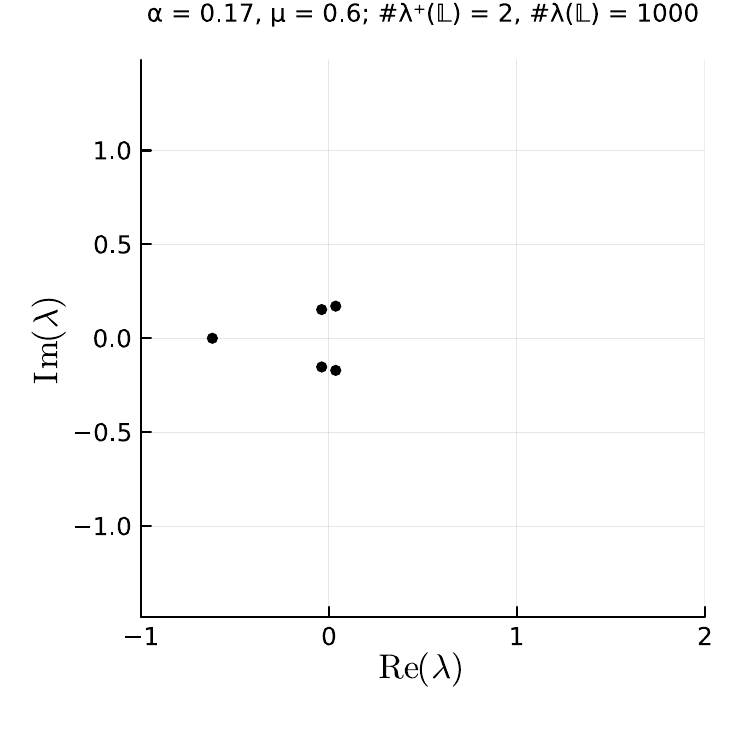}
    \end{subfigure}
    \caption{Same caption as for Figure \ref{fig:timesol_beta1_part1} but now from top to bottom: $(\alpha, \mu) = (0.493, 1.4)$, $(\alpha, \mu) = (0.499, 1.4)$, $(\alpha, \mu) = (0.6, 1.4)$, $(\alpha, \mu) = (0.17, 0.6)$. The interesting temporal defects that can be observed simply suggest that the real parts $\Re(\lambda)>0$ of the growth rate of the in-phase Hopf oscillation and of the anti-phase Hopf oscillation are of similar magnitude. Hence, in-phase and anti-phase oscillations mix with each other, creating the impression of a sum of two $\cos(\omega t)$ with slightly different oscillation frequency $\omega$.}
    \label{fig:timesol_beta1_part2}
\end{figure}

\paragraph{Hopf dance.}
The attentive reader may have noticed the intertwining and intersecting nature of the in-phase and anti-phase Hopf bifurcation curves in Figure \ref{fig:in-phase_Hopfbranches}. Such behavior has been called ``Hopf dance'' by the authors of \cite{doelman2018destabilization} and the references it builds on. However, their underlying reaction-diffusion system in (1.1) and (1.2) is situated on the unbounded real line $\bbR$ and has nonlinear reaction kinetics at every point inside the domain that generates spatially periodic pulse solutions, opposed to the two discrete diffusively-coupled but separate locations where nonlinear reaction kinetics take place in system \eqref{eq:sys} herein that exhibit temporal periodicity. Furthermore, their Hopf dance is illustrated in (spatial wavenumber, bifurcation parameter)-space. The authors describe that the last (or first) destabilizing state to a spatially homogeneous base state is located at the point of simultaneous emergence of the intertwining Hopf curves in (spatial wavenumber, bifurcation parameter)-space where the spatial wavenumber has converged to zero. In contrast, in this study the Hopf curves emerge from the steady-state bifurcation curves that share their respective (symmetric or antisymmetric) symmetry at low norms for $(\alpha, \mu)$ and seem to come exponentially closer to each other for increasing norms for $(\alpha, \mu)$. Figure \ref{fig:hopfdance} shows a plot of the ``Hopf dance'' of system \eqref{eq:sys_tocubic} for $(L, \sigma) = (10, 0.1)$ that is visualized once through the difference of the in-phase with the anti-phase Hopf bifurcation lines and further through the natural logarithm of that difference. The latter can be used to get an approximate estimate of the frequency of the intertwining of the lines and the exponential rate with which the two Hopf lines approach each other as the bifurcation parameter $\mu$ increases.

System \eqref{eq:sys_tocubic} also exhibits homoclinic and double-heteroclinic limits, now in the temporal domain, when following the in-phase and anti-phase Hopf branches that emerge at a point on the intertwining Hopf curves (see Figures \ref{fig:in-phase_Hopfbranches} and \ref{fig:anti-phase_Hopfbranches}, except for the case of $(\alpha, \beta, \gamma) = (0.4, 0, -1)$). It would be highly interesting to connect the theory of \cite{doelman2018destabilization} to systems with diffusively coupled discrete active structures, e.g., the boundaries of \eqref{eq:sys}.
\begin{figure}[H]
    \centering
    \begin{subfigure}{0.37\textwidth}
    \includegraphics[width=1\textwidth]{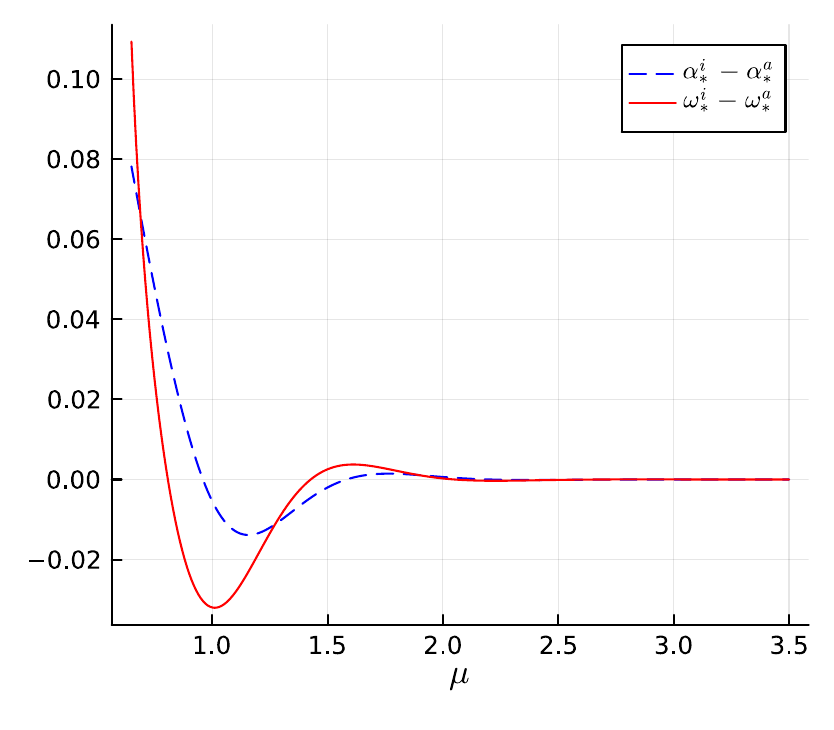}
    \end{subfigure}
    \begin{subfigure}{0.37\textwidth}
    \includegraphics[width=1\textwidth]{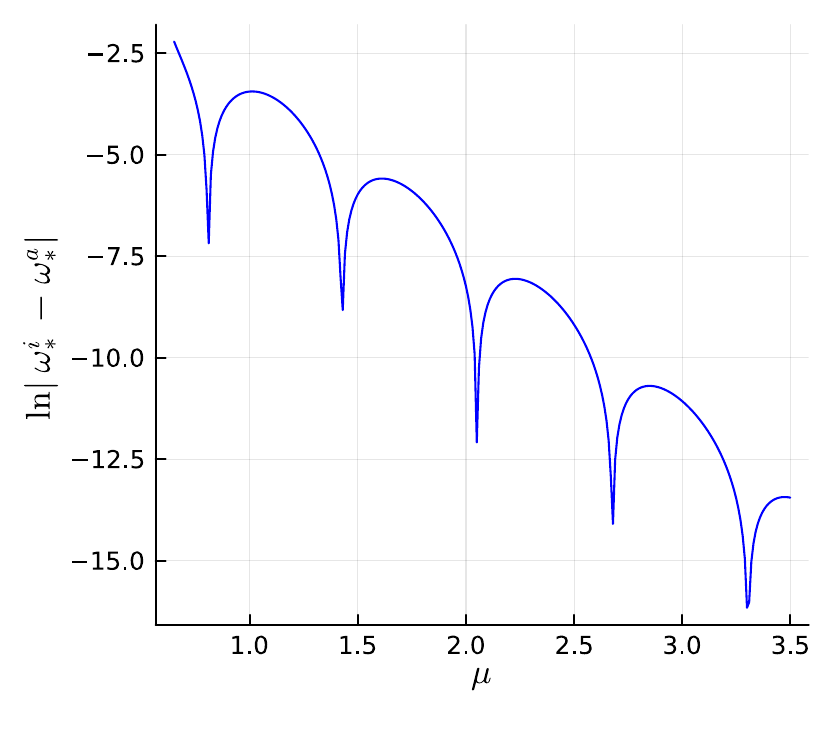}
    \end{subfigure}
    \caption{Left: Visualization of the intertwining nature of the in-phase and anti-phase Hopf bifurcation lines through their difference either in $(\mu, \alpha_*)$-space that can be compared with Figure \ref{fig:in-phase_Hopfbranches} or in $(\mu, \omega_*)$-space. Right: Plot of the natural logarithm of the difference, from which the frequency and exponential convergence rate of the Hopf lines toward each other can be approximately read off. The intertwining nature of the Hopf lines is coined ``Hopf dance'' in \cite{doelman2018destabilization}, where the authors have spectrally analyzed this seemingly general phenomenon thoroughly in the case of a class of reaction-diffusion systems with pulse solution with spatial periodicity.}
    \label{fig:hopfdance}
\end{figure}

\paragraph{Connection to Schrödinger/Gross-Pitaevskii equations.}
It would be interesting to find connections of system \eqref{eq:sys} with the class of Schrödinger/Gross-Pitaevskii equations with a simple double-well potential of \cite{kirr2008symmetry}. The authors' bifurcation diagram and that of Figure \ref{fig:Relambda_pos_i} look similar while their system exhibits an interaction of the two minima of the potential with each other, similar to the boundary interactions in \eqref{eq:sys}. Making the ansatz $\psi(t, x) = u(x) e^{it}$ transforms their equation into a reaction-diffusion steady-state problem for $u$ that can be compared to the steady-state problem for \eqref{eq:sys}. The authors highlight an interesting bifurcation from a symmetric steady-state to an asymmetric (mixed) steady-state through destabilizing antisymmetric perturbations. Their results have to be interpreted as a bifurcation from a symmetric oscillation to a stable asymmetric oscillation, as opposed to steady-states with the corresponding symmetry. See Figure  for the numerical spatiotemporal solution for $(\alpha, \beta, \gamma, L, \mu, \sigma) = (0.03, 1, -1, 10, 0.22, 0.1)$ (red scattered point with black circle in first row left plot of Figure \ref{fig:Relambda_pos_i}).

\begin{figure}[H]
    \centering
    \begin{subfigure}{0.32\textwidth}
    \includegraphics[width=1\textwidth]{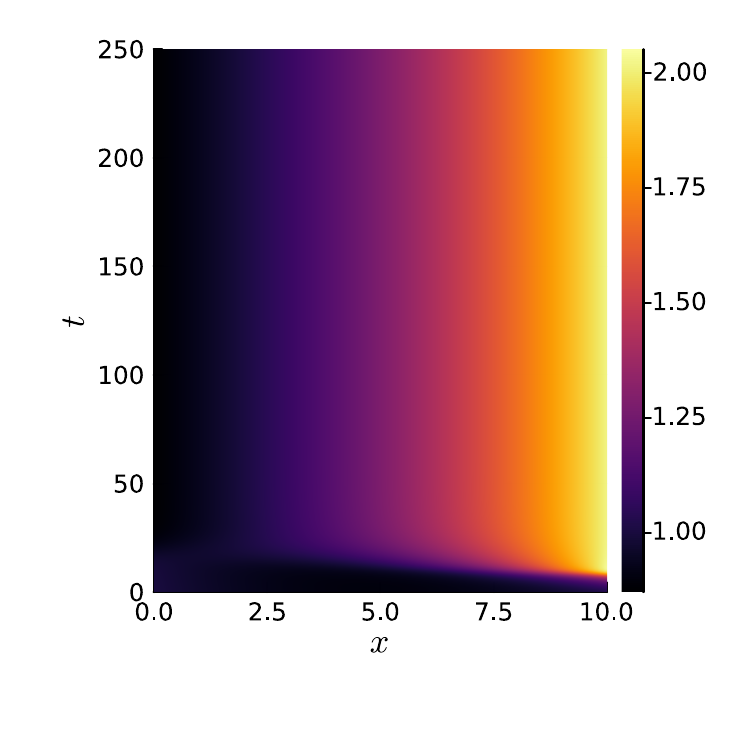}
    \end{subfigure}
    \begin{subfigure}{0.32\textwidth}
    \includegraphics[width=1\textwidth]{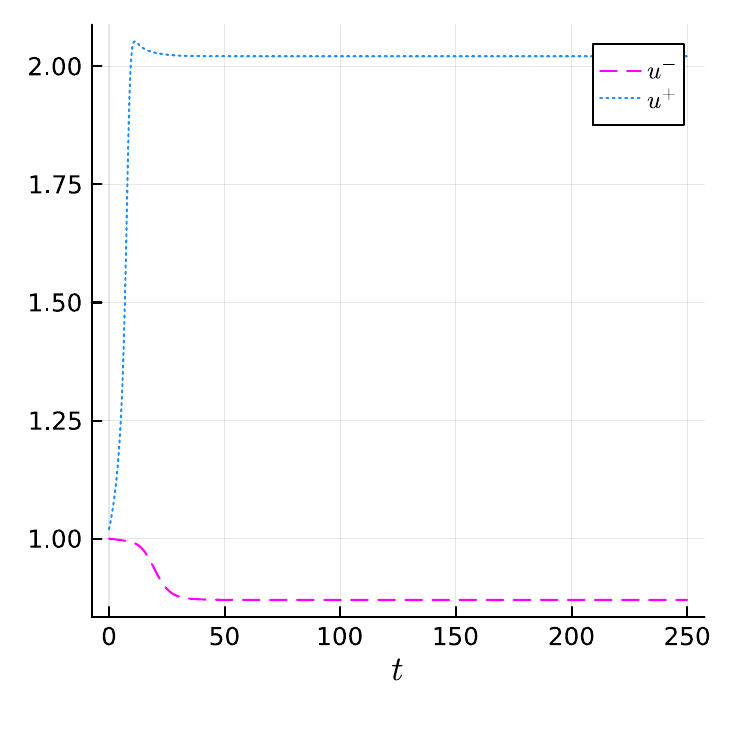}
    \end{subfigure}
    \begin{subfigure}{0.32\textwidth}
    \includegraphics[width=1\textwidth]{figures/timesol_beta1/spec_alph0.03_bet1.0_gam-1.0_L10.0_mu0.22_sig0.1_us0.0_pnt2.pdf}
    \end{subfigure}
    \caption{The numerical spatiotemporal solution converges to the asymmetric (mixed symmetric and antisymmetric) steady-state branch that emerges from the symmetric base steady-state with $u_*^s = 1$. 
    The parameters are $(\alpha, \beta, \gamma, L, \mu, \sigma) = (0.03, 1, -1, 10, 0.22, 0.1)$, corresponding to the red scattered point with black circle in first row left plot of Figure \ref{fig:Relambda_pos_i} and to the second row of Figure \ref{fig:timesol_beta1_part1} with $\mu=0.22 > 0.13864 \approx \mu_*^{as}$. Similar to Figures \ref{fig:timesol_beta1_part1} and \ref{fig:timesol_beta1_part2}, the initial condition is the symmetric base steady-state now with $u_*^s=1$ perturbed by the linear combination of a symmetric and antisymmetric perturbation with boundary values $\delta\cdot (1, 1) + \delta\cdot(1, -1)$ with $\delta=0.01$ and linear interpolation in between the boundary points.}
\end{figure}

\section{Discussion} \label{sec:discussion}

\paragraph{Summary \& further steps.}
In this study two dynamic (or Wentzell) boundaries coupled with each other through linear diffusion with degradation in the bulk has been analyzed with respect to bifurcations from a symmetric steady-state to both symmetric and asymmetric steady-states and to both in-phase and anti-phase Hopf oscillations. The existence and uniqueness (up to phase shifts) of both oscillations has been proven using the three developed foundational Assumptions 1--3 for the general nonlinear system \eqref{eq:sys}. For general cubically expanded reaction kinetics that are equal on both dynamic boundaries, leading-order expansions of the shape of the bifurcation branches, for steady-state and Hopf oscillations, close to their bifurcation point from the symmetric base steady-state along with the stability change of the base branch at the respective bifurcation points have been computed. The stability properties of the symmetric and asymmetric steady-state bifurcation branches have been derived rigorously and the stability properties of the in-phase and anti-phase Hopf branches rest on the key conjecture that they have opposite stability than the base branch at the respective value of the bifurcation parameter $\mu$ close to bifurcation onset. It is outstanding and interesting to prove (or disprove) this usually correct conjecture in an elegant and minimal way, potentially without direct expansions of the boundary-integral equations in the oscillation amplitudes. All numerical illustrations (parameter bubbles to stable steady-state and Hopf oscillations with magnitude of growth rates, critical curves of transition from sub- to supercritical Hopf branches, shape of nonlinear Hopf branches until far from onset, and full spatiotemporal numerical solutions showing competition or rather mixing of in-phase and anti-phase Hopf oscillations causing the occurrence of temporal defects) confirm both the proved theory and the stability conjecture, again for the cubically expanded reaction kinetics on the boundaries. Furthermore, the competition between stable steady-states and stable Hopf oscillations has been addressed (only in the linearized sense) through the proof that the symmetric steady-state gives birth to growing in-phase oscillations and the asymmetric steady-state gives birth to growing anti-phase oscillations, with the impossibility of coexistence of the stable steady-state and stable Hopf oscillation of the respective symmetry with respect to the linearized stability analysis. However, the nonlinearities usually change this behavior due to bending of the steady-state and Hopf branches, which provides an important direction for further study on the competition between the steady and oscillatory states. The in-phase and anti-phase Hopf bifurcation lines in parameter space have been shown numerically to intertwine in an oscillatory manner, coined \emph{Hopf dance} in previous literature.

\paragraph{Anti-flux and systems of equations.}
Both the shape of the illustrating Hopf branches in Figures \ref{fig:in-phase_Hopfbranches} (in-phase) and \ref{fig:anti-phase_Hopfbranches} (anti-phase) and the bifurcation bubbles, i.e., parameter regions of stable bifurcating steady-states and Hopf oscillations, to the linearized system of Figure \ref{fig:Relambda_pos_i} illustrate and the proof of Theorem \ref{thm:num_stable_sol} with Remark \ref{rmk:num_stable_sol} proves that there cannot be destabilizing bifurcations from the symmetric steady-state for the system with cubically expanded boundary reaction kinetics for $\mu\leq 0$. However, a usual diffusive flux balance with bulk diffusivity $d=1$ gives $\mu = -d = -1 < 0$ for bulk-surface systems with usual boundary conditions (cf. \cite{gomez2007self, levine2005membrane, gou2015synchronized, gomez2021pattern} and many more). Hence, precise mechanisms that lead to an effective dynamic boundary need to be found. In a different way, one could obtain the existence of destabilizing bifurcations to steady-states and to Hopf oscillations for $\mu \geq 0$ (flux instead of anti-flux) when including more than one abstract species that either diffuse on different time scales through different diffusivities in the bulk or that are differently attracted to the boundary by the dynamic boundary condition along the lines of \cite{rauch2004role, pelz2023emergence} (e.g., in the two species case with diffusing $u$ and $v$, two bifurcation parameters $\mu\geq 0$ multiplying $\partial_n u$ and $\nu\geq 0$ multiplying $\partial_n v$). 

\paragraph{Higher spatial dimensions.}
As mentioned in the preliminary study \cite{pelz2026oscillations}, going to higher spatial dimension would provide a fruitful path to connecting to real biological systems and their underlying dynamical laws. In 2-D, two disk-shaped compartments modeling two immobile cells without protrusions each with the same dynamic boundaries/membranes could be analyzed with respect to the formation of diffusively interacting ``spirals'' (oscillations along each disk's boundaries) raying away from each disk and synchronizing or pulsating in an in-phase or anti-phase manner. Analytically, one would deal with modified Bessel functions capturing the radial diffusive symmetry close to a single disk, which appear in the Dirichlet-to-Neumann and further differential operators that quantify the coupling between both disks.

\paragraph{Diffusive interaction of dynamic boundaries.}
Compared to a single dynamic boundary at $x=0$ coupled to a diffusive field on the half-line $(0, \infty)$ of the previous work \cite{pelz2026oscillations}, both the Hopf theorem \ref{theorem:hopf} together with the asymptotic theory and the in-phase and anti-phase bifurcation branches of Figures \ref{fig:in-phase_Hopfbranches} and \ref{fig:anti-phase_Hopfbranches} are highly similar. The higher tendency of turning asymptotically stable through a fold bifurcation observed in Figures \ref{fig:in-phase_Hopfbranches} and \ref{fig:anti-phase_Hopfbranches} just stems from the fact that the parameter value $\alpha = 0.49$ was chosen for this study while $\alpha = 1$ was taken in \cite{pelz2026oscillations}, due to the interesting numerical spatiotemporal solution at $(\alpha, \beta, \gamma) = (0.49, 1, -1)$ pictured in Figure \ref{fig:timesol_beta1_part2}. A true distinction between the single and coupled dynamic boundary case is that the essential spectrum of the single-boundary Fredholm operator is touching the imaginary axis for $\sigma_*=0$, which is not the case for the bounded domain $[0, L]$ with the two coupled bounding dynamical systems at $x=0$ and $x=L$ of this study. A bifurcation to oscillations seems to be obtained with similar ease in the coupled boundary case because the in-phase and anti-phase Hopf bifurcation points are close to the single-boundary Hopf bifurcation point $\mu_*=\sqrt{2\alpha}$ --- the in-phase bifurcation point being located at a slightly higher value for $\mu$ than $\mu_*$ and the anti-phase bifurcation point being slightly lower than $\mu_*$, with $\mu_*$ being symmetrically in the middle. This can be seen by noticing that
\begin{eqnarray*}
    c^i(\omega) &=& \sqrt{\sigma^2_* + i\omega}\;\tanh(\sqrt{\sigma^2_* + i\omega}\;\frac{L}{2}) \sim \sqrt{\sigma^2_* + i\omega} \;(1-2e^{-L\sqrt{\sigma^2_* + i\omega}}) + \cO(e^{-2L\sqrt{\sigma^2_* + i\omega}})\to \sqrt{\sigma^2_* + i\omega}, \\c^a(\omega) &=& \sqrt{\sigma^2_* + i\omega}\;\coth(\sqrt{\sigma^2_* + i\omega}\;\frac{L}{2}) \sim \sqrt{\sigma^2_* + i\omega} \;(1+2e^{-L\sqrt{\sigma^2_* + i\omega}}) + \cO(e^{-2L\sqrt{\sigma^2_* + i\omega}}) \to \sqrt{\sigma^2_* + i\omega}
\end{eqnarray*}
as $L\to \infty$ with the limits approached from opposite directions in $\bbC$, so that the in-phase and anti-phase Hopf frequencies and bifurcation points are converging to the single-boundary Hopf frequency $\omega_*=\sqrt{\alpha(\alpha-2\sigma_*^2)}$ and bifurcation point $\mu_* = \sqrt{2\alpha}$ from opposite directions in $\bbR$. Using the above expansions of $c^i$ and $c^a$, the difference of the in-phase Hopf and anti-phase Hopf frequencies behaves like
\begin{equation*}
    \omega_*^i - \omega_*^a = \alpha \frac{-\Im(e^{-Lz})|z|^2}{\frac{1}{4}\Re(z)^2 - \Re(z)^2\Re(e^{-Lz})^2+\Im(z)^2\Im(e^{-Lz})^2} 
\end{equation*}
with $z=\sqrt{\sigma_*^2+i\omega_*}$, so that $\omega_*^i - \omega_*^a \sim \cO(|e^{-Lz}|)$, and the difference of the bifurcation points behaves like
\begin{equation*}
    \mu_*^i - \mu_*^a = \alpha \frac{\Re(z)\Re(e^{-Lz}) - \Im(z)\Im(e^{-Lz})}{\frac{1}{4}\Re(z)^2 - \Re(z)^2\Re(e^{-Lz})^2+\Im(z)^2\Im(e^{-Lz})^2},
\end{equation*}
so that $\mu_*^i - \mu_*^a \sim \cO(|e^{-Lz}|)$. Continuing the Hopf bifurcation branches of the single-boundary system of \cite{pelz2026oscillations} at all scattered points of Figure \ref{fig:critcurve} with $\alpha = 0.49$, $L=10$, and $\sigma = 0.1$, does not consistently show that the in-phase and anti-phase branches are sandwiching the sole bifurcation branch for the single-boundary case further away from bifurcation onset. It would be a curiosity to study how the double-boundary in-phase and anti-phase Hopf branches are ``dancing'' about the single-boundary Hopf branch when walking along the Hopf dance described in Section \ref{sec:numericalasympt}.

\paragraph{Quantification of diffusive delay.}
In an attempt to explicitly quantify the diffusive delay in information transmission from one boundary to another through diffusion and the gradient term within the dynamic boundary equation, one can proceed similarly as for Proposition 1 of \cite{pelz2023emergence}. Close to the boundaries, $u$ can be expanded as $u_b = B_b(t)(x-b) + u(t, b)$ for $b\in\{0, L\}$ and functions $B_b$ to be determined via matching to the global solution satisfying
\begin{eqnarray*}
    \partial_t u = \partial_{xx} u - \sigma^2u, \quad x\in(0, L) \\
    u \sim B_b(t) (x-b) + u(t, b) \quad as \quad x\to b,
\end{eqnarray*}
where I have chosen the initial condition $u(0, x)\equiv 0$ for simplicity. The auxiliary free-space Green function problem 
\begin{eqnarray*}
    \partial_t v_b = \partial_{xx} v_b - \sigma^2 v_b - 2B_b(t)\delta(x-b), \quad x\in\bbR\setminus\{b\}, \\
    v_b\to 0 \quad \text{as} \quad |x|\to \infty \qquad \qquad \qquad
\end{eqnarray*}
with $v_b(0, x)\equiv 0$ can match with its bounded solution to the local expansions at the boundaries since $v_b \sim B_b(t)(x-b) + R_b(t) \quad \text{as} \quad x \to b$ with a function $R_b$ to be determined for $b\in\{0, L\}$. The bounded solution can easily be computed to be
\begin{equation*}
    v_b(t, x) =-\int_0^t B_b(t-\tau) \frac{\exp(-\sigma^2\tau-\frac{|x-b|^2}{4\tau})}{\sqrt{\pi\tau}}\; d\tau = B_b(t)(x-b) - \int_0^t B_b(t-\tau) \frac{e^{-\sigma^2\tau}}{\sqrt{\pi\tau}} + \cO(|x-b|^2),
\end{equation*}
which is expanded with the help of the Laplace transform and convolution theorem. It can be read off that $R_b = - \int_0^t B_b(t-\tau) \frac{e^{-\sigma^2\tau}}{\sqrt{\pi\tau}}\;d\tau$. Then the global solution is given by $u = v_0+v_L$, and matching to the local expansions $u_b$ leads to $R_0+v_L=u^-$ and $R_L+v_0=u^+$. In this way, the boundary dynamics are given by  
\begin{eqnarray*}
    \frac{du^-}{dt} &=& \fint(u^-) - \mu B_0(t) \quad \text{where} \quad -\int_0^t \frac{e^{-\sigma^2\tau}}{\sqrt{\pi\tau}}(B_0(t-\tau) + B_L(t-\tau) e^{-\frac{(x-L)^2}{4\tau}})\;d\tau = u^-, \\
    \frac{du^+}{dt} &=& \fint(u^+) + \mu B_L(t) \quad \text{where} \quad -\int_0^t \frac{e^{-\sigma^2\tau}}{\sqrt{\pi\tau}}(B_L(t-\tau) + B_0(t-\tau) e^{-\frac{x^2}{4\tau}})\;d\tau = u^+.
\end{eqnarray*}
The delay is thus precisely incorporated in the coupling functions $B_0$ and $B_L$ that are determined by the integral equations --- more details along these lines to follow at a later time.

\paragraph{Acknowledgments.} 
I would like to thank Arnd Scheel for helpful discussions. I am further grateful for Joel Dahne's tip to look into the Mittag-Leffler expansion of $\tanh$. I am acknowledging help from the artificial intelligence tool Claude Fable 5 by Anthropic for the proof of $\partial|_{\mu=\mu_*^{i|a}}\lambda^{i|a}>0$ in Theorem \ref{thm:leading-orderHopf} and for bounding the winding numbers $\Delta_\Gamma \mathrm{arg}(\dm^{i|a})$ of Theorem \ref{thm:num_stable_sol}. Artificial intelligence has not been used for the remaining parts of this study.

\paragraph{Competing Interests.} The author declares that there are no competing interests.

\bibliographystyle{abbrv}
\bibliography{refs}

\end{document}